\documentclass[11pt,a4paper]{article}
\usepackage{amsfonts}
\usepackage{bbm}
\usepackage[authoryear, sort]{natbib}
\usepackage{amsmath,amssymb,array,graphicx,mathrsfs}
\usepackage{enumerate}
\usepackage[english]{babel} 
\usepackage[shortlabels]{enumitem}
\usepackage{graphicx}
\usepackage[dvipsnames]{xcolor}
\usepackage{dsfont}
\usepackage{float}
\usepackage{epsfig}
\usepackage{epstopdf}
\usepackage[normalem]{ulem}
\usepackage{multirow}
\usepackage{leftidx}
\usepackage{setspace}
\usepackage{bbm} 
\usepackage{subcaption}
\usepackage{tikz}
\usetikzlibrary{arrows.meta}     
\usepackage{amsthm}
\usepackage{color}
\usepackage{mathtools}
\usepackage{multirow,cases}
\usepackage{booktabs}
\usepackage[colorlinks,linkcolor=black,anchorcolor=black,citecolor=black]{hyperref}
\usetikzlibrary{shapes,arrows,positioning,fit,calc}
\usepackage[titletoc,toc,title]{appendix}
\usepackage{verbatim}
\usepackage{tensor}
\usepackage{hyperref}
\usepackage{cleveref}
\usepackage{empheq} 
\usepackage{algorithm}
\usepackage{algorithmic}
\usepackage{tikz}
\usetikzlibrary{shapes,backgrounds}
\theoremstyle{plain}

\newtheorem{definition}{Definition}
\newtheorem{theorem}{Theorem}

\newtheorem{lemma}{Lemma}

\newtheorem{assumption}{Assumption}

\newtheorem{proposition}{Proposition}
\newtheorem{condition}{Condition}

\newcommand{\ordinalsuffix}[1]{%
  \ifcase#1th\or st\or nd\or rd\else th\fi%
}

\AtBeginDocument{%
  \renewcommand{\today}{%
    \number\day\textsuperscript{\ordinalsuffix{\day}}~%
    \ifcase\month\or
      January\or February\or March\or April\or May\or June\or
      July\or August\or September\or October\or November\or December\fi
    , \number\year%
  }%
}

\usepackage[a4paper,top=2.7cm,bottom=2.8cm,left=2.2cm,right=2.2cm,marginparwidth=1.75cm]{geometry}

\title{
Learning under  Opponent Unawareness in Linear-Quadratic Stochastic Games
} 
\author{}

\date{\today}

\author{
	{Dantong Chu\footnote{Department of Systems Engineering and Engineering Management, The Chinese University of Hong Kong, Hong Kong, China.
    Email:dtchu@link.cuhk.edu.hk}}
	\and
	{Xuefeng Gao\footnote{Department of Systems Engineering and Engineering Management, The Chinese University of Hong Kong, Hong Kong, China.
    Email:xfgao@se.cuhk.edu.hk}}
        \and
	{Yufei Zhang\footnote{Department of Mathematics, Imperial College London, UK.
    Email: yufei.zhang@imperial.ac.uk}}
}

\begin{document}

\maketitle
\begin{abstract}
As firms increasingly deploy machine learning for strategic decision-making, 
understanding algorithmic interactions has become central to operations research and economics.
This paper studies learning in infinite-horizon, nonzero-sum linear-quadratic stochastic games under a \emph{radically uncoupled} information structure, where players are either unaware of opponents or strategically oblivious, observing only a common state and their own action history. Under this minimal information,
we analyze an asynchronous decentralized learning process in which each player independently runs a single-agent  $\epsilon$-greedy iterated least-squares algorithm. 
We prove that, despite being unable to identify the system parameters, players' learning dynamics converge almost surely to the complete-information Nash equilibrium and characterize the convergence rate. We then apply the framework to a dynamic Cournot competition with sticky prices. Numerical experiments validate the theoretical results and show that learning under limited information reduces firm profits under both low and high price stickiness, while total surplus declines and market concentration increases when price stickiness is high. Publicly revealing aggregate market output substantially accelerates convergence and mitigates these welfare losses.
\end{abstract}
	

\section{Introduction}

Driven by advances in machine learning and artificial intelligence, autonomous pricing and automated trading algorithms have become ubiquitous across e-commerce and financial markets. 
In practice, these algorithms plausibly operate under a radically uncoupled information structure, termed as in  \cite{foster2006regret}. Under this paradigm, agents/algorithms either do not recognize the presence of their peers, lacking access to their actions, or remain strategically oblivious by treating competitors’ dynamic behavior as exogenous environmental noise.  Such obliviousness arises both by necessity, when limited information (e.g., in over-the-counter markets) or substantial market noise obscures competitors' signals, and by design, as it enables the use of computationally simpler single-agent   learning algorithms. While this framework mirrors the informational realities of modern tech and financial markets, it could pose severe economic risks. 
Recent literature on algorithmic collusion demonstrates that such uncoupled agents, despite independently employing reinforcement learning algorithms, may inadvertently converge to seemingly collusive outcomes rather than to a competitive Nash equilibrium (see e.g. \cite{abada2023artificial}).

This gap reflects a fundamental limitation in our theoretical understanding. While radically uncoupled learning is known to converge in simple stateless repeated games \citep{bervoets2020learning, xiong2024thompson, ba2025doubly}, convergence guarantees remain largely absent for stochastic (Markov) games, where the state evolves according to the current state and the joint actions of all players \citep{shapley1953stochastic, ozdaglar2021independent}.  
These state dynamics substantially complicate learning under radically uncoupled information, as   players will misjudge their impact on the state in the absence of observations of their opponents  \citep{zhang2021multi}.  Existing convergence guarantees are therefore largely limited to specialized settings, such as (two-player) zero-sum games \citep{sayin2021decentralized, chen2024last} and potential games \citep{maheshwari2025independent}, and   assume finite state and action spaces  to make the analysis tractable. Little is known about convergence for general $N$-player  stochastic games with continuous state and action spaces.

In this paper, we study radically uncoupled learning in infinite-horizon, nonzero-sum linear-quadratic (LQ) stochastic games. LQ games are a fundamental class of models for dynamic strategic interactions in continuous state and action spaces and have been widely used in management science, economics, and financial engineering; see Section \ref{sec:moti} for an application to dynamic Cournot competition and \citealt{fershtman1987dynamic, dockner2000differential, heston2024dynamic, muhle2026pre} for related examples. 
By combining linear state dynamics with quadratic costs, LQ games admit closed-form equilibrium solutions and provide rigorous benchmarks for evaluating learning algorithms in dynamic   games.  Our central question is to understand the behavior of radically uncoupled  learning agents: When agents learn from their own experiences while remaining unaware of or strategically oblivious to their competitors, can decentralized learning dynamics lead to a Nash equilibrium?

\subsection{A Motivating Application: Dynamic Cournot Competition}\label{sec:moti}

To motivate our learning framework, we consider a dynamic Cournot competition model with sticky prices as an illustrative application. This classical model captures multi-period quantity competition under market frictions and has been extensively studied   \citep{fershtman1987dynamic,heston2024dynamic}. 
Consider a discrete-time, infinite-horizon market comprising $M$ firms. In each period $t \in \mathbb{N}$, each firm $m$ chooses its output quantity $q^m_t$. The market price $p_t$ evolves according to the following discrete-time process:
\begin{equation}\label{dynamic.oligopoly}
p_{t+1} = s p_t + (1 - s) \left( a - b \sum\limits_{m=1}^M q^m_t \right) + \omega_t,  
\end{equation}
where $s \in (0,1)$ is the price stickiness coefficient, $a > 0$ is the demand intercept, $b > 0$ is the slope of the inverse demand function, and the noise term $\{\omega_t\}_{t=0}^\infty$ is an i.i.d.~random shock (e.g., consumer mood swings, market noise) with zero mean and variance $\sigma_\omega^2$. 
In this model, price adjusts gradually toward the level implied by the current demand function rather than instantaneously, with the speed of adjustment governed by a stickiness parameter $s \in (0,1)$ (e.g., due to market inertia). 
The goal of each firm is to select a sequence of output quantities $\{q^m_t\}_{t=0}^\infty$ to maximize its total discounted expected profit:
\begin{equation} \label{oligopoly.profitfunction}
\mathbb{E}\left[\sum_{t=0}^{\infty} \rho^{t} \left( p_{t+1} - c_m - \frac12 q^m_t \right) q^m_t\right],
\end{equation}
where \(\rho \in (0,1)\) is the discount factor, and the term \((c_m + \frac12 q^m_t) q^m_t\) represents the production cost, consisting of a linear term \(c_m q^m_t\) with private constant marginal cost parameter \(c_m>0\) and a quadratic term \(\frac12 (q^m_t)^2\) reflecting increasing marginal production costs \citep{fershtman1987dynamic}. 
The feature of the sticky price places the model within the class of stochastic games rather than repeated games, as the price acts as a state variable.

In new or emerging markets, firms may have limited information about their competitors or even be unaware of their existence.   In particular, 
at time $t+1$,
each firm $m$ observes only the historical price sequence $\{p_k\}_{k=0}^{t+1}$ and its own quantity decisions 
$\{q^m_k\}_{k=0}^{t}$. 
Moreover, the key structural parameters of \eqref{dynamic.oligopoly}, including $s$, $a$, $b$, and $\sigma_\omega$, are also unknown to all firms. 
Under this information constraint, each firm behaves as if it were the sole major participant in the market and treats the effects of other firms' actions as unobserved noise. Consequently,  firm $m$ forms a misspecified belief about price dynamics: 
\begin{equation}\label{oligopoly.statisticalmodel}
p_{t+1} = s' p_t + (1 - s') \left( a' - b' q^m_t \right) + \omega'_t,  
\end{equation}
where $s'$ denotes the firm's perceived price stickiness coefficient, $a'$ and $b'$ are the perceived intercept and slope of the inverse demand function, and $\omega'_t$ represents the perceived market noise with mean zero and variance $\sigma_m^2$. 
The firm's objective remains to choose an output sequence  $\{q^m_t\}_{t=0}^\infty$ to maximize its total discounted expected profit \eqref{oligopoly.profitfunction}, but 
subject to \emph{its perceived and misspecified  price dynamics
\eqref{oligopoly.statisticalmodel}}.

Since the parameters of the perceived dynamics \eqref{oligopoly.statisticalmodel} are unknown, firm $m$  adopts an online learning framework that balances exploration and exploitation.
Specifically, 
each firm independently employs an \(\epsilon \)-greedy learning algorithm based on least-squares estimation of the perceived model parameters (e.g., $s', a', b'$, $\sigma_m^2$) over a sequence of expanding epochs. 
Within each epoch, firm $m$ uses a linear feedback strategy of the form  \(q_{t}^{m}=-F_{m}p_{t}+f_{m}+\alpha_{m}v_{t}^{m}\), where   \(F_{m}\) and \(f_{m}\) are derived from the estimated model parameters, and \(v_{t}^{m}\) is an independent exploration noise scaled by the exploration coefficient \(\alpha_{m}\). Linear inverse demand, quadratic costs, and linear strategies are standard assumptions in the Cournot competition literature; see, e.g., \citealt{vives1984duopoly,gal1985information,hauk2001secret,bonatti2017dynamic}.

The learning process described above provides a natural model for how firms adapt in emerging markets with limited information. In contrast, dynamic Cournot games with complete information are typically analyzed through feedback (closed-loop) Nash equilibria, which characterize firm behavior in mature markets \citep{fershtman1987dynamic,heston2024dynamic}. This contrast raises two fundamental questions about market evolution. First, can the feedback Nash equilibrium of a mature market emerge from radically uncoupled learning?
That is, when firms learn solely from their own experiences under model misspecification and without knowledge of competitors, do their behaviors converge to the complete-information equilibrium? Second, how does the learning process shape market outcomes, such as economic welfare and market concentration, during this transition? This paper addresses these questions by establishing convergence guarantees   and characterizing the resulting market dynamics under different degrees of price stickiness.


\subsection{Our Contributions}

 Motivated by the dynamic Cournot competition model with sticky prices in Section~\ref{sec:moti}, we study online learning under radically uncoupled information in a broad class of infinite-horizon, $N$-player nonzero-sum LQ stochastic games. Each player observes only the common state and chooses her own actions, while the system dynamics and all information regarding other players remain unknown. 
 Consequently, each player treats the environment as a single-agent problem and applies a canonical $\epsilon$-greedy online learning algorithm following \citep{mania2019certainty, simchowitz2020naive}. We formalize this behavior through a radically uncoupled multi-agent $\epsilon$-greedy iterated least-squares (ILS) algorithm with growing epochs. At each epoch, players estimate a perceived (and generally misspecified) model, compute linear policies from the estimated dynamics, and inject decaying exploration noise. The algorithm requires no information about other players and allows for heterogeneous update schedules and exploration rates.

   Our contributions are twofold:  
\begin{itemize}
    \item 
    We establish the first convergence guarantee  for radically uncoupled learning in LQ stochastic games   with a common state affected by all players' actions.

The main challenge is that, without information about their opponents, players cannot disentangle their own impact on the state dynamics from that of other players, and will   misattribute the effects of other players’ actions to system noise. This prevents them from   identifying the true system dynamics, unlike in the single-agent setting \citep{simchowitz2020naive}.
We overcome this challenge through a key insight: although each player's perceived model is misspecified, the learned parameters implicitly capture the \emph{aggregate strategic effects} of all other players. Consequently, each player effectively computes a \emph{noisy best response} to the opponents' previous strategy profile. We further quantify the bias in this estimated best response induced by all players'  exploration noises and show that it vanishes asymptotically.

Building on this insight, we prove that the multi-agent $\epsilon$-greedy ILS algorithm converges to the   complete-information Nash equilibrium under a stability condition and characterize its convergence rate, 
both in finite samples (with high probability) and asymptotically (almost surely) (Theorem~\ref{thm.convergencyne}).
   The result applies to  asynchronous and random update schedules.
   The convergence rate reveals a fundamental interplay between  the intrinsic stability of the Nash equilibrium, quantified by a stability constant $\zeta\in(0,1)$, and the players' learning hyperparameters including the epoch growth rate $\lambda$ and the exploration-noise schedule. In particular, when all players adopt the optimal exploration-noise decay schedule from the single-agent setting, the convergence rate exhibits a bifurcation. 
  If the equilibrium is sufficiently stable, the algorithm (more precisely, the linear strategy parameters) achieves a last-iterate rate $\tilde{O}(t^{-1/4})$, matching the rate that enables optimal regret guarantees in the single-agent LQ learning benchmark \citep{simchowitz2020naive}. 
Otherwise, the performance 
 is limited by the game's stability and the convergence rate is $\tilde{O}\left(t^{\frac{\ln \zeta}{\ln \lambda}}\right)$, 
 reflecting the convergence behavior 
  of the underlying idealized best-response dynamics.  
Finally, numerical experiments show that when the equilibrium is not globally stable, the learning dynamics may exhibit persistent oscillations.

    \item We apply our framework to a dynamic Cournot competition model with sticky prices and study the economic consequences of decentralized learning.

    Numerical experiments in duopoly markets validate the theoretical convergence rates and characterize transitional market outcomes under different degrees of price stickiness. We examine firm profits, total surplus, and market concentration along the path from initial learning to equilibrium. While firms experience profit losses relative to the equilibrium benchmark during the transition in both low- and high-stickiness markets, welfare implications differ substantially across regimes: total surplus declines persistently only under high price stickiness, and market concentration increases in both cases, with larger effects under high stickiness. These findings highlight the importance of accelerating convergence in markets with strong frictions. We further show that publicly releasing aggregate market quantity substantially improves learning speed and mitigates the adverse transitional effects. Finally, by considering a (weakly) nonlinear inverse demand function, we demonstrate that the multi-agent $\epsilon$-greedy ILS algorithm remains effective beyond the LQ setting, converging to an accurate linear approximation of the equilibrium.

\end{itemize}

\subsection{Related Work}\label{sec.relatedwork}

Our work relates to several broad research areas, including learning in games,  
multi-agent reinforcement learning, 
and Cournot competitions. 
Given the vast literature in each area, we focus our discussion on the studies most closely related to our work.

 \paragraph{Nash equilibrium computation for LQ games.}
A growing body of work has developed learning algorithms   for computing Nash equilibria in $N$-player LQ games. When the system parameters are known, existing approaches include policy-gradient methods \citep{mazumdar2020policy, hambly2023policy,plank2026learning} and iterative Lyapunov- or Riccati-based schemes \citep{nortmann2024nash}. \citep{vamvoudakis2015non, hambly2023policy,nortmann2024nash} further analyze settings with unknown system parameters, but require access to joint state-action data from all players or coordinated exploration schemes that specify the behavior of all agents.

Our work differs from this literature in two fundamental aspects. Conceptually, rather than designing algorithms to compute Nash equilibria, we study whether the learning dynamics of strategic agents converge to a complete-information Nash equilibrium under radically uncoupled information, where players are unaware of and cannot observe their opponents. Technically, this information constraint prevents players from identifying the true game model. Each player instead learns a misspecified single-agent model from noisy observations and applies a standard single-agent learning algorithm, treating the strategic behavior of others as part of the environment. Our analysis therefore requires characterizing how the learned policy approximates a best response, quantifying the bias induced by exploration, and establishing convergence through a joint analysis of estimation error, exploration, and equilibrium stability.

  \paragraph{Cournot games with incomplete information.}
Our analysis of dynamic Cournot competition is also related to the literature on Cournot games with incomplete information in economics and operations research. A substantial body of literature studies Cournot games where firms possess private information and face uncertain  cost and  demand   \citep{gal1985information,hauk2001secret, bonatti2017dynamic}. The primary focus of this literature is on how firms learn hidden market information and strategically signal to maximize profits.  
For example, \cite{bonatti2017dynamic} considers a new market in which each firm has private marginal cost and observes only a common revenue process, but not its competitors' quantities or costs. Using filtering techniques, they show that firms converge to the static complete-information Nash equilibrium despite learning only the average cost of their competitors.

A key distinction from our work is that these studies assume the underlying market model, such as  the inverse demand function, is known to all firms, so learning is limited to opponents' private information such as their  costs or demand shocks. In contrast, we allow both the market model and even the presence of competitors to be unknown. Consequently, firms must simultaneously learn a misspecified model of the environment and adapt their strategies, rendering filtering-based approaches inapplicable.

\paragraph{Learning in Cournot games.}

Several studies examine how firms learn and compete in repeated Cournot games with unknown demand or cost parameters using trial-and-error methods \citep{huck2004through}, neural networks \citep{barr2005cournot}, policy-gradient methods \citep{shi2020multi}, and no-regret learning \citep{nadav2010no,ba2025doubly}. In particular, reinforcement learning agents in repeated Cournot games may learn collusive behaviors \citep{kimbrough2005learning,waltman2008q}, 
a   phenomenon further studied in recent work   \citep{lin2024strategic,deshpande2026strategic}.

These studies, however, focus on repeated games with static market environments and therefore do not directly apply to the price-stickiness setting considered in this paper. With price stickiness, firms' current decisions affect future market states, transforming the problem into a stochastic game rather than a repeated game. This dynamic state dependence substantially complicates the convergence analysis of learning dynamics compared with static games.

\paragraph{Notation.}

We denote by \(\Vert A \Vert\) the spectral norm of a matrix \(A\), by \(\sigma_{\min}(A)\) its minimum singular value, by \(r(A)\) its spectral radius, and by \(A^\top\) its transpose. For vectors, \(\|\cdot\|_2\) denotes the Euclidean norm. 
The notation \(\succcurlyeq\) (resp. \(\succ\)) refers to the Loewner order: 
\(X \succcurlyeq Y\) (resp. \(X \succ Y\)) means \(X - Y\) is positive semidefinite 
(resp. positive definite). We write \(I_{d_m}\) for the \(d_m\)-dimensional identity matrix and \(0_{m \times n}\) for the \(m \times n\) zero matrix. Finally, \(O(\cdot)\) denotes standard asymptotic notation, and \(\tilde{O}(\cdot)\) suppresses logarithmic factors.

The remainder of this paper is organized as follows. 
Section \ref{sec:uncoupled_framework} formulates the infinite-horizon, nonzero-sum LQ stochastic games and derives the multi-agent $\epsilon$-greedy ILS algorithm in the radically uncoupled setting. Section \ref{section:model_complete} defines the feedback Nash equilibrium under complete information. Section \ref{sec.convergencyanalysis} states the main assumptions and establishes the convergence guarantees of the algorithm. 
Section \ref{sec:dynamic_oligopoly} applies the results to analyze learning  in  dynamic Cournot competition  
with sticky prices.  
Section \ref{subsec:pf_sketch} provides a proof sketch of the main theorem in Section \ref{sec.convergencyanalysis}. Finally, Section \ref{sec.conclusion} concludes. All technical proofs and supplementary materials are collected in the 
appendix.


\section{Problem Formulation}\label{sec:uncoupled_framework}

This section introduces a general setup for LQ stochastic games in a radically uncoupled environment and outlines an online learning procedure for such games.

\subsection{Setup of LQ Stochastic Games}\label{sec:independent_background}

Consider an $M$-player LQ stochastic game. For each player $m$, let $u_t^m \in \mathbb{R}^{d_m}$ denote its control input at time $t$, and let $x_t \in \mathbb{R}^n$ denote the common state. The state evolves according to the linear dynamics
\begin{equation}\label{state}
x_{t+1} = A_0 + A x_t + \sum\limits_{m=1}^{M} B_m u_t^m + \omega_t, 
\quad x_0 = x,
\end{equation}
where  $\{\omega_t\}_{t=0}^\infty$ is an i.i.d.\ sub-Gaussian noise sequence with mean zero 
and 
covariance matrix \( \operatorname{Cov}(\omega_t) = \Sigma_{\omega}\succ0 \), $x \in \mathbb{R}^n$ is a given initial state, and the system parameters $(A_0, A, \{B_m\}_{1 \le m \le M}, \Sigma_{\omega})$ are unknown to all players. 
Each player is unaware of the existence of other players and observes only the historical state sequence $\{x_k\}_{k=0}^{t+1}$ and their own past controls $\{u_k^m\}_{k=0}^t$ up to time $t+1$.

Player $m$ aims to minimize the following discounted costs:
\begin{equation*}
 J_m \big( u  \big)
= \mathbb{E} \left[\sum_{t=0}^{\infty} \rho^{t} Y_m\big(x_t,x_{t+1},u_m\big)\right],
\end{equation*}
with the cost function is given by
\begin{align}\label{eqn.running}
Y_m(x_t, x_{t+1}, u_t^m) = x_t^\top Q_m x_t + (x_{t+1})^\top H_m u_t^m 
+ (u_t^m)^\top R_m u_t^m + x_t^\top q_m + (u_t^m)^\top r_m,
\end{align}
where the coefficient matrices $Q_m \succcurlyeq 0_{n\times n}$, $R_m \succ 0_{d_m\times d_m}$, $H_m$, and vectors $q_m$, $r_m$ are known to player $m$ as private information but unknown to other players.

Note that we include the cross term $(x_{t+1})^\top H_m u_t^m $   in the cost  \eqref{eqn.running}  to capture the dynamic Cournot competition model in Section \ref{sec:moti}. Due to these cross terms, the cost function may not be strongly convex in the state-control pair. This differs from the standard quadratic cost $x_t^\top Q_m x_t + (u_t^m)^\top R_m u_t^m$ in the literature, where $Q_m$ and $R_m$ are typically assumed to be positive definite to ensure strong convexity. The lack of strong convexity introduces additional technical challenges, including issues in the well-posedness of the game and the analysis of learning algorithms.

Since each player is unaware of the impact of other players' actions on the state dynamics, they form a misspecified model of the environment. Specifically, player $m$ believes that the state evolves according to the perceived model
\begin{equation}\label{dynamic:statistical}
x_{t+1} = \Theta_m^\top z_t^m + \omega_t', 
\quad x_0 = x,
\end{equation}
where $z_t^m = [1,\; x_t^\top,\; (u_t^m)^\top]^\top$, 
$\{\omega_t'\}_{t=0}^\infty$ is an i.i.d.\ sub-Gaussian sequence with mean zero and 
unknown covariance matrix  $\Sigma_m$, and $\Theta_m = [A'_0,\; A',\; B'_m]^\top \in \mathbb{R}^{(1 + n + d_m) \times n}$ 
is the unknown (drift) coefficient. This model is misspecified because it omits the control inputs of all other players. Player $m$ will estimate $(\Theta_m, \Sigma_m)$ from their own observations to compute their strategies.

We focus on linear feedback strategies of the form $u_t^m = -F_m x_t + f_m + \alpha_m v_t^m$, where $F_m$ is the feedback matrix governing the response to the current state $x_t$, $f_m$ is an offset term capturing state-independent actions, $\{v_t^m\}_{t=0}^\infty$ is an i.i.d.\ sub-Gaussian noise sequence with mean zero and covariance matrix $I_{d_m}$, and $\alpha_m \ge 0$ scales its magnitude. We denote a strategy for player $m$ by the triple $(F_m, f_m, \alpha_m)$, and refer to $(F_m, f_m)$ as the deterministic component (i.e., the strategy when $\alpha_m = 0$). Player $m$ chooses $\alpha_m$ to ensure suitable  exploration, and selects  $(F_m, f_m)$   to minimize the cost objective   given current information. 

To formulate each player's optimization problem and facilitate the theoretical analysis, we restrict strategy parameters to the following bounded \textit{admissible strategy sets}: 
\begin{align*}
\mathcal{A}_m = \{F_m \in \mathbb{R}^{d_m \times n} : \|F_m\| \le \kappa_m\}, \quad
\mathcal{B}_m = \{f_m \in \mathbb{R}^{d_m} : \|f_m\|_2 \le \kappa'_m\}.
\end{align*}
 The   constants $\kappa_m, \kappa'_m > 0$ bound the allowable feedback and offset, respectively, reflecting physical or economic constraints  on  player $m$’s ability to implement control. For the Cournot model of Section \ref{sec:moti},  $\kappa_m$ and $\kappa'_m$ directly represent production limits for price‑based and active adjustments. The admissible sets $\mathcal{A}_m$ and $\mathcal{B}_m$ constitute private information for player $m$ and are unknown to other players.

Given 
a perceived parameter pair \((\Theta_m, \Sigma_m)\), the initial point $x$, and the exploration noise scale $\alpha_m$, player $m$ faces the following optimization problem
\begin{align}\label{discount_cost_m}
&\inf\limits_{(F_m, f_m) \in \mathcal{A}_m \times \mathcal{B}_m}J^{(\Theta_m,\Sigma_m)}_m \big( (F_m,f_m) ; x,\alpha_m \big), \\\nonumber
&\text{with} \quad J^{(\Theta_m,\Sigma_m)}_m \big( (F_m,f_m) ; x,\alpha_m \big)
= \mathbb{E} \left[\sum_{t=0}^{\infty} \rho^{t} Y_m\big(x_t,x_{t+1},-F_m x_t+f_m+ \alpha_m v^m_t\big)\right],
\end{align}
where for each \((F_m, f_m) \in \mathcal{A}_m \times \mathcal{B}_m\), the state \(x_t\) evolves according to the perceived dynamics \eqref{dynamic:statistical} under the parameters \((\Theta_m,\Sigma_m)\) and the control law \(u^m_t = -F_m x_t + f_m+ \alpha_m v^m_t\).

To ensure  \eqref{discount_cost_m} is finite, player $m$ restricts the parameters to the following \textit{feasible parameter set} \(\mathcal{S}_m\):
\begin{equation}\label{def:feasible_set}
\mathcal{S}_m = \left\{(\Theta_m,\Sigma_m) : -\infty< \inf_{(F_m,f_m)\in \mathcal{A}_m \times \mathcal{B}_m} J^{(\Theta_m,\Sigma_m)}_m \big( (F_m,f_m) ; x,\alpha_m \big) < \infty, \text{ for all } x \in \mathbb{R}^n, \alpha_m\ge 0 \right\}.
\end{equation}  
For any \((\Theta_m,\Sigma_m) \in \mathcal{S}_m\) and $x \in \mathbb{R}^n, \alpha_m\ge 0$, there exists at least one optimal strategy \\ \((\Phi_m(\Theta_m,\Sigma_m,x,\alpha_m),\phi_m(\Theta_m,\Sigma_m,x,\alpha_m)) \in \mathcal{A}_m \times \mathcal{B}_m\) to problem \eqref{discount_cost_m}.

Within this setting, the next section introduces an online learning procedure in which every player independently learns   the environment and applies adaptive strategies to minimize her own discounted cost.

\subsection{Radically Uncoupled Multi-agent $\epsilon-$Greedy ILS Algorithm}\label{sec.olalgorithm}

We assume that each player independently follows an $\epsilon$-greedy ILS scheme 
  over a sequence of exponentially increasing epochs. In each epoch, a player implements a certainty-equivalent control strategy by applying the optimal linear feedback policy corresponding to their current estimate of the (potentially misspecified) system, together with an exploratory noise process to ensure sufficient exploration.
  The main components of the learning process   are summarized below.  

\paragraph{Update schedule and exploration.}   
    Let $\tau^{(k)}_m$ denote the length (in time steps) of the $k$-th epoch for player $m$, e.g. $\tau^{(k)}_m= \lambda^k$ for some $\lambda>1.$
    At the end of epoch $k$, player $m$ updates her strategy, which is then held fixed throughout the subsequent epoch of length   $\tau^{(k+1)}_m$.  
Player $m$  also specifies a decaying sequence of exploration noise scales
  $\{\alpha_m^{(k)}\}_{k=0}^\infty$. 
The decaying exploration schedule follows the standard single-agent online learning framework \citep{simchowitz2020naive}, reflecting the assumption that each player behaves as if she were the sole decision-maker. 

Players may employ heterogeneous update schedules and exploration noise scales, leading to   asynchronous learning. The precise conditions on  $\tau_m^{(k)}$ and $\alpha_m^{(k)}$ (for convergence analysis) are given in  Condition~\ref{condition:tauk_exp} in Section~\ref{sec.convergencyanalysis}.

\paragraph{Regularized least‑squares (RLS) estimation and fallback.}  
    At the initial update, each player independently selects a feasible initial parameter pair based on the information available at time zero. At the $k$-th update point, player $m$ 
    estimates the parameters of the perceived model \eqref{dynamic:statistical} by solving a regularized least-squares problem  using only data collected during the most recent epoch. Specifically, 
    player $m$ considers the estimate \begin{align}\label{def.RLS}
        \hat{\Theta}^{(k)}_m 
        \in \arg\min_{y\in \mathbb{R}^{(1+n+d_m)\times n}} \left\{ 
            \beta_m \|y\|^2 + \sum\limits_{t=\bar{\tau}^{(k-1)}_m}^{\bar{\tau}^{(k)}_m-1} \| x_{t+1} - y^\top z^m_t \|^2_2 
        \right\}, 
    \end{align}
where 
 $\bar{\tau}^{(k)}_m = \sum_{i=1}^{k} \tau^{(i)}_m$ is  
 the cumulative number of time steps 
 completed by player $m$
 up to the end of the  $k$-th update, and 
$\beta_m > 0$ is a regularization constant.
    The estimate \eqref{def.RLS} admits the closed‑form expression
    \begin{align}\label{def.hatTheta}
        \hat{\Theta}^{(k)}_m 
        = \left( V_m^{(k)} \right)^{-1} \left( \sum\limits_{t=\bar{\tau}^{(k-1)}_m}^{\bar{\tau}^{(k)}_m-1} z^m_t x_{t+1}^\top \right), \quad 
        \text{with} \quad 
        V_m^{(k)} 
        = \beta_m I_{(1+n+d_m)} + \sum\limits_{t=\bar{\tau}^{(k-1)}_m}^{\bar{\tau}^{(k)}_m-1} z^m_t (z^m_t)^{\top}.
    \end{align}
    Here we restrict the estimation to data from the most recent epoch. This  reflects a natural behavioral assumption: without observing the  opponents, players update their perceived model using recent observations that best represent the current   environment.    
    Similarly, the  estimate of the perceived noise covariance matrix is
\begin{equation}\label{def.hatSigma}
        \hat{\Sigma}^{(k)}_{m}=\frac{1}{\tau^{(k)}_m} \left( \sum\limits_{t=\bar{\tau}^{(k-1)}_m}^{\bar{\tau}^{(k)}_m-1} (x_{t+1}-(\hat{\Theta}^{(k)}_m)^{\top}z^m_t) (x_{t+1}-(\hat{\Theta}^{(k)}_m)^{\top}z^m_t)^\top \right).
\end{equation}

 Player
$m$ checks whether the estimated parameter pair $(\hat{\Theta}^{(k)}_m,\hat{\Sigma}^{(k)}_{m})$ belongs to the feasible set $\mathcal{S}_m$   in \eqref{def:feasible_set} to ensure the solvability of the corresponding control problem.  If feasible, the estimate is accepted; otherwise, the player reuses the previous epoch's parameter pair: 
\begin{equation}\label{eqn.fallback}
(\tilde{\Theta}^{(k)}_m,\tilde{\Sigma}^{(k)}_{m}) 
\leftarrow 
\begin{cases} 
(\hat{\Theta}^{(k)}_m,\hat{\Sigma}^{(k)}_{m}) & \text{if } (\hat{\Theta}^{(k)}_m,\hat{\Sigma}^{(k)}_{m}) \in \mathcal{S}_m, \\[4pt]
(\tilde{\Theta}^{(k-1)}_m,\tilde{\Sigma}^{(k-1)}_{m}) & \text{otherwise}.
\end{cases}
\end{equation}

\paragraph{$\epsilon$-Greedy strategy execution.}  
    Using the   parameter estimate $(\tilde{\Theta}^{(k)}_m,\tilde{\Sigma}^{(k)}_{m})$, together with the initial state $x_{\bar{\tau}^{(k)}_m}$ and the current exploration scale $\alpha^{(k)}_m$, player $m$ solves the LQ  control problem \eqref{discount_cost_m} to obtain a minimizer 
    \begin{align}\label{eq:Ff}
    (F^{(k)}_m, f^{(k)}_m) =(\Phi_m(\tilde{\Theta}^{(k)}_m,\tilde{\Sigma}^{(k)}_{m}, x_{\bar{\tau}^{(k)}_m},\alpha^{(k)}_m), \phi_m(\tilde{\Theta}^{(k)}_m,\tilde{\Sigma}^{(k)}_{m}, x_{\bar{\tau}^{(k)}_m},\alpha^{(k)}_m) ).
    \end{align}
For the next epoch, player $m$ implements the $\epsilon$-Greedy  strategy with an additive exploratory noise:
\begin{equation}\label{perturbed_strategyk}
    u_t^{m,(k+1)} = -F^{(k)}_m x_t + f^{(k)}_m + \alpha^{(k)}_m v_t^m, 
\end{equation}
where $\{v_t^m\}_{t=0}^\infty$ is an i.i.d.\ sub-Gaussian sequence with mean zero and covariance matrix $I_{d_m}$. The noise scale \(\alpha_{m}^{(k)}\) is 
kept positive within each epoch to guarantee sufficient excitation, while gradually decaying across epochs to anneal exploration as the player gains more information about the environment.

\begin{algorithm}
\caption{Radically Uncoupled Multi-agent $\epsilon$-Greedy ILS Algorithm}
\label{algorithm.indepedentlearning.Mplayer}
\begin{algorithmic}[1]
\STATE \textbf{Require:} Epoch length $\{\tau_m^{(k)}\}_{k\ge 1}$, exploration noise scales $\{\alpha_m^{(k)}\}_{k\ge 0}$, RLS parameter $\beta_m$, feasible set $\mathcal{S}_m$, for $1\le m \le M$

\FOR{$t = 0, 1, 2, \ldots,$}
    \FOR{$m = 1, 2, \ldots, M$}
        \STATE Calculate max number of epochs up to time $t$: $k_m = \sup\{k:\bar{\tau}^{(k)}_m \le t\}$
        \IF{$t = 0$}
            \STATE Initialize: select $(\tilde{\Theta}^{(0)}_m,\tilde{\Sigma}^{(0)}_{m}) \in \mathcal{S}_m$
            \STATE Compute $F^{(0)}_m = \Phi_m(\tilde{\Theta}^{(0)}_m,\tilde{\Sigma}^{(0)}_{m},x,\alpha^{(0)}_m)$ and $f^{(0)}_m = \phi_m(\tilde{\Theta}^{(0)}_m,\tilde{\Sigma}^{(0)}_{m},x,\alpha^{(0)}_m)$
        \ELSIF{$t = \bar{\tau}^{(k_m)}_m$}
            \STATE Estimate and select $(\tilde{\Theta}^{(k_m)}_m,\tilde{\Sigma}^{(k_m)}_{m})$ via \eqref{eqn.fallback}
            \STATE Compute $F^{(k_m)}_m = \Phi_m(\tilde{\Theta}^{(k_m)}_m,\tilde{\Sigma}^{(k_m)}_{m},x_t,\alpha^{(k_m)}_m)$ and $f^{(k_m)}_m = \phi_m(\tilde{\Theta}^{(k_m)}_m,\tilde{\Sigma}^{(k_m)}_{m},x_t,\alpha^{(k_m)}_m)$
        \ENDIF
        \STATE Apply control $u_{t}^{m,(k_m+1)}$ as in \eqref{perturbed_strategyk}
    \ENDFOR
    \STATE Observe state $x_{t+1}$ from \eqref{state}
\ENDFOR
\end{algorithmic}
\end{algorithm}

Algorithm~\ref{algorithm.indepedentlearning.Mplayer} summarizes the procedure executed concurrently by all players and models learning in the radically uncoupled environment described in Section~\ref{sec:independent_background}. First, the algorithm satisfies the radically uncoupled requirement of \cite{foster2006regret}: each player's updates depend solely on her own observed state trajectory and applied controls, without access to other players' payoffs, actions, or admissible strategy sets. 
Second, the $\epsilon$-greedy ILS scheme 
naturally captures the behavior of a player who treats the environment as a single-agent online LQ control problem. 
This choice is further justified by its optimality in the single-agent (average-reward) LQ setting:  the $\epsilon$-greedy ILS scheme achieves optimal   regret  with exponentially growing epochs and exploration noise decaying as \((\tau^{(k+1)}_m)^{-1/4}\) \citep{simchowitz2020naive}.

The central question we address is whether the joint learning dynamics generated by Algorithm~\ref{algorithm.indepedentlearning.Mplayer} converge to the complete-information feedback Nash equilibrium.  To answer this question, we first characterize the feedback Nash equilibrium in the next section.

\section{Feedback Nash Equilibrium under Complete Information}\label{section:model_complete}

This section analyzes the LQ stochastic game introduced in Section \ref{sec:independent_background} under complete information, assuming that the true system dynamics \eqref{state} and all players' strategies are known to every player.

We consider   deterministic linear feedback strategies of the form \(u^m_t = -F_m x_t + f_m\) with \(F_m \in \mathcal{A}_m\) and \(f_m \in \mathcal{B}_m\). Unlike the learning setting in  Section \ref{sec:independent_background}, no exploration noise is introduced since players have complete information. Let 
$\mathcal{A} := \prod_{m=1}^M \mathcal{A}_m$ and 
$\mathcal{B} := \prod_{m=1}^M \mathcal{B}_m$ be the joint feedback gains and  offsets of all players, respectively, and let
$ F = \{F_m\}_{1\le m\le M}$ 
and 
$f = \{f_m\}_{1\le m\le M} $  be  generic elements in $  \mathcal{A}$
and 
$ \mathcal{B}$, respectively. 
We write \(\mathcal{A}_{-m} := \prod_{j\neq m} \mathcal{A}_j\) and \(\mathcal{B}_{-m} := \prod_{j\neq m} \mathcal{B}_j\), and denote by 
\((F_{-m},f_{-m})=(\{F_j\}_{j \neq m},\{f_j\}_{j \neq m})\in \mathcal{A}_{-m} \times \mathcal{B}_{-m}\)
the strategies of all players other than player \(m\).

Given other players' strategies  \((F_{-m},f_{-m})\in \mathcal{A}_{-m} \times \mathcal{B}_{-m}\) and initial state \(x \in \mathbb{R}^n\), player \(m\) considers the minimization problem 
\begin{align}\label{obj.ne}
&\inf_{(F_m,f_m)\in \mathcal{A}_m\times\mathcal{B}_m} J_m \big( (F_m,f_m) ; (F_{-m},f_{-m}),x \big), \\\nonumber
&\text{with} \quad J_m \big( (F_m,f_m) ; (F_{-m},f_{-m}),x \big)  
= \mathbb{E} \left[\sum_{t=0}^{\infty} \rho^{t} Y_m\big(x_t,x_{t+1},-F_m x_t+f_m\big)\right],
\end{align}
where $\{x_t\}_{t=0}^\infty$
follows 
the dynamic \eqref{state} 
with the strategy profile $(F_m,f_m,F_{-m},f_{-m})$:
\begin{equation}\label{state3}
x_{t+1} = \left(A_0+\sum_{j\neq m} B_j f_j\right) + \left(A-\sum_{j\neq m} B_j F_j\right) x_t + B_m(-F_m x_t+f_m)+\omega_t.
\end{equation}
Note that, under complete information, player $m$ knows the true state dynamics \eqref{state3}, in   contrast  to  the radically uncoupled setting in Section~\ref{sec:independent_background}, where player $m$ perceives the state as evolving according to the misspecified model \eqref{dynamic:statistical}.

Following the standard treatment of LQ games \citep{bacsar1998dynamic,sun2019linear}, we define a feedback Nash equilibrium as a strategy profile from which no player has an incentive to unilaterally deviate.

\begin{definition}\label{def:feedback_nash}
A policy profile \((F^*,f^*)=(\{F_m^*\}_{1\le m \le M}, \{f_m^*\}_{1\le m \le M}) \in \mathcal{A} \times \mathcal{B}\) is a feedback Nash equilibrium if for all 
$1\le m\le  M$ and $x\in \mathbb R^n$,
$ J_m \big( (F^*_m,f^*_m) ; (F^*_{-m},f^*_{-m}), x \big)$ is finite, and 
$  (F^*_m,f^*_m)=\arg\inf_{(F_m,f_m)\in \mathcal{A}_m\times\mathcal{B}_m} J_m \big( (F_m,f_m) ; (F^*_{-m},f^*_{-m}),x \big).
$ 
\end{definition}

To characterize best responses, we assume that given the   strategy profile of   other players, each player's optimization problem \eqref{obj.ne} admits a unique optimal strategy.

\begin{assumption}\label{ass.def.bestresponsemap} (Well-posedness of best responses)
    For all \(1\le m\le M\), \(x \in \mathbb{R}^n\), and \((F_{-m},f_{-m})\in \mathcal{A}_{-m} \times \mathcal{B}_{-m}\),
    the infimum in problem \eqref{obj.ne} is attained at a unique minimizer that is independent of the initial state \(x\).
    Moreover, this minimizer is continuous in \((F_{-m},f_{-m})\).
\end{assumption}

We denote the unique minimizer of problem \eqref{obj.ne} by
\((\psi_m(F_{-m},f_{-m}),\, \varphi_m(F_{-m},f_{-m}))\)
and refer to it as player \(m\)'s \textit{best response} to the strategy profile \((F_{-m},f_{-m})\) of the other players.
This defines a \textit{collective best-response map}
\(\Psi: \mathcal{A}\times\mathcal{B} \to \mathcal{A}\times\mathcal{B}\) by
\begin{equation}\label{def.bestresponsemap}
\Psi((F,f)) = \Bigl( \bigl\{ \psi_m(F_{-m},f_{-m}) \bigr\}_{1\le m\le M},\;
                      \bigl\{ \varphi_m(F_{-m},f_{-m}) \bigr\}_{1\le m\le M} \Bigr).
\end{equation}
Because each component of \(\Psi\) is continuous, \(\Psi\) itself is continuous on the compact set
\(\mathcal{A}\times\mathcal{B}\); hence, by Brouwer's fixed-point theorem, \(\Psi\) admits at least one fixed point.
It follows directly from Definition~\ref{def:feedback_nash} and the definition of \(\Psi\) that a feedback Nash equilibrium corresponds to a fixed point of this map. We state this existence and characterization   in the following lemma.

\begin{lemma}\label{lemma:characterize_ne}
    Under Assumption~\ref{ass.def.bestresponsemap}, there exists a fixed point \((F^*,f^*)\) of \(\Psi\), i.e., \(\Psi((F^*, f^*)) = (F^*, f^*)\); 
    moreover, a strategy profile is a feedback Nash equilibrium  if and only if it is a fixed point of \(\Psi\).
\end{lemma}

To analyze the convergence of Algorithm \ref{algorithm.indepedentlearning.Mplayer}, we further impose the following (global) stability condition on the Nash equilibrium.

\begin{assumption}\label{ass.Psi_contractive_selfmap} (Stability of the feedback Nash equilibrium)
    There exist constants \(\mu>0\) and \(0<\zeta<1\) such that for all \((F,f)\in \mathcal{A} \times \mathcal{B}\),
    \begin{equation*}
    \big\Vert \Psi((F,f))-\Psi((F^*,f^*))\big\Vert_{\mu,1,2}
    \le \zeta \,\big\Vert (F,f)-(F^*,f^*)\big\Vert_{\mu,1,2},
    \end{equation*}
    where \((F^*,f^*)\) is a fixed point of \(\Psi\), and \(\Vert (F,f) \Vert_{\mu,1,2}\coloneqq
    \mu \bigl(\sum_{m=1}^M \Vert F_m \Vert\bigr)+\bigl(\sum_{m=1}^M\Vert f_m \Vert_2\bigr)\)
    is a weighted composite norm over \(\mathcal{A}\times\mathcal{B}\).
\end{assumption}

 Assumption  \ref{ass.Psi_contractive_selfmap}   implies  that \(\Psi\) admits a unique fixed point, and hence a unique feedback Nash equilibrium for the LQ game:   any  fixed point \((F',f')\) of $\Psi$ satisfies \(\Vert (F',f')-(F^*,f^*)\big\Vert_{\mu,1,2}=\big\Vert \Psi((F',f'))-\Psi((F^*,f^*))\big\Vert_{\mu,1,2} <\big\Vert (F',f')-(F^*,f^*)\big\Vert_{\mu,1,2}\), which forces 
$(F',f')= (F^*,f^*)$. The stability constant $\zeta$ also plays a key role in establishing the convergence of the learning dynamics and characterizing its convergence rate (see Theorem \ref{thm.convergencyne}). Without the stability condition, learning algorithms may fail to converge to the Nash equilibrium, even when it is unique (see \citet{mazumdar2020policy} for examples involving gradient-based algorithms).  We investigate two representative cases (that violate Assumption \ref{ass.Psi_contractive_selfmap}):  a unique equilibrium exists but is locally unstable (i.e., the map $\Psi$ is not a local contraction around the equilibrium; see Section~\ref{numerical.violateassumption} for a formal definition), or multiple locally  stable equilibria coexist. In the former case, Section~\ref{numerical.violateassumption} shows that Algorithm~\ref{algorithm.indepedentlearning.Mplayer} fails to converge and exhibits persistent oscillations. In the latter case,  Appendix~\ref{appendix.multiNE} 
  shows that  the algorithm may converge to one of the equilibria or exhibit oscillatory behavior.

\section{Main Convergence Result}\label{sec.convergencyanalysis}

This section shows that, under suitable conditions, 
the asynchronous $\epsilon$-greedy ILS algorithm in Algorithm \ref{algorithm.indepedentlearning.Mplayer} converges to the feedback Nash equilibrium (Definition \ref{def:feedback_nash}) and  further characterizes its convergence rate.

We shall assume the state dynamics \eqref{state} is stable under all admissible policy profiles.

\begin{assumption}\label{ass.stationary_x}
It holds that $ 
       \sup_{F\in \mathcal{A}} \Big\| A- \sum\limits_{m=1}^M B_m F_m \Big\| < 1.
$ 
\end{assumption}

Assumption \ref{ass.stationary_x} ensures that the state process $\{x_t\}_{t=0}^\infty$ in \eqref{state} remains stable under every admissible strategy profile (see \eqref{state3}). 
If this assumption is violated, numerical examples in Section \ref{numerical.violateassumption} demonstrate that the state may explode. Consequently, the ridge regression in \eqref{def.RLS} may fail to yield meaningful estimates, making the \(\epsilon \)-greedy ILS algorithm unimplementable.

We further assume the Lipschitz continuity of the optimal strategy map with respect to parameter estimates. 
To state the assumption formally, we define for each player \(m\) and \((F_{-m},f_{-m}) \in \mathcal{A}_{-m}\times \mathcal{B}_{-m}\),
\begin{equation}\label{def.Thetam.Ff}
     \Theta^{ (F_{-m},f_{-m})}_m = \left[ A_0 + \sum_{j \neq m} B_j f_j,\; A - \sum_{j \neq m} B_j F_j,\; B_m \right]^\top,
\end{equation}
which captures the aggregate effect of other players' linear strategies on the state dynamics from player \(m\)'s perspective; see \eqref{state3}.

\begin{assumption}\label{ass.uniform.Phimcontinuity} 
   
    For all   $1\le m\le  M$, there exist constants $ \delta_m, L_m \in (0, \infty)$ such that for all   $(\tilde{\Theta}_m,\tilde{\Sigma}_m)$ with $\max\bigl\{\|\tilde{\Theta}_m - \Theta^{ (F_{-m},f_{-m})}_m\|,\;\|\tilde{\Sigma}_m-\Sigma_{\omega}\|\bigr\} \le \delta_m$, we have 
    \begin{enumerate}[(a)]
        \item 
        \label{item:feasible}$(\tilde{\Theta}_m,\tilde{\Sigma}_m) \in \mathcal{S}_m$.
        \item 
        \label{item:lipschitz}
        For all initial state $x$ and $\alpha_m\in [0,  \delta_m]$, the control problem \eqref{discount_cost_m} admits a unique minimizer $(\Phi_m(\tilde{\Theta}_m,\tilde{\Sigma}_m,x, \alpha_m)$, $\phi_m(\tilde{\Theta}_m,\tilde{\Sigma}_m,x, \alpha_m))$ satisfying 
        \begin{align*}
            \big\|\Phi_m(\tilde{\Theta}_m,\tilde{\Sigma}_m,x, \alpha_m) - \Phi_m(\Theta^{ (F_{-m},f_{-m})}_m,\Sigma_{\omega},x,0)\big\| &\le L_m \Delta_m,\\
            \big\|\phi_m(\tilde{\Theta}_m,\tilde{\Sigma}_m,x, \alpha_m) - \phi_m(\Theta^{ (F_{-m},f_{-m})}_m,\Sigma_{\omega},x,0)\big\|_2 &\le L_m \Delta_m,
        \end{align*}
        where  $\Delta_m = \max\bigl\{\|\tilde{\Theta}_m - \Theta^{ (F_{-m},f_{-m})}_m\|,\;\|\tilde{\Sigma}_m-\Sigma_{\omega}\|, \alpha_m\bigr\} $.
    \end{enumerate}
\end{assumption}

Assumption~\ref{ass.def.bestresponsemap} ensures that \((\Theta^{(F_{-m},f_{-m})}_m,\Sigma_{\omega})\in\mathcal{S}_m\) for all \((F_{-m},f_{-m}) \in \mathcal{A}_{-m}\times \mathcal{B}_{-m}\). Assumption~\ref{ass.uniform.Phimcontinuity}(\ref{item:feasible}) strengthens this by requiring that parameters sufficiently close to \((\Theta^{(F_{-m},f_{-m})}_m,\Sigma_{\omega})\) remain feasible. 
Consequently, once estimation errors become sufficiently small,  \eqref{eqn.fallback} accepts the new estimates and updates the parameters.  Assumption~\ref{ass.uniform.Phimcontinuity}(\ref{item:lipschitz}) 
requires the optimal strategy map to be locally Lipschitz continuous, with a neighborhood size   \(\delta_m\) and the Lipschitz constant \(L_m\) that are independent of the opponents' strategies  \((F_{-m},f_{-m})\). 
This regularity is needed for our convergence analysis, as it controls how estimation errors propagate through players' best-response updates (see Step 2 of the proof sketch in Section~\ref{subsec:pf_sketch}).  

For standard LQ games without cross-product or linear terms (i.e., \(H_m, q_m, r_m,\) and \(A_{0}\) in \eqref{state} and \eqref{eqn.running} are zero), Assumption \ref{ass.uniform.Phimcontinuity} follows from   classical sensitivity analysis \citep{konstantinov1993perturbation}, provided that  
the unconstrained minimizer of \eqref{obj.ne} 
lies in the interior of \(\mathcal{A}_m \times \mathcal{B}_m\). These results, however, rely on positive definiteness of the quadratic objective. To the best of our knowledge, corresponding sensitivity results for the more general class of LQ games considered here remain unavailable. 
For the dynamic Cournot competition model in Section~\ref{sec:dynamic_oligopoly}, we establish explicit sufficient conditions ensuring Assumption~\ref{ass.uniform.Phimcontinuity}; see  Appendix~\ref{appendix.oligopoly.suff_condition}.

For convergence analysis, we impose  the following condition on the   epoch lengths and exploration noise scales in    Algorithm~\ref{algorithm.indepedentlearning.Mplayer}.

\begin{condition}\label{condition:tauk_exp}
For all $1\le m \le M$ and     $k \ge 0$,
the  epoch length and exploration noise used by player $m$ in epoch $k$
satisfy  
\begin{equation*}
    \tau^{(k+1)}_m = \big\lfloor \tau^{(k+1)} + \tilde{\tau}^{(k+1)}_m \big\rfloor, \quad \alpha^{(k)}_m = \bar{\alpha} (\tau^{(k+1)}_m)^{-\nu_m}, 
\end{equation*}
where 
$ \tau^{(k+1)} = \underline{\tau} + \tau \lambda^{k+1} $ is the  base epoch length  with some  $\lambda > 1$,  $\tau > 0$, and $\underline{\tau}\ge 0$,
$\tilde{\tau}^{(k+1)}_m $ is an asynchronous shift  taking values in $ [0,\bar{\tau}]$ for some $\bar{\tau}>0$, 
and 
$\bar{\alpha}>0$ and  $\nu_m \in (0, \tfrac{1}{2})$ are the base exploration rate and decay rate, respectively.

\end{condition}

Condition \ref{condition:tauk_exp} 
captures the asynchronous learning of different players using a single-agent algorithm, due to  their unawareness of competitors.
Consistent with the single-agent online LQ learning literature \citep{simchowitz2020naive},
we assume an exponential growth schedule for the base epoch length  \(\{\tau^{(k)}\}_{k=1}^\infty \), reflecting that players deploy their policies for increasingly longer durations as the learning proceeds. 
The choice  $\lambda=2$ recovers    the standard doubling trick \citep{besson2018doubling}. 
Similarly, 
the decay rate \(\nu_m \in (0,\frac12)\) also follows the standard exploration schedule in single-agent learning, ensuring persistent excitation for consistent parameter estimation while allowing exploration to vanish asymptotically.

Unlike single-agent learning, we incorporate a player- and epoch-dependent shift \(\tilde{\tau }_{m}^{(k+1)}\) into the learning epochs to accommodate asynchronicity across players when updating their strategies (recall that strategy parameters are updated only at the end of each epoch). These shifts are allowed to be random (e.g., uniformly distributed on \([0,\bar{\tau}]\)), provided they are independent of the state noise \(\{\omega_t\}_{t=0}^\infty\) and the exploration noise \(\{v_t^m\}_{t=0}^\infty\). While we assume bounded shifts for simplicity of presentation, our convergence analysis can be readily extended to settings with polynomially growing shifts.

We are now ready to present the main convergence theorem. To facilitate the discussion, we introduce a few notations. Define $k_m(t) = \sup\{k : \bar{\tau}^{(k)}_m \le t\}$. Let $(F^{(k_m(t))}_m, f^{(k_m(t))}_m)$ denote the strategy parameters used by player $m$ at time $t$. 
We write \((F^{(k(t))}, f^{(k(t))})\) to denote the strategy profile where player \(m\) employs the strategy \((F^{(k_m(t))}_m, f^{(k_m(t))}_m)\) at time \(t\).

\begin{theorem}\label{thm.convergencyne}
Suppose Assumptions \ref{ass.def.bestresponsemap}-\ref{ass.uniform.Phimcontinuity} and Condition \ref{condition:tauk_exp} hold.
For any \(\delta \in (0,1)\), there exist constants \(c_0, c_1 > 0\) and \(T(\delta) = O((\log(1/\delta))^{c_1})\) such that, with probability at least \(1-\delta\), for all \(t \ge T(\delta)\),
\begin{equation}
\label{eq:high_probability_rate}
\|(F^{(k(t))}, f^{(k(t))}) - (F^*, f^*) \|_{\mu,1,2}
= O\bigl( \max\{ R_1(t,\delta),\; R_2(t,\delta),\; R_3(t) \} \bigr),
\end{equation}
where \(\|\cdot \|_{\mu,1,2}\) is defined in Assumption \ref{ass.Psi_contractive_selfmap}, and
\begin{align}\label{def.R123}
R_1(t,\delta) = \bigl((\log t)^{\frac32}+(\log\tfrac{1}{\delta})^{c_0}\bigr)\,
                 t^{\frac{\ln\zeta}{\ln\lambda}},\quad 
R_2(t,\delta) = \sqrt{\log t+\log\tfrac{1}{\delta}}\;
                 t^{-(\frac12 - \max\limits_m  \nu_m)},\quad 
R_3(t)        = t^{-\min\limits_m \nu_m } .
\end{align}
Consequently, 
it holds almost surely that as $t\to \infty$,
\[
\|(F^{(k(t))}, f^{(k(t))}) - (F^*, f^*) \|_{\mu,1,2}
= \tilde{O}\bigl(t^{-q}\bigr),
\quad 
q \coloneqq   \min\!\Bigl\{ -\frac{\ln\zeta}{\ln\lambda},\; \frac12 - \max\limits_m \nu_m ,\; \min\limits_m  \nu_m  \Bigr\}.
\]
\end{theorem}

Theorem \ref{thm.convergencyne} shows that Algorithm \ref{algorithm.indepedentlearning.Mplayer} converges to the complete-information equilibrium at a polynomial rate, both with high probability (in finite samples)
and almost surely (asymptotically).
The precise convergence rate is characterized by the interplay between the game characteristics and the hyperparameters governing each player's learning dynamics:
\begin{enumerate}[(1)]
\item  The exponent \(-\frac{\ln\zeta}{\ln\lambda}\) of \(R_1\)
is the idealized convergence rate
that would be achieved if all players performed exact best-response oracle updates at the end of each epoch. Indeed, each best-response update contracts the error by a factor of $\zeta$ (the stability constant in Assumption~\ref{ass.Psi_contractive_selfmap}). Since the epoch lengths grow geometrically with ratio $\lambda$, the number of updates completed by time $t$ is approximately \(\ln t/\ln\lambda\). Hence, the error at  time $t$ scales as 
 \(\zeta^{\ln t/\ln\lambda}= t^{\ln\zeta/\ln\lambda}\).
 A smaller value of $\lambda$ leads to more frequent best-response updates and hence faster asymptotic convergence.

\item 
The above idealized rate is, however, degraded  in our  radically uncoupled   learning setting, where players do not observe their opponents  and perform   estimation and exploration    under  misspecified models.
Consequently, the convergence rate is constrained by the minimum of $\frac12- \max\limits_m \nu_m $ and $\min\limits_m \nu_m$, which correspond to the bottleneck decay rates of least-squares estimation errors and exploration noise, respectively. 
The radically uncoupled structure is reflected in the dependence of the convergence rate on \emph{individual} exploration schedules: each player must independently generate informative data through its own exploration, even though all players' exploration noises simultaneously affect the state dynamics.

\end{enumerate}

If all players adopt $\nu_m=\frac14$, the optimal exploration noise decay rate for single-agent LQ learning \citep{simchowitz2020naive}, then the convergence rate of Algorithm \ref{algorithm.indepedentlearning.Mplayer} exhibits a bifurcation phenomenon governed by the stability of the Nash equilibrium. When $\zeta \le \lambda^{-\frac14}$ (i.e., the    equilibrium  is sufficiently stable), the convergence rate \eqref{eq:high_probability_rate} (with fixed $\delta$)  simplifies to
\begin{align*}
\label{eq:rate14}
\begin{split}
\|(F^{(k(t))}, f^{(k(t))}) - (F^*, f^*) \|_{\mu,1,2}
= O\bigl( \sqrt{\log t} \cdot t^{-\frac14} \bigr).
\end{split}
\end{align*}
which is confirmed by 
our numerical experiments in Section \ref{sec.oligopoly_incomplete_information}.   
This rate coincides with the parameter convergence rate that enables optimal regret guarantees in single-agent LQ learning \citep{simchowitz2020naive}.
In contrast, when 
 $\zeta > \lambda^{-\frac14}$
(i.e., the equilibrium is less stable), 
the convergence rate is   dominated by the factor $t^{\frac{\ln \zeta}{\ln \lambda}}=\zeta^{\frac{\ln t}{\ln \lambda}}$,
which depends on   the effective number of strategy updates performed up to time $t$. 

\section{Application: Dynamic Cournot Competition with Sticky Prices}\label{sec:dynamic_oligopoly}

In this section, we apply our LQ game framework to a dynamic Cournot competition model with sticky prices.
The dynamic Cournot competition model introduced in Section \ref{sec:moti} is a special case of the general LQ   game     in Section \ref{sec:independent_background},
 with price serving as the state and quantities as the control actions. 
 
 Specifically,
the correspondence between the two formulations is established in the following three aspects. First, the price dynamics \eqref{dynamic.oligopoly} can be expressed in the general form \eqref{state} by taking $A_0 = (1 - s)a$, $A = s$, $B_m = -(1 - s)b$, and $\Sigma_{\omega}=\sigma_{\omega}^2$. Second, a firm's misspecified statistical model \eqref{oligopoly.statisticalmodel} coincides with \eqref{dynamic:statistical} upon setting \(\Theta_m = [(1 - s')a', s', -(1 - s')b']^{\top}, \Sigma_m=\sigma_m^2\). 
Third, for player \(m\), maximizing the total expected discounted profit is equivalent to minimizing the following discounted cost:
\begin{equation}\label{oligopoly.costfunction}
\mathbb{E}\left[\sum_{t=0}^{\infty} \rho^{t} \left( \frac12 q^m_{t} - p_{t+1} + c_m \right) q^m_{t}\right].
\end{equation}
The running cost in 
\eqref{oligopoly.costfunction} aligns with the general form in 
\eqref{eqn.running} by setting the parameters as \(Q_m=0, R_m=\frac12, H_m=-1, q_m=0, r_m=c_m\). Consequently, Algorithm~\ref{algorithm.indepedentlearning.Mplayer} is directly applicable 
to this setting.
Appendix~\ref{appendix.algoimplement} provides further implementation details, including a verifiable characterization of the feasible set \eqref{def:feasible_set} for implementing \eqref{eqn.fallback}  and a procedure for computing the greedy strategy.

In the sequel,
we first present numerical experiments for a duopoly Cournot game to confirm the theoretical   convergence results and examine how the price stickiness parameter affects market outcomes along the transition path to equilibrium. We then show that publicly releasing aggregate quantity information significantly accelerates convergence and mitigates the adverse transition effects identified above. Finally, we demonstrate that the multi-agent $\epsilon$-greedy ILS algorithm remains effective under a weakly nonlinear inverse demand function and investigate cases violating Assumptions~\ref{ass.def.bestresponsemap}-\ref{ass.uniform.Phimcontinuity} to illustrate the role of these conditions. For simplicity, we   use Gaussian noise for both market disturbances and exploration noises.

\subsection{Numerical Experiments in Duopoly Markets}\label{sec.oligopoly_incomplete_information}

Duopoly market is a common setting in the dynamic Cournot competition model with sticky prices \citep{fershtman1987dynamic,heston2024dynamic}. In this section, we investigate the performance of Algorithm~\ref{algorithm.indepedentlearning.Mplayer} under different levels of price stickiness and examine the resulting market outcomes.

We use the following   model and algorithm parameters: number of firms \(M = 2\), discount factor \(\rho = 0.8\), noise variance \(\sigma_\omega^2 = 1.00\), inverse demand parameters \(a = 60.0\) and \(b = 0.3\), cost parameters \(c_1 = 20.0\) and \(c_2 = 30.0\), initial price \(p_0 = 40\), and production capacity limits \(\kappa_m = 1.50\), \(\kappa'_m = 500\) for \(m = 1,2\). 
We examine three   price stickiness levels: \(s = 0.2\) (low stickiness, fast price adjustment), \(s = 0.5\) (moderate stickiness), and \(s = 0.8\) (high stickiness, slow price adjustment). 
For each selected stickiness parameter, we perform \(200\) independent simulations of Algorithm~\ref{algorithm.indepedentlearning.Mplayer} with hyperparameters \(\lambda=1.1\), \(\tau=300\), \(\underline{\tau}=2000\), \(\bar{\tau}=200\), \(\nu_m=0.25\), \(\bar{\alpha}=2.0\), and \(\beta_m=0.01\) for \(m=1,2\). In each simulation, every player performs \(130\) updates over a total of \(T\approx 7.2\times10^8\) time steps. At each update, the asynchronous shift of each player is sampled independently from a uniform distribution on \([0,\bar{\tau}]\). 
In all simulations, no negative prices or quantities are observed, indicating that such events occur with negligible probability.

The chosen model parameters satisfy Assumptions~\ref{ass.def.bestresponsemap}--\ref{ass.uniform.Phimcontinuity}, and the algorithm parameters satisfy Condition~\ref{condition:tauk_exp}; see Appendix~\ref{appendix.oligopoly.suff_condition} for details. 
Hence, Theorem \ref{thm.convergencyne} applies to the settings considered here. The equilibrium strategies \((F^*,f^*)\) 
and the stability coefficient \(\zeta\) for each stickiness level are reported in Table~\ref{tab:regression_results_convergence}, where we note that \(F^*_1 = F^*_2\).

\subsubsection{Convergence Rates of Learning Dynamics}

To estimate the convergence rate,  we define for each epoch \(k\) the relative error of the updated strategy parameters with respect to the equilibrium and take the maximum across the two players:
\begin{equation}\label{def.REk}
\text{RE}^{(k)}_{F} = \max_{m=1,2}\left\{\frac{|F_m^{(k)} - F_m^*|}{|F_m^*|}\right\},\qquad 
\text{RE}^{(k)}_{f} = \max_{m=1,2}\left\{\frac{|f_m^{(k)} - f_m^*|}{|f_m^*|}\right\}.
\end{equation}
We regress the relative error against the elapsed time on a log-log scale.
For each epoch $k$, we record the relative errors together with the corresponding elapsed time step (using the cumulative time of the player attaining the maximum error).
To reduce Monte Carlo variability, we average both the relative errors and the elapsed time steps over  200 independent runs before fitting the power-law model
\(\text{RE}^{(k(t))} = \theta_0 t^{-\theta_1}\) using the final 50 epochs.

The estimated decay exponent  $\theta_1$ is remarkably close to the theoretical value of 0.25 across all tested values of $s$; see Table~\ref{tab:regression_results_convergence}.
This agreement is consistent with Theorem~\ref{thm.convergencyne}.
Indeed, for every tested value of   \(s\), the estimated stability coefficient \(\zeta\) satisfies  \(\zeta<\lambda^{-1/4}=0.9765\), 
implying that the equilibrium is sufficiently stable, and hence the convergence rate is governed by the decay of the estimation error and exploration noise. Figures~\ref{figOliregF} and \ref{figOliregf} further illustrate the predicted $t^{-1/4}$ decay   for   \(s = 0.2\) and \(s = 0.8\).

\begin{table}[!htbp]
\caption{Equilibrium strategies, stability coefficients, and convergence rates with different price stickiness levels}
\centering
\begin{tabular}{lccccccccc}
\hline
Stickiness & $F^*_1 (F^*_2)$ & $f^*_1$ & $f^*_2$ & Stability & \multicolumn{2}{c}{$\text{RE}_{F}$} & \multicolumn{2}{c}{$\text{RE}_{f}$} \\
$s$ & & & & $\zeta$ & $\theta_0$ & $\theta_1$ & $\theta_0$ & $\theta_1$ \\
\hline
0.2 & -0.1160 & 16.8841 & 9.0615 & 0.5831 & 1.1209 & 0.2699 & 1.0407 & 0.2573 \\
0.5 & -0.3390 & 6.6769 & -1.4260 & 0.5625 & 0.9550 & 0.2498 & 3.8618 & 0.2740 \\
0.8 & -0.6520 & -7.0131 & -15.6836 & 0.5324 & 1.6430 & 0.2486 & 1.0045 & 0.2547 \\
\hline
\end{tabular}
\label{tab:regression_results_convergence}
\end{table}

\begin{figure}[!htb]
\centering
\begin{minipage}{0.45\textwidth}
\centering
\includegraphics[width=\textwidth]{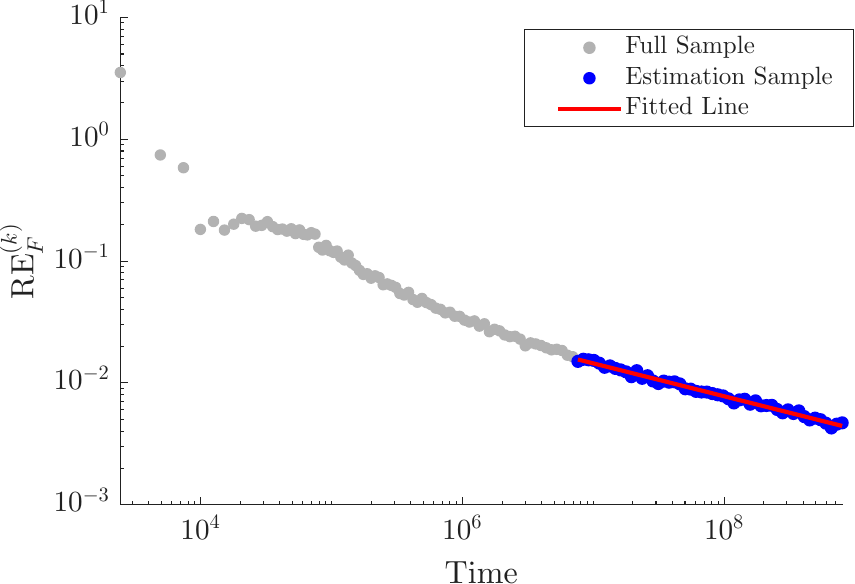}
\subcaption{$s=0.2$}
\end{minipage}
\hfill
\begin{minipage}{0.45\textwidth}
\centering
\includegraphics[width=\textwidth]{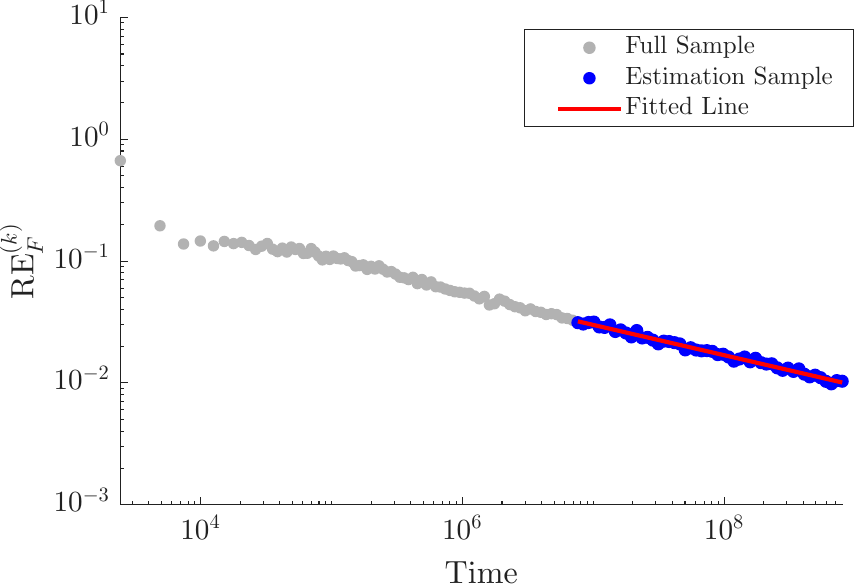}
\subcaption{$s=0.8$}
\end{minipage}
\caption{Convergence of feedback gain $F$ under different price stickiness levels}  
\label{figOliregF}
\end{figure}

\begin{figure}[!htb]
\centering
\begin{minipage}{0.45\textwidth}
\centering
\includegraphics[width=\textwidth]{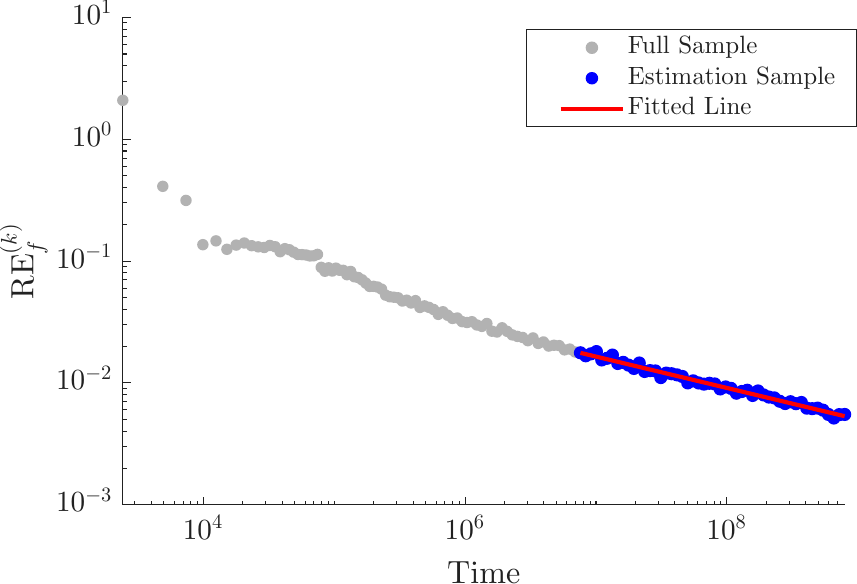}
\subcaption{$s=0.2$}
\end{minipage}
\hfill
\begin{minipage}{0.45\textwidth}
\centering
\includegraphics[width=\textwidth]{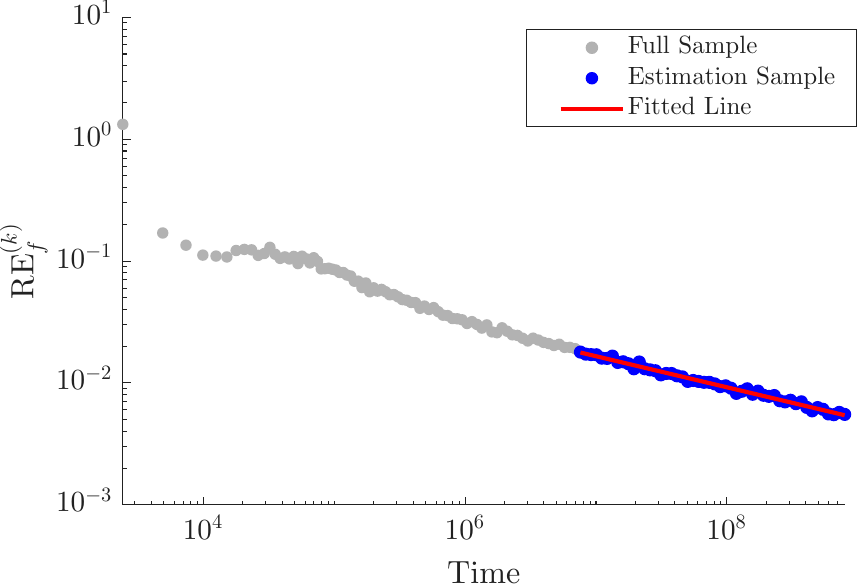}
\subcaption{$s=0.8$}
\end{minipage}
\caption{Convergence of offset $f$ under different price stickiness levels}
\label{figOliregf}
\end{figure}

\subsubsection{Market Outcomes under Low and High Price Stickiness}
\label{sec.oligopoly.market_dynamics}

We now compare the market outcomes generated by the learning dynamics under  low  (\(s=0.2\)) and   high   (\(s=0.8\)) price stickiness. We focus on  four  key performance measures: average price, firm profits, economic welfare, and market concentration.

We begin by introducing the corresponding performance measures under the complete-information Nash equilibrium $(F^*,f^*)$, 
  which serves as the benchmark for evaluating the learning dynamics.
Under the equilibrium strategy, the price process \(\{p_t\}_{t=0}^\infty\) defined by \eqref{dynamic.oligopoly} is asymptotically stationary,
allowing the long-run market performance measures to be defined in terms of its stationary distribution.
As \(t\to\infty\), the distribution of \(p_t\) converges to a Gaussian distribution with mean and variance
\begin{equation*}
\mu_p = \frac{(1 - s)(a - b\sum_{m=1}^2 f^*_m)}{1 - \big(s + (1-s)b\sum_{m=1}^2 F^*_m\big)}, \qquad
\sigma_p^2 = \frac{\sigma_\omega^2}{1 - \big(s + (1-s)b\sum_{m=1}^2 F^*_m\big)^2}.
\end{equation*}
The expected stationary profit of firm \(m\) under the equilibrium is then given by
\begin{equation}\label{eq:profit_def}
\Pi^*_m = \mathbb{E}\Bigl[\bigl(p_{t+1} - c_m - \tfrac12 q^m_t\bigr) q^m_t\Bigr],
\end{equation}
where \(q^m_t = -F^*_m p_t + f^*_m\), \(p_{t+1}\) follows \eqref{dynamic.oligopoly}, and \(p_t \sim \mathcal{N}(\mu_p, \sigma_p^2)\).
To measure overall social welfare, including both producers and consumers, we use total surplus (TS), a standard welfare measure in Cournot games \citep{cellini2007differential,bonatti2017dynamic}:
\begin{equation*}
\mathrm{TS}^* = \sum_{m=1}^2 \Pi^*_m + \mathbb{E}\Bigl[\tfrac12(a - p_{t+1}) \sum_{m=1}^2 q^m_t\Bigr],
\end{equation*}
where \(\Pi^*_m\) is given by \eqref{eq:profit_def},
and $q^m, p_{t+1}$ are defined as in \eqref{eq:profit_def}. 
Finally, to measure market concentration, we use the Herfindahl–Hirschman Index (HHI), the standard measure in Cournot games \citep{bonatti2017dynamic}:
\begin{equation*}
\mathrm{HHI}^* = \mathbb{E}\left[\sum_{m=1}^2 \left(\frac{q^m_t}{\sum_{m=1}^2 q^m_t}\right)^{\!2}\right],
\end{equation*}
with    $q^m$   defined as in \eqref{eq:profit_def}. 
The HHI captures how market output is distributed across firms, with higher values indicating greater concentration.
Table \ref{tab:equilibrium_indicators} reports 
the   values of \(\mu_p\), \(\sigma_p^2\), \(\Pi^*_1\), \(\Pi^*_2\), \(\mathrm{TS}^*\), and \(\mathrm{HHI}^*\) with different price stickiness. 

\begin{table}[!htbp]
\centering
\caption{Equilibrium values of market indicators}
\begin{tabular}{lccccccc}
\hline
\(s\) & \(\mu_p\) & \(\sigma_p^2\) & \(\Pi^*_1\) & \(\Pi^*_2\) & \(\mathrm{TS}^*\) & \(\mathrm{HHI}^*\) \\
\hline
0.2 & 48.82 & 1.0213 & 395.62 & 168.71 & 772.63 & 0.5220 \\
0.8 & 48.02 & 2.0874 & 386.24 & 160.08 & 784.42 & 0.5237 \\
\hline
\end{tabular}
\label{tab:equilibrium_indicators}
\end{table}

We now investigate the evolution of the above performance measures under the learning dynamics. 
For each of the 200 independent simulations of Algorithm~\ref{algorithm.indepedentlearning.Mplayer}, we compute at every epoch the sample mean of each performance indicator \(W\) (average price, firm profits, TS, HHI) from 300 time points sampled uniformly throughout the epoch, representing and approximating the true epoch‑average performance despite varying epoch lengths. 
Let \(W^{(k)}\) denote this sample mean for epoch \(k\). The deviation from the equilibrium benchmark is \(\Delta W^{(k)} = W^{(k)} - W^*\). We examine the bias \(\mathbb{E}[\Delta W^{(k)}]\) to capture systematic over‑/underestimation, where the expectation is taken over the 200 independent runs. 
To improve plot readability, all results reported below exclude the initial fifteen epochs.

Figure~\ref{fig:bias_price} shows that the bias of the epoch‑average price, \(\mathbb{E}\Delta\mu_p^{(k)}\), is negative for \(s=0.2\) (the price is systematically lower than the equilibrium level) but oscillates around zero for \(s=0.8\). 
\begin{figure}[!htbp]
\centering
\begin{minipage}{0.45\textwidth}
\centering
\includegraphics[width=\textwidth]{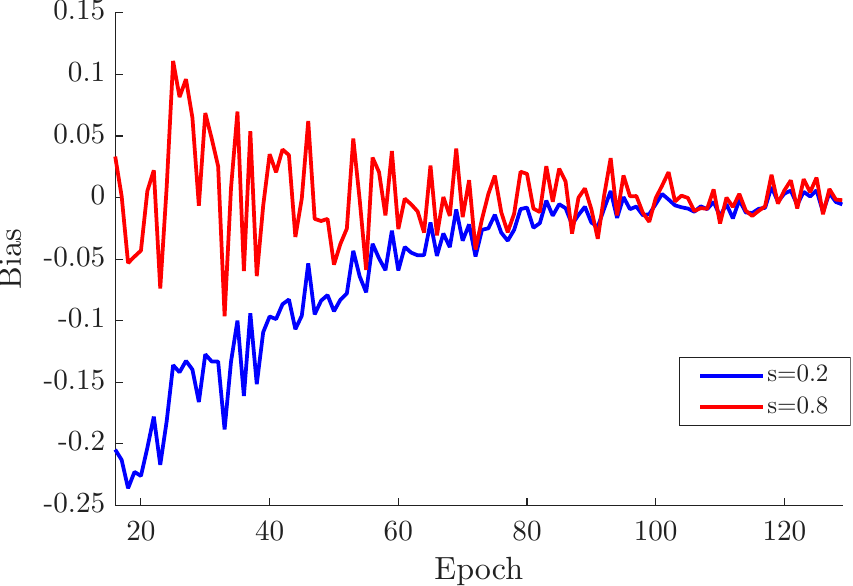}
\caption{Bias in average price}
\label{fig:bias_price}
\end{minipage}
\end{figure}
Turning to the firm level, the profits of the two firms exhibit systematic negative biases under both stickiness levels (Figure~\ref{fig_bias_profit}). 
The extent of these biases depends on the model parameters. Under our parameter choices, for firm~1, which has the lower marginal cost, the negative bias is substantially larger under \(s=0.8\) than under \(s=0.2\). For firm~2, the magnitude of the negative bias is comparable across the two stickiness levels, but the bias exhibits greater fluctuation across epochs under \(s=0.8\).

\begin{figure}[!htb]
\centering
\begin{minipage}{0.45\textwidth}
\centering
\includegraphics[width=\textwidth]{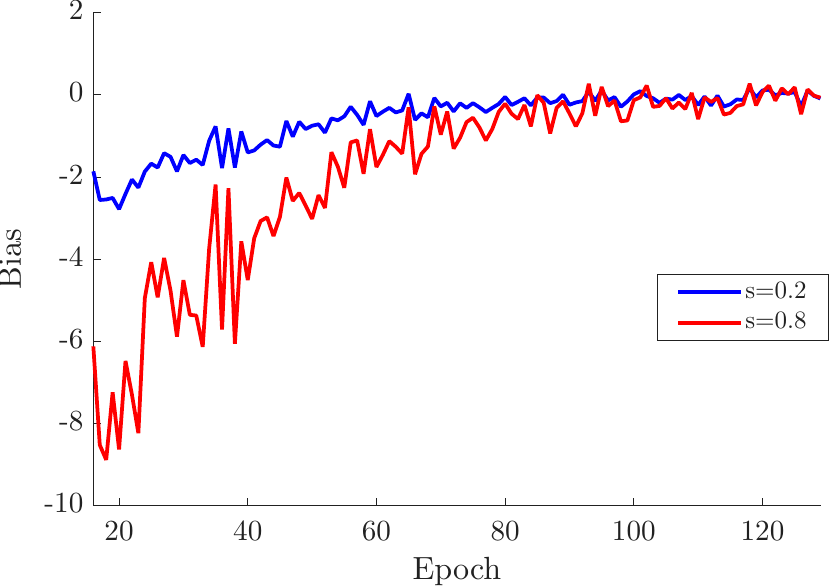}
\subcaption{Firm 1}
\end{minipage}
\hfill
\begin{minipage}{0.45\textwidth}
\centering
\includegraphics[width=\textwidth]{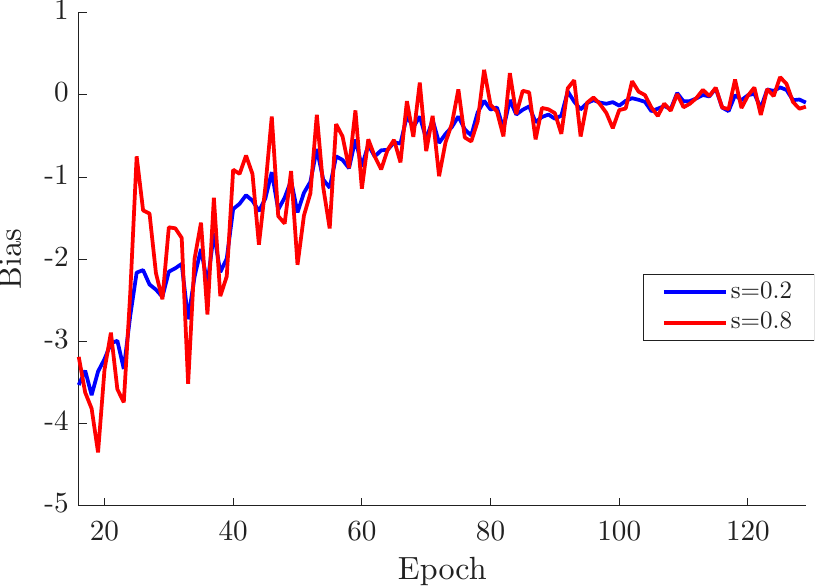}
\subcaption{Firm 2}
\end{minipage}
\caption{Bias in firm profits}
\label{fig_bias_profit}
\end{figure}

At the society level, the bias of total surplus (TS) is shown in Figure~\ref{fig:bias_TS}. TS exhibits a slight positive bias for \(s=0.2\), while for \(s=0.8\) it remains consistently below its equilibrium level (negative bias). 
\begin{figure}[!htb]
\centering
\begin{minipage}{0.45\textwidth}
\centering
\includegraphics[width=\textwidth]{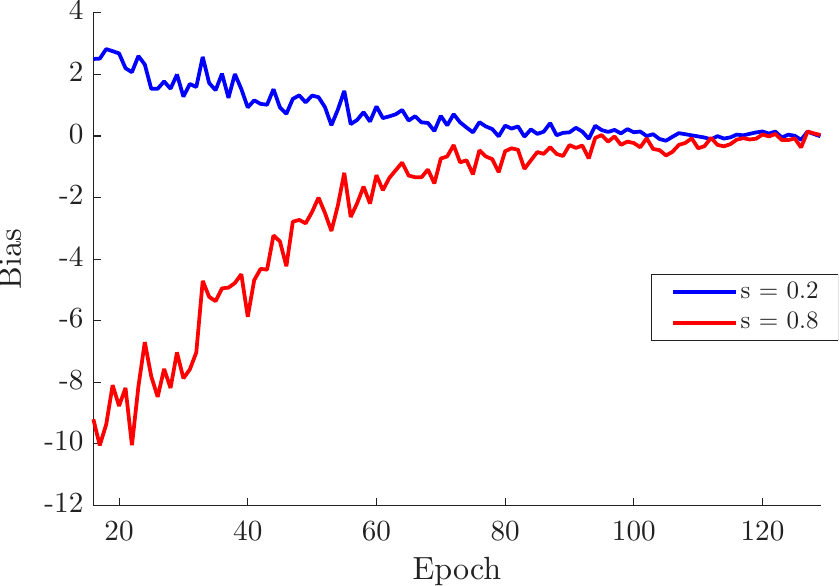}
\caption{Bias in total surplus}
\label{fig:bias_TS}
\end{minipage}
\hfill
\begin{minipage}{0.45\textwidth}
\centering
\includegraphics[width=\textwidth]{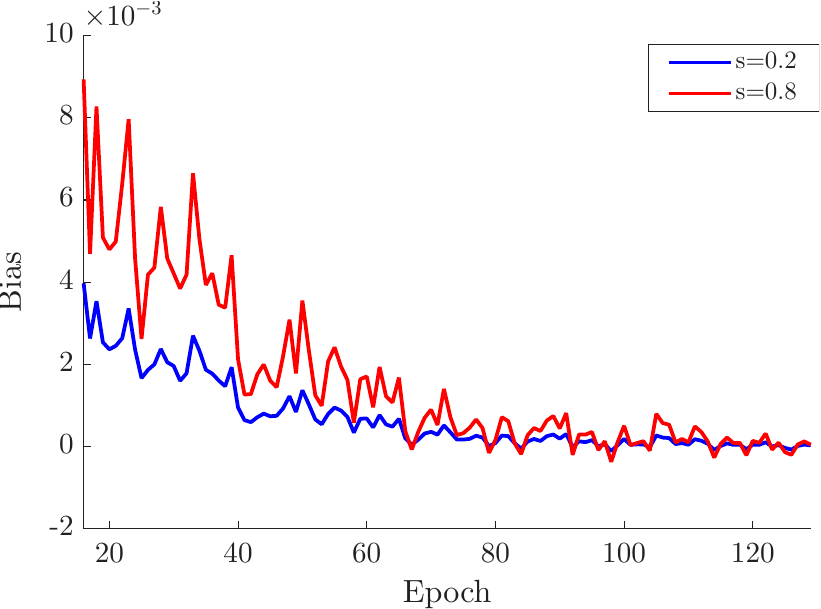}
\caption{Bias in HHI}
\label{fig_hhi_bias}
\end{minipage}
\end{figure}
Regarding market concentration, during the transition from an emerging to a mature market, the HHI remains above its equilibrium level and gradually declines toward it (Figure~\ref{fig_hhi_bias}). This over-concentration phenomenon is consistent with observations in the literature on Cournot games with incomplete information (see \cite{bonatti2017dynamic}). The extent of over-concentration is more pronounced under higher stickiness (\(s=0.8\)).

Taken together, these results reveal both common features and important differences between the two price-stickiness regimes. From the perspective of individual firms, profits remain below their equilibrium benchmarks throughout the transition under both regimes. At the societal level, however, total surplus remains persistently below its equilibrium benchmark only under high price stickiness, while market concentration is elevated in both regimes and is substantially higher when price stickiness is high. These findings suggest that radically uncoupled learning is particularly costly in highly sticky markets, reducing both firm profitability and overall social welfare during the transition to equilibrium. 
Moreover, these qualitative patterns are robust across a range of model and algorithm parameters we test. In particular, we obtain similar results for markets with more than two firms and when players employ heterogeneous exploration noise decay rates.

\subsection{Accelerating Convergence Through Aggregate Information}

The findings of the previous section suggest that under high price stickiness, both firms and a market regulator have strong incentives to accelerate convergence to the equilibrium. One natural intervention is for the regulator to publicly release  the aggregate market quantity \(Q_t = \sum_{m=1}^M q_t^m\)  in real time. 
With this additional information, the information structure no longer falls within the radically uncoupled setting studied earlier: each firm can now observe the aggregate outputs of its competitors  \(Q_t^{-m} := Q_t - q_t^m\),
and   thus explicitly account for their presence. 
This additional information enables firms to better infer the underlying market environment and competitors' behavior, leading them to adopt more informed best-response strategies. 
We therefore introduce a modified version of Algorithm \ref{algorithm.indepedentlearning.Mplayer} that exploits the publicly available aggregate market output. Although the algorithm is still executed independently by each firm, it is no longer radically uncoupled because it leverages the aggregate market quantity.
As we show below, this additional information substantially accelerates convergence to the equilibrium. 

 Firms continue to follow the asynchronous update schedules and exploration noise scales specified in Algorithm~\ref{algorithm.indepedentlearning.Mplayer}. The main modification concerns the system estimation step in Section~\ref{sec.olalgorithm}:  instead of performing a single RLS regression, each firm now conducts two separate RLS regressions using the additional information on aggregate market output. The first regression directly estimates the true price dynamics   \eqref{dynamic.oligopoly} based on the observed aggregate output. The second regression estimates the aggregate policy profile of the competitors, allowing the firm to anticipate their strategic behavior and formulate a more accurate best-response strategy. Specifically,  taking firm \(m\) as the representative firm,  it assumes that the aggregate quantity of the other firms is linear in price:
\[
Q_t^{-m} = -F'_{-m} p_t + f'_{-m} + v_t^{-m},
\]
where \(F'_{-m}\) and \(f'_{-m}\) are the aggregated feedback coefficient and offset term of all other firms, and \(v_t^{-m}\) is a residual term with perceived variance \(\sigma_{v,-m}^2\). At the \(k\)-th update time \(\bar{\tau}^{(k)}_m\), firm \(m\) uses the data collected during the \(k\)-th epoch to perform the two regressions, obtaining estimates of the price dynamics parameters \((\hat{s}^{(k)},\hat{a}^{(k)},\hat{b}^{(k)},\hat{\sigma}^{2,(k)}_\omega)\) and the prediction of other firms' aggregate behavior \((\hat{F}^{(k)}_{-m},\hat{f}^{(k)}_{-m},\hat{\sigma}^{2,(k)}_{v,-m})\). 
Based on these estimates, firm \(m\) then forms the following belief about the (controlled) price dynamics:
\begin{align}
\label{aggregate.dynamic.p}
 p_{t+1} &= \hat{s}^{(k)} p_t + (1-\hat{s}^{(k)}) \left( \hat{a}^{(k)} - \hat{b}^{(k)} \bigl(q^m_t - \hat{F}^{(k)}_{-m} p_t + \hat{f}^{(k)}_{-m} + v_t^{-m,(k)}\bigr) \right) + \omega^{(k)}_t, 
\end{align}
where 
 \(Q_t^{-m} \coloneqq  -\hat{F}_{-m}^{(k)} p_t + \hat{f}_{-m}^{(k)} + v_t^{-m,(k)}\) represents the estimated aggregate behavior of other firms, 
and 
\(\{\omega^{(k)}_t\}_{t=1}^\infty\) and \(\{v_t^{-m,(k)}\}_{t=1}^\infty\) are   i.i.d. zero-mean sub-Gaussian sequences with variances \(\hat{\sigma}^{2,(k)}_\omega\) and \(\hat{\sigma}^{2,(k)}_{v,-m}\), respectively. 
Note that  \eqref{aggregate.dynamic.p} can be rewritten in the form of \eqref{dynamic:statistical} with \(\Theta_m=\hat{\Theta}^{(k),c}_m\) and \(\Sigma_m=\hat{\Sigma}^{(k),c}_m\), where
\begin{align*}
    \hat{\Theta}^{(k),c}_m&=[(1-\hat{s}^{(k)})(\hat{a}^{(k)}-\hat{b}^{(k)}\hat{f}^{(k)}_{-m}),\ \hat{s}^{(k)}+(1-\hat{s}^{(k)})\hat{b}^{(k)}\hat{F}^{(k)}_{-m},\ -(1-\hat{s}^{(k)})\hat{b}^{(k)}]^{\top}, \\
    \hat{\Sigma}^{(k),c}_m &= \hat{\sigma}^{2,(k)}_\omega + ((1-\hat{s}^{(k)})\hat{b}^{(k)})^2 \hat{\sigma}^{2,(k)}_{v,-m}.
\end{align*}
Firm $m$ then applies the same fallback rule as in \eqref{eqn.fallback} to maintain feasibility of \((\hat{\Theta}^{(k),c}_m,\hat{\Sigma}^{(k),c}_m)\),   computes the greedy strategy using the accepted parameters, 
and   executes the resulting strategy with additive exploration noise in the next epoch, as in Algorithm \ref{algorithm.indepedentlearning.Mplayer}.

We test this modified  algorithm using  the same model and algorithm settings  as in Section~\ref{sec.oligopoly_incomplete_information}, 
focusing on the high‑stickiness case \(s = 0.8\). We perform 200 independent runs, each consisting of 100 updates per player, corresponding to a total of \(T \approx 4.2\times 10^7\) time steps. 
The   results 
show that the modified algorithm converges to equilibrium approximately six times faster, measured by the number of time steps required for the relative errors \(\text{RE}^{(k)}_{F}\) and \(\text{RE}^{(k)}_{f}\)    in \eqref{def.REk}  to fall below  0.01. 
Moreover,  Figures~\ref{fig:bias_profit_new} and \ref{fig:bias_TS_new} show that 
during the transition phase the negative biases of both firms’ profits and total surplus are substantially smaller (cf.~Figures \ref{fig_bias_profit} and \ref{fig:bias_TS}). 
These findings suggest that publicly revealing the aggregate quantity accelerates market maturation by improving the efficiency of the learning dynamics and bringing them closer to the equilibrium.

Intuitively, the acceleration from publicly releasing the aggregate quantity $Q_t$ arises from the improved information structure. In the radically uncoupled setting, each firm can only leverage its own exploration noise to improve estimation, whereas observing $Q_t$ allows firms to exploit the aggregate exploration signals of all market participants. Consequently, the regression estimation benefits from richer, more informative data. This improves the accuracy of parameter estimation and accelerates convergence to equilibrium.

\begin{figure}[!htbp]
\centering
\begin{minipage}{0.45\textwidth}
\centering
\includegraphics[width=\textwidth]{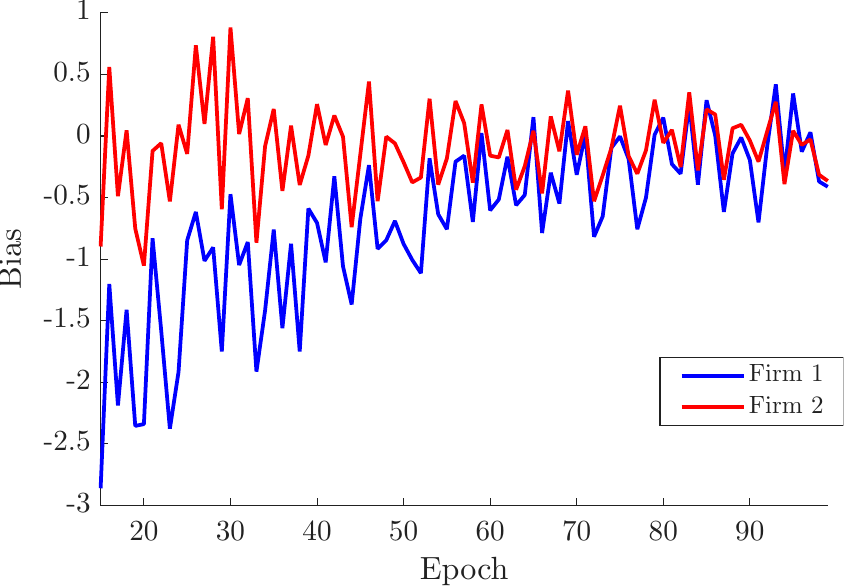}
\caption{Bias in firm profits}
\label{fig:bias_profit_new}
\end{minipage}
\hfill
\begin{minipage}{0.45\textwidth}
\centering
\includegraphics[width=\textwidth]{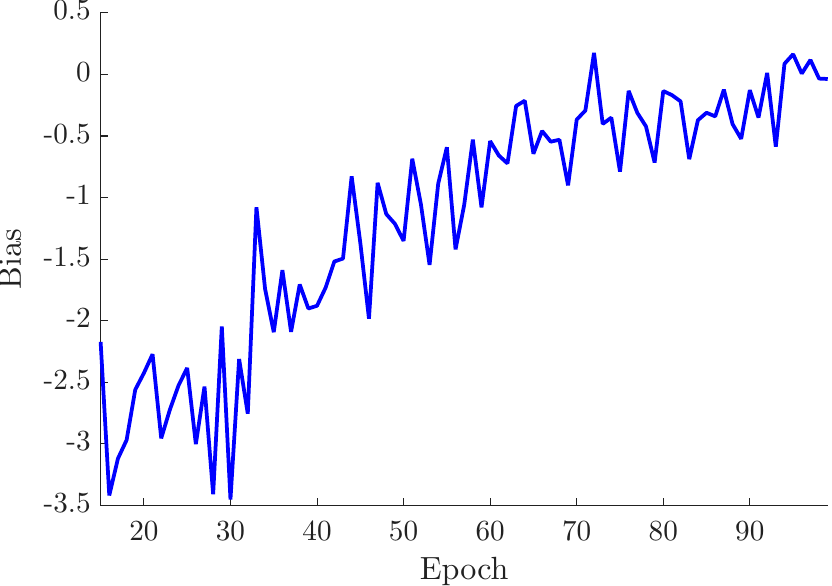}
\caption{Bias in total surplus}
\label{fig:bias_TS_new}
\end{minipage}
\end{figure}

\subsection{Robustness  with   Nonlinear Demand Functions}

This section illustrates the robustness of  Algorithm~\ref{algorithm.indepedentlearning.Mplayer}   when the true inverse demand function is weakly nonlinear. We consider the   dynamic Cournot competition model as in Section~\ref{sec:moti}, but replace the linear inverse demand with the   nonlinear function $a - b \bigl(\sum_{m=1}^M q^m_t\bigr)^{\delta}$, where $\delta>0$ is a curvature parameter \citep{bulow1983note}. 
The true price dynamics become
\begin{equation}\label{dym.weaknonlinear}
p_{t+1} = s p_t + (1 - s) \left( a - b \Bigl(\sum_{m=1}^M q^m_t\Bigr)^{\delta} \right) + \omega_t, 
\end{equation}
which  reduces to the linear case studied in Section~\ref{sec:moti} when  $\delta=1$. We now focus on the nonlinear case with $\delta \neq 1$. Each firm $m$ continues to maximize its discounted expected profit given by \eqref{oligopoly.profitfunction}. 
The complete-information Nash equilibria are nonlinear functions of price $p$, and are computed numerically by solving   coupled Bellman equations. 
In the radically uncoupled learning setting, each firm continues to  employ linear feedback strategies and adopt the misspecified linear   model \eqref{oligopoly.statisticalmodel}, since the true nonlinear form of the inverse demand function is unknown to the firm. Consequently, the firms' behaviors follow Algorithm~\ref{algorithm.indepedentlearning.Mplayer}, except that the price evolves according to the true nonlinear dynamics \eqref{dym.weaknonlinear}.

We set \(s=0.8\) and consider two nonlinearity parameters \(\delta=0.9\) and \(\delta=1.1\), with all other parameters fixed as in Section~\ref{sec.oligopoly_incomplete_information}. For each configuration, we first compute the feedback equilibrium numerically and conduct \(50\) independent simulations. In each simulation, we compare the trajectory generated by Algorithm~\ref{algorithm.indepedentlearning.Mplayer} with a benchmark trajectory in which all firms follow the equilibrium strategy throughout. Both trajectories are driven by the same sequence of price shocks. Denote by   \(q_t^m\) the      quantity produced by  firm $m$  under Algorithm~\ref{algorithm.indepedentlearning.Mplayer} at time $t$, and by   \(\bar q_t^{\,m} \) the corresponding equilibrium quantity.   For each epoch \(k\), we uniformly sample \(300\) time points \(\mathcal{T}_k\), measure the trajectory error by
\[
\mathrm{Err}^{(k)}
=
\max_{m=1,2}\frac{1}{300}\sum_{t\in\mathcal{T}_k}
\bigl(q^m_t-\bar{q}^{\,m}_t\bigr)^2,
\]
and report the average the   error  over the \(50\) independent simulations.

 Figure~\ref{fig:nonlinear_l2_error} shows the    error in production quantities 
  for \(\delta=0.9\) and \(\delta=1.1\).  In both cases, the error gradually decreases and eventually stabilizes at a positive level. The   error in the final epoch is \(0.1422\) for \(\delta=0.9\) and \(0.0484\) for \(\delta=1.1\), which are significantly smaller than the typical quantity levels observed during the simulations (in the range \(10\)--\(30\), depending on the firm and price level). The remaining error does not vanish because the true equilibrium strategy is nonlinear, while Algorithm~\ref{algorithm.indepedentlearning.Mplayer} restricts the learned strategy to be linear. Nevertheless, the approximation error is small relative to the scale of production decisions. These results show that, even under weak nonlinearity, the radically uncoupled learning algorithm converges to an accurate linear approximation of the complete-information equilibrium.

\begin{figure}[!htbp]
\centering
\begin{minipage}{0.45\textwidth}
\centering
\includegraphics[width=\textwidth]{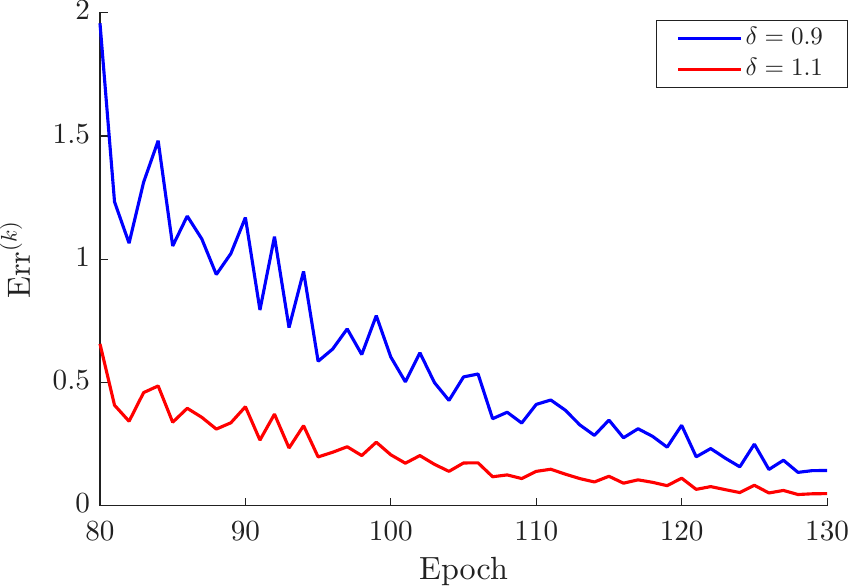}
\caption{Errors in production quantities for different $\delta$.}
\label{fig:nonlinear_l2_error}
\end{minipage}
\end{figure}

\subsection{Results Beyond Theoretical Assumptions}\label{numerical.violateassumption}

This section presents numerical experiments that illustrate the importance of the assumptions underpinning our theoretical analysis.  

Assumption~\ref{ass.stationary_x} concerns the stability of the state process \(\{x_t\}_{t=0}^\infty\) under  all  admissible strategy profiles; it affects the consistency of the RLS estimator and can be checked directly. The other three assumptions (Assumptions~\ref{ass.def.bestresponsemap}, \ref{ass.Psi_contractive_selfmap}, and \ref{ass.uniform.Phimcontinuity}) are about the game structure and are closely related. We therefore treat them as a group. They can be verified via sufficient conditions expressed solely in terms of the model parameters (see conditions \ref{condition.oligopoly.selfmap}-\ref{condition.oligopoly.contractivity} in Appendix~\ref{appendix.oligopoly.suff_condition}) for the Cournot model). Alternatively, we can detect violations by examining the number of equilibria and their local behavior. 
When an equilibrium \((F^*,f^*)\) exists, we call it \textit{locally stable} if 
 the collective best‑response map \(\Psi\) is well‑defined in a neighborhood of \((F^*,f^*)\) (i.e., problem \eqref{obj.ne} admits a unique minimizer for all \(1\le m\le M\), \(x \in \mathbb{R}^n\), and \((F_{-m},f_{-m})\) sufficiently close to \((F^*_{-m},f^*_{-m})\)) and the spectral radius of its Jacobian matrix at \((F^*,f^*)\) satisfies \(r(D\Psi(F^*,f^*)) < 1\); otherwise it is locally unstable. 
 Note that 
local stability is a weaker condition than   Assumption \ref{ass.Psi_contractive_selfmap}.
Thus, if the spectral radius of    
  $D\Psi$ at the equilibrium exceeds $1$,  then Assumption~\ref{ass.Psi_contractive_selfmap} must be violated.  
In addition, the existence of multiple equilibria also violates    Assumption~\ref{ass.Psi_contractive_selfmap}; we examine the behavior of Algorithm~\ref{algorithm.indepedentlearning.Mplayer} in this case in Appendix~\ref{appendix.multiNE}, where the algorithm may converge to one of the equilibria or exhibit oscillatory behavior.

Here we consider two parameter configurations under the Cournot setting. The common parameters across both groups are
$ 
a = 40.0$, $  b = 2.0$, $  s = 0.2$, and $\sigma_\omega^2 = 1.00$. 

\noindent\textbf{Group 1 (violation of state stability).} 
This group violates Assumption~\ref{ass.stationary_x} while satisfying the remaining assumptions, in particular having a globally stable equilibrium. The specific parameters are:
\begin{equation*}
\rho = 0.4,\; M = 2,\; c_1 = 10.0,\; c_2 = 20.0,\; \kappa_m = 0.7,\; \kappa'_m = 5000\;(m=1,2).
\end{equation*}

\noindent\textbf{Group 2 (violation of equilibrium stability).
} 
This group satisfies Assumption~\ref{ass.stationary_x} but violates Assumption~\ref{ass.Psi_contractive_selfmap}. 
The specific parameters 
(for a Cournot game with four players) are:
\begin{equation*}
\rho = 0.8,\; M = 4,\; c_1 = c_3 = 10.0,\; c_2 = c_4 = 20.0,\; \kappa_m = 0.12,\; \kappa'_m = 500\;(1\le m\le 4).
\end{equation*}
 Our numerical analysis suggests a unique equilibrium for this group. Computing the spectral radius of the Jacobian matrix at this equilibrium yields \(r(D\Psi) = 1.1160 > 1\), so the equilibrium is not locally stable and Assumption~\ref{ass.Psi_contractive_selfmap} is violated. 

For each group, we execute 200 independent runs of Algorithm~\ref{algorithm.indepedentlearning.Mplayer}. Each player uses the same algorithm parameters as in Section~\ref{sec.oligopoly_incomplete_information}, and the total number of time steps is set identically. We classify the outcome of each simulation into three categories:
\begin{itemize}
    \item \textit{Explode}: the state variable price \(p_t\) grows without bound, causing the algorithm to fail.
    \item \textit{Oscillatory}: the strategy parameters continue to fluctuate and do not settle near any fixed point. 
    \item \textit{Converge to Equilibrium}: the strategy parameters approach (a small neighborhood of) the feedback Nash equilibrium. 
\end{itemize}

Table~\ref{tab:experimental_outcomes} summarizes the percentage of simulations falling into each category for the two groups.

\begin{table}[!htbp]
\caption{Distribution of learning outcomes}
\centering
\begin{tabular}{lccc}
\hline
Parameter Group & Explode (\%) & Oscillatory (\%) & Converge to Equilibrium (\%) \\
\hline
Group 1  & 52.00 & 0.00 & 48.00  \\
Group 2  & 0.00 & 100.00 & 0.00  \\
\hline
\end{tabular}
\label{tab:experimental_outcomes}
\end{table}

The results reveal the following patterns. 
A violation of Assumption~\ref{ass.stationary_x} means that some admissible strategy profiles render the state process unstable. When the learning algorithm encounters such a profile, the state explodes; indeed, a substantial fraction of simulations for Group~1 (52\%) exhibit explosive behavior. However, there also exist admissible strategy profiles that are stabilizing. In the remaining 48\% of simulations, the particular realized sequence of strategy profiles \(\{(F^{(k)},f^{(k)})\}_{k=0}^\infty\) does not cause the state process to become unstable, and the algorithm runs without explosion even though 
Assumption~\ref{ass.stationary_x} fails. In these non‑explosive runs, since the other assumptions 
are satisfied, all paths converge to the feedback Nash equilibrium.

For Group~2, Assumption~\ref{ass.stationary_x} holds, so no explosion occurs. However, because the equilibrium is not locally stable, none of the simulation paths converge to it. Instead, all trajectories exhibit a persistent \emph{best-response cycle}: the strategy profile alternates between two distinct non-equilibrium points, each constituting the collective best response to the other. Similar cyclic behavior has been observed in evolutionary games with logit dynamics \citep{hommes2012multiple} and policy gradient  methods for LQ games \citep{mazumdar2020policy}. Figure~\ref{figviolate_oscillatory} shows a representative trajectory in the later epochs, illustrating that the strategy parameters continue to cycle without settling. This demonstrates that a unique Nash equilibrium alone is insufficient to guarantee convergence; local stability is also essential.

\begin{figure}[!htb]
\centering
\includegraphics[width=1.0\textwidth]{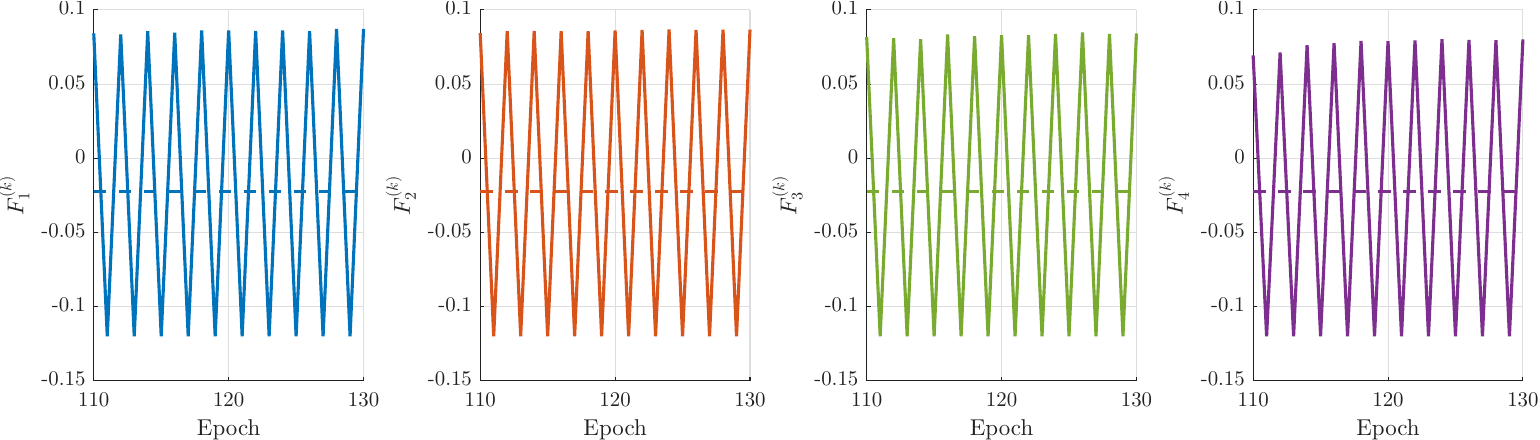}
\caption{Oscillatory dynamics when  Assumption~\ref{ass.Psi_contractive_selfmap} 
is violated (Group~2). The four subfigures correspond to four players. In each subfigure, the solid line is \(F^{(k)}_m\) and the dashed line is the equilibrium value \(F^*_m\).}
\label{figviolate_oscillatory}
\end{figure}

\section{Sketched Proof of Theorem \ref{thm.convergencyne}}\label{subsec:pf_sketch}

This section outlines the main steps in proving Theorem \ref{thm.convergencyne}. 
We first present the main ideas in the synchronous setting (Steps 1–3) and then discuss the additional technical challenges involved in extending the analysis to asynchronous updates (Step 4). Throughout this section, \(C\) denotes a generic finite constant independent of the epoch index \(k\) and the probability bound \(\delta\); its value may change from line to line.

In the synchronous case, all players share the same update times: \(  {\tau}^{(k)}_m = {\tau}^{(k)}\) for every \(m\). Let \((F^{(k)},f^{(k)})\) denote the strategy profile after the \(k\)-th update. From player \(m\)'s perspective, during epoch \(k\) the true state dynamics \eqref{state} induced by     all players' strategies become
\begin{equation*}\label{eq:effective_dyn}
x_{t+1} = (\Theta^{ (F^{(k-1)}_{-m},f^{(k-1)}_{-m})}_m)^{\top} z^m_t + \tilde\omega^m_t, 
\end{equation*}
where \(\Theta^{ (F^{(k-1)}_{-m},f^{(k-1)}_{-m})}_m\) is defined in \eqref{def.Thetam.Ff} with \((F_{-m},f_{-m})=(F^{(k-1)}_{-m},f^{(k-1)}_{-m})\), and $\{\tilde\omega^m_t\}_{t=0}^\infty$ 
is an i.i.d.\ sub-Gaussian sequence with mean zero and covariance matrix 
$\Sigma_{\omega} + \sum\limits_{j\neq m} (\alpha_j^{(k-1)})^2 B_j B_j^{\top}$. 
Crucially, 
in contrast to the single-player case \citep{simchowitz2020naive}, 
player \(m\) 
faces  a  misspecified   model \eqref{dynamic:statistical}, and hence  
cannot identify the true system parameters \((A_0,A,\{B_m\},\Sigma_{\omega})\). Instead, they will asymptotically learn the \textit{effective   system parameters} at the $k$-th epoch given by
\begin{equation}\label{def.Thetamkstar}
\Theta_{m}^{(k), *} = \Theta^{ (F^{(k-1)}_{-m},f^{(k-1)}_{-m})}_m,\qquad
\Sigma_{m}^{(k), *} = \Sigma_{\omega} + \sum_{j\neq m} (\alpha_j^{(k-1)})^2 B_j B_j^{\top}.
\end{equation}
We show that this procedure corresponds to a biased best response to $(F^{(k-1)}_{-m},f^{(k-1)}_{-m})$, and quantify the decay rate of the resulting bias in \eqref{diff.Ff}.

\textbf{Step 1: Estimation error bounds.}
We first quantify  the estimation error of the RLS estimate  \((\hat{\Theta}_m^{(k)},\hat{\Sigma}_m^{(k)})\) relative to the effective system parameters \((\Theta_{m}^{(k), *},\Sigma_{m}^{(k), *})\). Specifically, we establish the following high-probability   error bound: for any $\delta \in (0,1)$, there exists 
\(K_m(\delta)=O(\log\log(1/\delta))\) such that, with probability at least $1-\delta$,
\begin{equation}\label{estimationerror.ThetaSigma}
\forall k \ge K_m(\delta),\quad
\max\Bigl\{\bigl\|\hat{\Theta}_m^{(k)} - \Theta_m^{(k),*}\bigr\|,\;
\bigl\|\hat{\Sigma}_m^{(k)} - \Sigma_m^{(k),*}\bigr\|\Bigr\}
\;\le\; C\,\eta_m^{(k)}(\delta),
\end{equation}
where  the decay rate $\eta_m^{(k)}(\delta) $ is given by 
\begin{equation}\label{def.etak}
\eta_m^{(k)}(\delta) = \sqrt{\frac{k + \log(1/\delta)}{\lambda^{(1-2\nu_m)k}}}.
\end{equation}
The proof builds upon the arguments in \cite{simchowitz2020naive}, but requires substantially more involved analysis due to the presence of constant offsets in the dynamics and policies, which introduce non-excited directions, and  the need to jointly estimate the drift and covariance parameters \((\Theta_{m}^{(k), *},\Sigma_{m}^{(k), *})\). Extending the analysis to   asynchronous updates requires further  technical efforts (see Step 4).

\textbf{Step 2: From estimation error to strategy deviation.} 
Bounding $\|\Sigma_m^{(k),*} - \Sigma_{\omega}\|$ via \eqref{def.Thetamkstar} and combining it with \eqref{estimationerror.ThetaSigma} yield the 
following high-probability bound 
\begin{equation}\label{estimationerror.ThetaSigma2}
\max\{\|\hat{\Theta}_{m}^{(k)} - {\Theta}_{m}^{(k),*}\|,\;\|\hat{\Sigma}_{m}^{(k)} - \Sigma_{\omega}\|\} \le C\max\left\{\eta^{(k)}_m(\delta), \max\limits_{j\neq m} (\alpha_j^{(k-1)})^2\right\}.
\end{equation}
Since both $\eta_m^{(k)}(\delta)$ and $\alpha_m^{(k)}$  vanish as $k$ grows, the right-hand side of \eqref{estimationerror.ThetaSigma2} is smaller than the threshold   $\delta_m$ in Assumption~\ref{ass.uniform.Phimcontinuity} for all sufficiently large $k$. 
Assumption~\ref{ass.uniform.Phimcontinuity}(\ref{item:feasible}) then guarantees that player $m$ accepts the current estimate, so the bound \eqref{estimationerror.ThetaSigma2} also holds for the accepted parameter pair $(\tilde{\Theta}_m^{(k)},\tilde{\Sigma}_m^{(k)})$. By 
Assumption~\ref{ass.uniform.Phimcontinuity}(\ref{item:lipschitz}),   for any $\delta \in (0,1)$, there exists $K'_m(\delta)\in \mathbb N$ such that, with probability at least $1-\delta$, for all $k \ge K'_m(\delta)$,
\begin{equation}\label{eq:lip-f}
\max\left\{\left\|F^{(k)}_m - \Phi_m(\Theta_m^{(k),*},\Sigma_{\omega},x,0)\right\|,\left\|f^{(k)}_m - \phi_m(\Theta_m^{(k),*},\Sigma_{\omega},x,0)\right\|_2\right\} \le C\max\left\{\eta^{(k)}_m(\delta) ,\max_{1\le j\le M} \alpha_j^{(k-1)} \right\},
\end{equation}
where $(F^{(k)}_m,f^{(k)}_m)=(\Phi_m(\tilde{\Theta}_m^{(k)},\tilde{\Sigma}_m^{(k)},x,\alpha_m^{(k)}),\phi_m(\tilde{\Theta}_m^{(k)},\tilde{\Sigma}_m^{(k)},x,\alpha_m^{(k)}))$.

Now we observe the following key identity:  
\begin{align} \label{eq:equiv1}(\Phi_m(\Theta_m^{(k),*},\Sigma_{\omega},x,0),  \phi_m(\Theta_m^{(k),*},\Sigma_{\omega},x,0))  = (\psi_m(F_{-m}^{(k-1)},f_{-m}^{(k-1)}), \varphi_m(F_{-m}^{(k-1)},f_{-m}^{(k-1)})),
\end{align}
where the best response map $(\psi_m, \varphi_m)$ is given in \eqref{def.bestresponsemap}. 
Indeed, 
consider 
the complete-information problem \eqref{obj.ne} with other players' strategies \((F_{-m}^{(k-1)}, f_{-m}^{(k-1)})\). 
In this case, the dynamic \eqref{state3} can be rewritten in the form of \eqref{dynamic:statistical} with \(\Theta_m = \Theta_m^{(k),*}\), \(\Sigma_m = \Sigma_{\omega}\), and \(u^m_t = -F_m x_t + f_m\),
which means that   the aggregate effect of the opponents' linear strategies on the state evolution is captured by   \(\Theta_m^{(k),*}\). Since player \(m\) faces the same objective in both 
settings, solving the   complete-information problem is equivalent to solving the control problem \eqref{discount_cost_m} with the   parameters   \((\Theta_m^{(k),*},\Sigma_{\omega})\) and \(\alpha_m = 0\). This implies that  
\begin{equation}\label{eqn.relationship}
     J_m \big( (F_m,f_m) ; (F_{-m}^{(k-1)}, f_{-m}^{(k-1)}),x \big)=J^{(\Theta_m^{(k),*},\Sigma_{\omega})}_m \big( (F_m,f_m) ; x,0 \big),
\end{equation}
and  directly yields \eqref{eq:equiv1}.

Combining  \eqref{eq:lip-f} and \eqref{eq:equiv1} allows for interpreting 
$(F_m^{(k)},f_m^{(k)})$   as a \emph{noisy best response} to the opponents' previous strategy profile $(F_{-m}^{(k-1)},f_{-m}^{(k-1)})$.
The error can be controlled   in a high probability  by 
$\eta^{(k)}_m(\delta)$ and $\max\limits_{1\le j\le M} \alpha_j^{(k-1)} $. Specifically,   we prove that, for any $\delta \in (0,1)$, there exists 
\(K(\delta)=O(\log\log(1/\delta))\) such that, with probability at least $1-\delta$,
\begin{equation}\label{diff.Ff}
\forall k\ge K(\delta),\quad \|(F^{(k)},f^{(k)}) - \Psi((F^{(k-1)},f^{(k-1)}))\|_{\mu,1,2} \le C \xi_k(\delta),
\end{equation}
where 
\begin{equation}\label{def.xidelta}
\xi_k(\delta) = \sqrt{k+\log(1/\delta)}\;\lambda^{-(\frac12 - \max\limits_m \{\nu_m\})k} + \lambda^{-(\min\limits_m \{\nu_m\})k}.
\end{equation}

\textbf{Step 3: Convergence via 
the stability of the feedback Nash equilibrium.} 
Under Assumptions \ref{ass.def.bestresponsemap}–\ref{ass.Psi_contractive_selfmap}, there exists a unique stable feedback Nash equilibrium \((F^*,f^*)\). 
On the event \eqref{diff.Ff}, we have
\begin{equation*}
\begin{aligned}
\|(F^{(k)},f^{(k)}) - (F^*,f^*)\|_{\mu,1,2}
&\le \|(F^{(k)},f^{(k)}) - \Psi((F^{(k-1)},f^{(k-1)}))\|_{\mu,1,2}  \\
&\quad + \|\Psi((F^{(k-1)},f^{(k-1)})) - \Psi((F^*,f^*))\|_{\mu,1,2}\\
&\le C \xi_{k}(\delta) + \zeta \|(F^{(k-1)},f^{(k-1)}) - (F^*,f^*)\|_{\mu,1,2},
\end{aligned}
\end{equation*}
where \(\zeta\in(0,1)\) is given in Assumption \ref{ass.Psi_contractive_selfmap}. Solving this recursive inequality case‑by‑case, 
depending on the relative sizes of \(\zeta\), \(\lambda^{-(\frac12-\max\limits_m \nu_m )}\), and \(\lambda^{-\min\limits_m \nu_m }\), and 
using the fact that the epoch lengths grow exponentially at rate \(\lambda\) to translate epoch counts into real time \(t\) prove   Theorem \ref{thm.convergencyne} in the synchronous case.  

\textbf{Step 4: Asynchronous updates.} When players update at different times (see 
Condition \ref{condition:tauk_exp}), we partition the timeline into synchronous and asynchronous phases, and analyze them separately. Synchronous phases are intervals during which all players have completed the same number of updates; the remaining time constitutes asynchronous phases.

Asynchronous updates introduce several additional technical complications in our convergence analysis: (i) the estimation error decompositions contain extra terms from asynchronous phases that must be bounded separately; (ii) the epoch index \(k_m(t)\) in the expression \((F^{(k(t))}, f^{(k(t))})\) may differ across players; and (iii) random asynchronous shifts introduce an additional source of randomness, making update times random as well. 
Condition \ref{condition:tauk_exp} ensures 
the asynchronous time in epoch \(k\) is \(O(k)\) and that \(k/\tau^{(k)}_m\) is summable, which allows us to control these extra terms and to replace individual indices by a common index when deriving rates. When the shifts are independent of the other randomness, we first condition on the \(\sigma\)-algebra generated by all shifts, carry out the analysis under this conditioning, and then remove the conditioning using the uniform boundedness of the shifts. By carefully handling these aspects, the convergence analysis in Steps 1–3 carries over to the asynchronous setting.

\section{Conclusion}\label{sec.conclusion}

In this paper, we provide the first convergence analysis for a radically uncoupled learning procedure in LQ stochastic games. We demonstrate that the last-iterate convergence rate is fundamentally driven by the interplay of three factors: the decay of the estimation error, the decay of the exploration noise, and 
the stability of the Nash equilibrium. Furthermore, we apply this framework to a dynamic Cournot competition model with sticky prices, illustrating how varying levels of price stickiness alter market trajectories during transitions from emerging to mature phases within a highly restrictive informational environment.

Several directions merit further investigation. First, it is of interest to rigorously justify the accelerated convergence induced by publicly revealing aggregate market output, as observed in our numerical experiments.  Second, extending our analysis to other stochastic games with infinite state-action spaces would broaden its applicability and deepen the theoretical understanding of radically uncoupled learning in complex dynamic environments. Third, 
 studying asymmetric information settings, where players possess heterogeneous structural knowledge, may provide further economic insights. Finally, applying uncoupled learning procedures to other dynamic economic and operational models remains a promising direction for future research.

 \newpage
\setlength{\bibsep}{0ex}
\begin{spacing}{1.1}
\bibliographystyle{apalike}
\bibliography{ref}
\end{spacing}

\newpage
\appendix

\section*{Additional Notation.}

 We use $\vee$ and $\wedge$ to denote maximum and minimum, respectively: for real numbers $a,b$, $a\vee b = \max\{a,b\}$ and $a\wedge b = \min\{a,b\}$. The relation $a_k \asymp b_k$ indicates that there exist positive constants $c, C$ such that $c b_k \le a_k \le C b_k$ for all sufficiently large $k$. The unit sphere in \(\mathbb{R}^n\) is denoted by \(\mathbb{S}^{n-1}\), i.e., the set of vectors \(v \in \mathbb{R}^n\) with \(\|v\|_2 = 1\). Throughout this appendix, \(0 \le C < \infty\) denotes a generic constant depending on model and algorithm parameters; its value may change from line to line. When a constant associated with a quantity \(X\) needs to be explicitly referenced, we use \(\overline{C}_X\), \(\underline{C}_X\), and \(L_X\) to denote upper bounds, lower bounds, and Lipschitz constants, respectively. All these constants are independent of the time index \(t\), the epoch index \(k\), and the probability bound \(\delta\).

\section{Proof of Theorem \ref{thm.convergencyne}}\label{sec.proofconvergence}
The proof is divided into several parts.  
Appendix \ref{appendix.RLSerror} bounds the error of  the RLS estimator for a representative player $m$ in each epoch.  
Appendix \ref{appendix.Fferror} bounds the deviation between the current strategy profile and the collective best response to the previous profile, i.e., \(\|(F^{(k+1)},f^{(k+1)}) - \Psi((F^{(k)},f^{(k)}))\|_{\mu,1,2}\).  
Based on these results, Appendix \ref{appendix.proofThm1}  
completes the proof of Theorem \ref{thm.convergencyne}. 
Appendix \ref{appendix.additionalproofs} presents the proofs of the auxiliary lemmas in Appendix \ref{appendix.RLSerror}. All analysis is conducted under Assumptions \ref{ass.def.bestresponsemap}–\ref{ass.uniform.Phimcontinuity} and Condition \ref{condition:tauk_exp}. 

In every run of Algorithm~\ref{algorithm.indepedentlearning.Mplayer} there are three independent sources of randomness: the state noise \(\{\omega_t\}_{t=0}^\infty\), the exploration noise \(\{v_t^m\}_{t=0}^\infty\), and the asynchronous shifts \(\{\tilde{\tau}_m^{(k)}\}_{k=1}^\infty\). We first condition on the \(\sigma\)-algebra \(\mathcal{T} = \sigma(\{\tilde{\tau}_m^{(k)} : 1\le m\le M,\; k\ge 1\})\) generated by all shifts and carry out the subsequent analysis under this conditioning. Conditioning on \(\mathcal{T}\) turns each epoch length \(\tau_m^{(k)}\) into a deterministic constant satisfying  
$\tau^{(k)} \le \tau_m^{(k)} \le \tau^{(k)}+\bar{\tau}$
for every \(m\) and \(k\); all remaining randomness then comes only from the state and exploration noises.  
Throughout the rest of this section up to the end of the derivation of Theorem~\ref{thm.convergencyne}, every statement is understood under the conditional probability \(\mathbb{P}(\cdot\mid\mathcal{T})\), and \(\mathbb{E}\) denotes the corresponding conditional expectation.  
For brevity we keep the notation \(\mathbb{P}\) and \(\mathbb{E}\) without explicit conditioning.  
All constants that appear in the bounds are deterministic and depend only on the parameters of Condition~\ref{condition:tauk_exp}; consequently they are uniform over all realizations of \(\mathcal{T}\).  
The unconditional statements of the theorem are recovered in Step~4 of Appendix~\ref{appendix.proofThm1} by taking expectation over \(\mathcal{T}\) via the law of total probability.

\subsection{Estimation error of \((\hat{\Theta}^{(k)}_m,\hat{\Sigma}^{(k)}_m)\)}\label{appendix.RLSerror}

In this section we fix a representative player~\(m\) and bound the errors of her
\(k\)-th RLS estimates \(\hat{\Theta}^{(k)}_m\) and \(\hat{\Sigma}^{(k)}_m\) relative to
the effective parameters \((\Theta_m^{(k),*},\Sigma_m^{(k),*})\) defined in~\eqref{def.Thetamkstar}.
The analysis for other players is identical.
The main result is the following proposition.

\begin{proposition}\label{prop.convergencyTheta}
Suppose Assumptions~\ref{ass.stationary_x} and~\ref{ass.uniform.Phimcontinuity} and
Condition~\ref{condition:tauk_exp} hold.  For any \(\delta\in(0,1)\) there exist
\(K_m(\delta)=O(\log\log(1/\delta))\) and a constant \(C<\infty\) such that
\(\mathbb{P}(E_m(\delta))\ge 1-\delta\), where
\[
E_m(\delta)=\Bigl\{\forall k\ge K_m(\delta):\;
\max\bigl\{\|\hat\Theta_m^{(k)}-\Theta_m^{(k),*}\|,\;
            \|\hat\Sigma_m^{(k)}-\Sigma_m^{(k),*}\|\bigr\}
      \le C\,\eta_m^{(k)}(\delta)\Bigr\},
\]
with \(\eta_m^{(k)}(\delta)\) defined in \eqref{def.etak}.
\end{proposition}

The proof of Proposition \ref{prop.convergencyTheta} requires 
addressing several key difficulties
beyond the analysis for   the classical  single‑player LQ problem 
\citep{simchowitz2020naive}.
First, due to the presence of constant offsets in the dynamics and linear cost terms, the regression
vector \(z_t^{m}\) contains a non‑excited direction (the constant~\(1\)), which requires
extra care when establishing the growth rate of the information matrix.
Second, the asynchronous periods introduce additional terms in the decomposition of
\(\hat{\Theta}_{m}^{(k)}-\Theta_{m}^{(k),*}\), which must be bounded separately.
Third, unlike the uncontrolled LQ setting where only the drift parameter error is
estimated, here we must also bound the covariance estimation error
\(\|\hat{\Sigma}^{(k)}_{m} - \Sigma_{m}^{(k),*}\|\).

We introduce several auxiliary lemmas to handle these three difficulties. To prepare for their statements, we first allocate the failure probability and
rewrite the state dynamics from player \(m\)'s viewpoint.

\paragraph{Probability allocation and dynamic characterization.}
For a given confidence \(\delta\in(0,1)\) set
\(\delta_k = 6\delta/[\pi^{2}(k+1)^{2}]\); then \(\sum_{k=0}^{\infty}\delta_k=\delta\).
All high‑probability statements that follow are first established for a \emph{single}
epoch~\(k\) with failure probability proportional to \(\delta_k\).
Summing these bounds over \(k\) via a union bound will yield the ``for all large~\(k\)''
statement in Proposition~\ref{prop.convergencyTheta}.

Under Algorithm~\ref{algorithm.indepedentlearning.Mplayer}, the state dynamics~\eqref{state}
can be rewritten from player~\(m\)'s viewpoint as
\begin{equation}\label{def.tildeomega}
x_{t+1} = (\Theta_{m,t}^{*})^{\top} z^{m}_t + \tilde{w}^{m}_t, \qquad
z^{m}_t = [1,\; x_t^{\top},\; (u_t^{m,(k_m(t))})^{\top}]^{\top},
\end{equation}
where \(\tilde{w}^{m}_t\) is independent sub‑Gaussian with mean zero and
covariance matrix \(\Sigma_{m,t}^{*}\), and
\begin{align*}\label{def.Thetamt}
\Theta_{m,t}^{*} = \Big[ A_0 + \sum_{j\neq m} B_j f_j^{(k_j(t))},\; A - \sum_{j\neq m} B_j F_j^{(k_j(t))},\; B_{m} \Big]^{\top}, \quad 
\Sigma_{m,t}^{*} = \Sigma_{\omega} + \sum_{j\neq m} (\alpha_j^{(k_j(t))})^2 B_j B_j^{\top}.
\end{align*}

\paragraph*{Auxiliary lemmas.}
Lemmas~\ref{lem:V_eigen}-\ref{lem:cov_noise} below address the three difficulties
mentioned above.  They employ two‑sided bounds on sub‑Gaussian matrices
(Theorem~4.6.1 of~\cite{vershynin2018high}), the block martingale small‑ball
(BMSB) condition~\citep{simchowitz2018learning}, the self‑normalised martingale
inequality (Lemma~E.2 of~\cite{simchowitz2020naive} or Theorem~14.7 of~\cite{de2009self}), and Rayleigh quotient
techniques; together they lead to the polylogarithmic threshold
\(O(\log\log(1/\delta))\).  Detailed proofs of Lemmas~\ref{lem:V_eigen}-\ref{lem:cov_noise} are deferred to
Appendix~\ref{appendix.additionalproofs}.  Lemma~\ref{lem:noise} is a direct
consequence of Lemma~E.2 of~\cite{simchowitz2020naive}.

\begin{lemma}[Information matrix]\label{lem:V_eigen}
There exist constants \(\underline{C}_V,\overline{C}_V>0\) and
\(K_{V,m}(\delta)=O(\log\log(1/\delta))\) such that for every \(k\ge K_{V,m}(\delta)\), it holds for some event \(E_{V,k}\) with \(\mathbb{P}(E_{V,k})\ge 1-\delta_k/4\) that
\[
\sigma_{\min}\!\bigl(V_m^{(k)}\bigr)\ge \underline{C}_V\,\tau_m^{(k)}(\alpha_m^{(k-1)})^{2},
\quad
\|V_m^{(k)}\|\le
\overline{C}_V\tau_m^{(k)} ,
\quad \det(V_m^{(k)})\le
(\overline{C}_V\tau_m^{(k)})^{1+n+d_m}.
\]
\end{lemma}

\begin{lemma}[Asynchronous bias]\label{lemma.diffTheta1}
There exists \(K_{A,m}(\delta)=O(\log\log(1/\delta))\) such that for all
\(k\ge K_{A,m}(\delta)\), it holds for some event \(E_{A,k}\) with
\(\mathbb{P}(E_{A,k})\ge 1-\delta_k/4\) that
\[
\max\left\{
\left\|\sum_{t=\bar\tau_m^{(k-1)}}^{\bar\tau_m^{(k)}-1}
      z_t^{m}(z_t^{m})^{\top}(\Theta_{m,t}^{*}-\Theta_{m}^{(k),*})^{\top}\right\|,\;
\left\|\sum_{t=\bar\tau_m^{(k-1)}}^{\bar\tau_m^{(k)}-1}
      (\Theta_{m,t}^{*}-\Theta_m^{(k),*})^{\top} z_t^{m}(\tilde{w}_t^{m})^{\top}\right\|
\right\}
\le C\bigl(k^{2}+\log(1/\delta)\bigr).
\]
\end{lemma}

\begin{lemma}[Noise covariance]\label{lem:cov_noise}
There exists \(K_{\mathrm{cov},m}(\delta)=O(\log\log(1/\delta))\) such that for every
\(k\ge K_{\mathrm{cov},m}(\delta)\), it holds for some event \(E_{\mathrm{cov},k}\) with
\(\mathbb{P}(E_{\mathrm{cov},k})\ge 1-\delta_k/4\) that
\[
\left\|\frac1{\tau_m^{(k)}}\!\sum_{t=\bar\tau_m^{(k-1)}}^{\bar\tau_m^{(k)}-1}
      \bigl(\tilde w_t^{m}(\tilde w_t^{m})^{\top}-\Sigma_{m}^{(k),*}\bigr)\right\|
\le C\eta_m^{(k)}(\delta).
\]
\end{lemma}

\begin{lemma}[Self‑normalised noise bounds]\label{lem:noise}
For each \(k\ge1\), it holds for some event \(E_{\omega,k}\) with
\(\mathbb{P}(E_{\omega,k})\ge 1-\delta_k/4\) that
\begin{align}
\left\|(V_m^{(k)})^{-1}\!\sum_{t=\bar\tau_m^{(k-1)}}^{\bar\tau_m^{(k)}-1}
      z_t^{m}(\tilde w_t^{m})^{\top}\right\|
&\le C\,\sqrt{\frac{\log\det(V_m^{(k)})+\log(1/\delta_k)}
                 {\sigma_{\min}(V_m^{(k)})}}, \label{eq:noise_selfnorm}\\[4pt]
\left\|\frac1{\tau_m^{(k)}}\!\sum_{t=\bar\tau_m^{(k-1)}}^{\bar\tau_m^{(k)}-1}
      z_t^{m}(\tilde w_t^{m})^{\top}\right\|
&\le C\,\sqrt{\frac{(\log\det(V_m^{(k)})+\log(1/\delta_k))\|V_m^{(k)}\|}
                 {(\tau_m^{(k)})^{2}}}. \label{eq:noise_av}
\end{align}
\end{lemma}

For the rest of this section we set
\(K_{\mathrm{joint},m}(\delta):=\max\{K_{V,m},K_{A,m},K_{\mathrm{cov},m}\}=O(\log\log(1/\delta))\)
and define, for each \(k\ge K_{\mathrm{joint},m}(\delta)\), the joint good event
\(
E_k = E_{V,k}\cap E_{A,k}\cap E_{\omega,k}\cap E_{\mathrm{cov},k}.
\)
By the union bound, \(\mathbb{P}(E_k)\ge 1-\delta_k\).
We now show that on \(E_k\) both the drift parameter error
\(\|\hat\Theta_m^{(k)}-\Theta_m^{(k),*}\|\) and the covariance error
\(\|\hat\Sigma_m^{(k)}-\Sigma_m^{(k),*}\|\) are bounded by
\(C\eta_m^{(k)}(\delta)\).

\paragraph*{Error bound for \(\hat\Theta^{(k)}_m\).}
By the definitions of \(\hat{\Theta}_{m}^{(k)}\) in \eqref{def.hatTheta}, \(\Theta_{m}^{(k),*}\) in \eqref{def.Thetamkstar}, and \eqref{def.tildeomega}, \begin{eqnarray}\label{decomposition.diffTheta}
\nonumber \hat{\Theta}_{m}^{(k)} - \Theta_{m}^{(k),*} 
&=&\underbrace{- \beta_m (V_m^{(k)})^{-1} \Theta_{m}^{(k),*}}_{I_1}+ \underbrace{(V_m^{(k)})^{-1} \sum_{t=\bar{\tau}^{(k-1)}_m}^{\bar{\tau}^{(k)}_m-1} z^m_t (z^m_t)^{\top} (\Theta_{m,t}^{*} - \Theta_{m}^{(k), *})^{\top}}_{I_2}  \\
&&+ \underbrace{(V_m^{(k)})^{-1} \sum_{t=\bar{\tau}^{(k-1)}_m}^{\bar{\tau}^{(k)}_m-1} z^m_t (\tilde{w}^m_t)^{\top}}_{I_3}.
\end{eqnarray}

We analyze each term on 
 \(E_k\):
\begin{itemize}
\item  By the eigenvalue lower bound in Lemma~\ref{lem:V_eigen},
      \(\|I_1\|\le \frac{\beta_m\|\Theta_m^{(k),*}\|}{\sigma_{\min}(V_m^{(k)})}
      \le C/[\tau_m^{(k)}(\alpha_m^{(k-1)})^{2}]\le C\eta_m^{(k)}(\delta)\).
\item By Lemma~\ref{lemma.diffTheta1} and the eigenvalue bound,
      \(\|I_2\|\le \frac{C(k^{2}+\log(1/\delta))}{\sigma_{\min}(V_m^{(k)})}
      \le C\frac{k^{2}+\log(1/\delta)}{\tau_m^{(k)}(\alpha_m^{(k-1)})^{2}}\).
      Because \(\tau_m^{(k)}(\alpha_m^{(k-1)})^{2}\asymp\lambda^{(1-2\nu_m)k}\)
      grows exponentially with exponent \(1-2\nu_m>\frac12-\nu_m\), while the
      numerator is polynomial in \(k\) and \(\log(1/\delta)\), there exists
      \(K_{\Theta,m}(\delta)=O(\log\log(1/\delta))\) such that for all
      \(k\ge K'_{\Theta,m}(\delta)\) this term is bounded by
      \(C\eta_m^{(k)}(\delta)\).
\item Substituting the eigenvalue bounds of Lemma~\ref{lem:V_eigen} into
      \eqref{eq:noise_selfnorm} gives \(\|I_3\|\le C\eta_m^{(k)}(\delta)\).
\end{itemize}
Hence, for every \(k\ge \max\{K_{\mathrm{joint},m}(\delta),K_{\Theta,m}(\delta)\}\),
on the event \(E_k\),
\begin{equation}\label{bound.Theta}
\|\hat\Theta_m^{(k)}-\Theta_m^{(k),*}\|\le C\eta_m^{(k)}(\delta).
\end{equation}

\paragraph*{Error bound for \(\hat\Sigma^{(k)}_m\).}
By the definitions of \(\hat{\Sigma}_{m}^{(k)}\) in \eqref{def.hatSigma}, \(\Sigma_{m}^{(k),*}\) in \eqref{def.Thetamkstar}, and \eqref{def.tildeomega}, 
\begin{equation}\label{decomposition.hatSigma}
        \hat{\Sigma}^{(k)}_{m}-\Sigma^{(k),*}_{m}= \Lambda_1 + \Lambda_2 + \Lambda_3+\Lambda_4 + \Lambda_5+\Lambda_6+(\Lambda_4 + \Lambda_5+\Lambda_6)^\top,
\end{equation}
where
\begin{equation*}
\Lambda_1 = (\Theta_m^{(k),*} - \hat{\Theta}_m^{(k)})^\top\left(\frac{1}{\tau_m^{(k)}} \sum\limits_{t=\bar{\tau}^{(k-1)}_m}^{\bar{\tau}^{(k)}_m-1}  z_t^m (z^m_t)^{\top}\right) (\Theta_m^{(k),*} - \hat{\Theta}_m^{(k)}), \quad \Lambda_2 = \frac{1}{\tau_m^{(k)}} \sum\limits_{t=\bar{\tau}^{(k-1)}_m}^{\bar{\tau}^{(k)}_m-1} \tilde{w}_t^m (\tilde{w}_t^m)^\top-\Sigma^{(k),*}_{m},
\end{equation*}
\begin{equation*}
\Lambda_3 = \frac{1}{\tau_m^{(k)}} \sum\limits_{t=\bar{\tau}^{(k-1)}_m}^{\bar{\tau}^{(k)}_m-1} (\Theta_{m,t}^* - \Theta_m^{(k),*})^\top z_t^m (z^m_t)^{\top} (\Theta_{m,t}^* - \Theta_m^{(k),*}),
\end{equation*}
\begin{equation*}
\Lambda_4 = \frac{1}{\tau_m^{(k)}} \sum\limits_{t=\bar{\tau}^{(k-1)}_m}^{\bar{\tau}^{(k)}_m-1} (\Theta_{m,t}^* - \Theta_m^{(k),*})^\top z_t^m (z^m_t)^{\top} (\Theta_m^{(k),*} - \hat{\Theta}_m^{(k)}),
\end{equation*}
\begin{equation*}
\Lambda_5 = \frac{1}{\tau_m^{(k)}} \sum\limits_{t=\bar{\tau}^{(k-1)}_m}^{\bar{\tau}^{(k)}_m-1} (\Theta_{m,t}^* - \Theta_m^{(k),*})^\top z_t^m (\tilde{w}_t^m)^\top, \quad \Lambda_6 = \frac{1}{\tau_m^{(k)}} \sum\limits_{t=\bar{\tau}^{(k-1)}_m}^{\bar{\tau}^{(k)}_m-1} (\Theta_m^{(k),*} - \hat{\Theta}_m^{(k)})^\top z_t^m (\tilde{w}_t^m)^\top.
\end{equation*}

We now bound the \(\Lambda_i\) on \(E_k\).
\begin{itemize}
\item \(\Lambda_2\) is directly controlled by Lemma~\ref{lem:cov_noise}:
      \(\|\Lambda_2\|\le C\eta_m^{(k)}(\delta)\).

\item Asynchronous terms \(\Lambda_3,\Lambda_4,\Lambda_5\). 
      Using Lemma~\ref{lemma.diffTheta1} and the bound
      \(\|\Theta_{m,t}^* - \Theta_m^{(k),*}\|\le C\) and \(\|\hat\Theta_m^{(k)}-\Theta_m^{(k),*}\|\le C\eta_m^{(k)}(\delta)\),
      we obtain
\begin{equation}\label{bound.345}
    \|\Lambda_3\|+\|\Lambda_5\| \le \frac{C(k^{2}+\log(1/\delta))}{\tau_m^{(k)}}, \quad \|\Lambda_4\| \le \frac{C(k^{2}+\log(1/\delta))}{\tau_m^{(k)}}\eta_m^{(k)}(\delta).
\end{equation}

\item Remaining terms \(\Lambda_1,\Lambda_6\).
On \(E_{V,k}\) the empirical moment satisfies
\(\left\|\frac1{\tau_m^{(k)}}\sum\limits_{t=\bar{\tau}^{(k-1)}_m}^{\bar{\tau}^{(k)}_m-1}z_t^{m}(z_t^{m})^{\top}\right\|
 \le C\). Together with the cross term bound  in Lemma \ref{lem:noise} and \(\|\hat\Theta_m^{(k)}-\Theta_m^{(k),*}\|\le C\eta_m^{(k)}(\delta)\) gives that 
\begin{equation}\label{bound.16}
      \|\Lambda_1\|
      \le C
         \bigl(\eta_m^{(k)}(\delta)\bigr)^{2}, \quad 
      \|\Lambda_6\|
      \le C\sqrt{\frac{\log(\tau_m^{(k)})+\log(1/\delta)}{\tau_m^{(k)}}}\eta_m^{(k)}(\delta).
\end{equation}

\end{itemize}
As with \(I_2\) in the analysis of \(\hat\Theta^{(k)}_m\), the exponential growth of the denominator guarantees the
      existence of \(K_{\Sigma,m}(\delta)=O(\log\log(1/\delta))\) such that
      all error bounds in \eqref{bound.345} and \eqref{bound.16} are bounded by \(C\eta_m^{(k)}(\delta)\) for all
      \(k\ge K_{\Sigma,m}(\delta)\). 
Thus,  \(k\ge \max\{K_{\mathrm{joint},m}(\delta),K_{\Sigma,m}(\delta)\}\), on the event $E_k$, we have 
\begin{equation}\label{bound.Sigma}
    \|\hat\Sigma_m^{(k)}-\Sigma_m^{(k),*}\|\le C\eta_m^{(k)}(\delta).
\end{equation}

\paragraph*{Proof of Proposition~\ref{prop.convergencyTheta}.}
Take \(K_m(\delta)=\max\bigl\{K_{\mathrm{joint},m}(\delta),\,
K_{\Theta,m}(\delta),\,K_{\Sigma,m}(\delta)\bigr\}
=O(\log\log(1/\delta))\). Define the event
\(E_m(\delta)=\bigcap_{k\ge K_m(\delta)} E_k\).
By the union bound,
\(\mathbb{P}(E_m(\delta)^{c})\le \sum_{k\ge K_m(\delta)}\mathbb{P}(E_k^{c})
 \le \sum_{k\ge K_m(\delta)}\delta_k \le \delta\).
On \(E_m(\delta)\), for all \(k\ge K_m(\delta)\) the estimates \eqref{bound.Theta} and \eqref{bound.Sigma} hold
simultaneously, giving
\(\max\{\|\hat\Theta_m^{(k)}-\Theta_m^{(k),*}\|,\,
        \|\hat\Sigma_m^{(k)}-\Sigma_m^{(k),*}\|\}
\le C\eta_m^{(k)}(\delta)\).

\subsection{Strategy Deviation of $(F^{(k)},f^{(k)})$}\label{appendix.Fferror}

We first translate the estimation error bounds obtained in Appendix~\ref{appendix.RLSerror} into bounds on the accepted parameters  
\((\tilde{\Theta}_{m}^{(k)},\tilde{\Sigma}_{m}^{(k)},\alpha_m^{(k)})\) relative to \((\Theta_{m}^{(k),*},\Sigma_{\omega},0)\).

\begin{lemma}\label{lemma.convergencyTheta}
Suppose Assumptions~\ref{ass.stationary_x} and~\ref{ass.uniform.Phimcontinuity} and Condition~\ref{condition:tauk_exp} hold. For any $\delta \in (0,1)$, there exist $K'_m(\delta)=O(\log\log(1/\delta))$ and a constant $C < \infty$ such that $\mathbb{P}(E'_m(\delta)) \ge 1-\delta$, where
\begin{equation*}
E'_m(\delta) = \Bigl\{ \forall k \ge K'_m(\delta),\; \max\bigl\{\|\tilde{\Theta}_{m}^{(k)} - \Theta_{m}^{(k),*}\|,\;\|\tilde{\Sigma}_{m}^{(k)} - \Sigma_{\omega}\|,\;\alpha_m^{(k)}\bigr\} \le C\bigl(\eta_m^{(k)}(\delta) \vee \max_{1\le j\le M} \{\alpha_j^{(k-1)}\}\bigr) \Bigr\}.
\end{equation*}
\end{lemma}

\begin{proof}{Proof}
Let $E_m(\delta)$ be the event from Proposition~\ref{prop.convergencyTheta}.
On $E_m(\delta)$, for all $k\ge K_m(\delta)$ we have
$\|\hat{\Theta}_m^{(k)}-\Theta_m^{(k),*}\|\le C\eta_m^{(k)}(\delta)$ and
$\|\hat{\Sigma}_m^{(k)}-\Sigma_m^{(k),*}\|\le C\eta_m^{(k)}(\delta)$.
By definition of $\Sigma_{m}^{(k),*}$ in~\eqref{def.Thetamkstar},
$\|\Sigma_{m}^{(k),*}-\Sigma_{\omega}\| \le C \max\limits_{j\neq m}(\alpha_j^{(k-1)})^2$, so on the same event
\begin{equation}\label{bound.ThetaSigma}
\max\{\|\hat{\Theta}_{m}^{(k)} - {\Theta}_{m}^{(k),*}\|,\;
        \|\hat{\Sigma}_{m}^{(k)} - \Sigma_{\omega}\|\}
\le C\bigl(\eta_m^{(k)}(\delta) \vee \max\limits_{j\neq m} (\alpha_j^{(k-1)})^2\bigr).
\end{equation}
From the definition of $\eta_m^{(k)}(\delta)$ in \eqref{def.etak} and the fact that
$\alpha_j^{(k)}\asymp\lambda^{-\nu_j k}$, there exists
$K''_m(\delta)=O(\log\log(1/\delta))$ such that for all
$k\ge K''_m(\delta)$ the right‑hand side of \eqref{bound.ThetaSigma} is
smaller than the constant $\delta_m$ from
Assumption~\ref{ass.uniform.Phimcontinuity}.
Take $K'_m(\delta)=\max\{K_m(\delta),K''_m(\delta)\}$; then
$K'_m(\delta)=O(\log\log(1/\delta))$.
Part~1) of Assumption~\ref{ass.uniform.Phimcontinuity} then guarantees
$\tilde{\Theta}_m^{(k)}=\hat{\Theta}_m^{(k)}$ and
$\tilde{\Sigma}_m^{(k)}=\hat{\Sigma}_m^{(k)}$ for all $k\ge K'_m(\delta)$.

Finally, the deterministic bound
$\max\{\alpha_m^{(k)},\,\max\limits_{j\neq m}(\alpha_j^{(k-1)})^2\}
 \le C\max\limits_{1\le j\le M}\{\alpha_j^{(k-1)}\}$
holds for all $k$.
Let $E'_m(\delta)$ be the event $E_m(\delta)$, but with the statement
required only for $k\ge K'_m(\delta)$.
Since $K'_m(\delta)\ge K_m(\delta)$, we have
$\mathbb{P}(E'_m(\delta))\ge \mathbb{P}(E_m(\delta))\ge 1-\delta$.
\end{proof}

Now we consider the strategy profile $(F^{(k)},f^{(k)})$, by which we mean that every player uses the strategy $(F_m^{(k)},f_m^{(k)})$ obtained after her $k$-th update.  
Due to the presence of asynchronous shifts, this profile may differ from the profile $(F^{(k(t))},f^{(k(t))})$ that appears in Theorem~\ref{thm.convergencyne}; the two profiles coincide only during the synchronous phases.  
Here we first derive the deviation of $(F^{(k+1)},f^{(k+1)})$ from the collective best response $\Psi((F^{(k)},f^{(k)}))$ based on the error bound for the accepted parameters established above.  
This strategy deviation serves as an intermediate result for the proof of Theorem~\ref{thm.convergencyne}; the actual profile $(F^{(k(t))},f^{(k(t))})$ for arbitrary $t$ will be handled in the next section.

\begin{lemma}\label{lemma.LipFf}
Suppose Assumptions~\ref{ass.def.bestresponsemap}-\ref{ass.uniform.Phimcontinuity} and Condition~\ref{condition:tauk_exp} hold. For any $\delta \in (0,1)$, there exist $K(\delta)=O(\log\log(1/\delta))$ and $L_{Ff} < \infty$ such that $\mathbb{P}(E(\delta)) \ge 1-\delta$, where
\begin{equation*}
E(\delta) = \left\{ \forall k \ge K(\delta),\; \|(F^{(k+1)},f^{(k+1)}) - \Psi((F^{(k)},f^{(k)}))\|_{\mu,1,2} \le L_{Ff}\,\xi_{k+1}(\delta) \right\},
\end{equation*}
with $\xi_k(\delta)$ defined in \eqref{def.xidelta} and the norm $\|\cdot\|_{\mu,1,2}$ defined in Assumption~\ref{ass.Psi_contractive_selfmap}.
\end{lemma}

\begin{proof}{Proof}
Set $K(\delta) = \max\limits_m K'_m(\delta/M)$, where $K'_m(\cdot)$ is from Lemma~\ref{lemma.convergencyTheta}; then $K(\delta)=O(\log\log(1/\delta))$.
By definition, $\eta_m^{(k)}(\delta) \vee \max\limits_{1\le j\le M}\{\alpha_j^{(k-1)}\} \le C\,\xi_k(\delta)$ for all $k\ge1$.
Since $\Theta_{m}^{(k+1),*} = \Theta^{(F^{(k)}_{-m},f^{(k)}_{-m})}_m$ by~\eqref{def.Thetamkstar},
the equivalence in Assumption~\ref{ass.def.bestresponsemap} gives
\[
\psi(F^{(k)}_{-m},f^{(k)}_{-m}) = \Phi_m(\Theta_{m}^{(k+1),*},\Sigma_{\omega},x,0),\qquad
\varphi(F^{(k)}_{-m},f^{(k)}_{-m}) = \phi_m(\Theta_{m}^{(k+1),*},\Sigma_{\omega},x,0).
\]
On the event $\bigcap_{m=1}^M E'_m(\delta/M)$, for all $k \ge K(\delta)$,
\begin{align*}
\|(F^{(k+1)},f^{(k+1)}) - \Psi((F^{(k)},f^{(k)}))\|_{\mu,1,2}
&= \mu\sum_{m=1}^{M}\|F_m^{(k+1)} - \Phi_m(\Theta_{m}^{(k+1), *},\Sigma_{\omega},x,0)\|\\
&\quad + \sum_{m=1}^{M}\|f_m^{(k+1)} - \phi_m(\Theta_{m}^{(k+1), *},\Sigma_{\omega},x,0)\|_2 \\
&\le C\Bigl(\sum_{m=1}^{M} L_m\Bigr)(\mu+1)\,\xi_{k+1}(\delta),
\end{align*}
where we used Lemma~\ref{lemma.convergencyTheta} and Assumption~\ref{ass.uniform.Phimcontinuity}. Taking $L_{Ff} = C(\sum_m L_m)(\mu+1)$ gives $\bigcap_{m=1}^M E'_m(\delta/M) \subset E(\delta) $. The probability estimate follows from
$\mathbb{P}(E(\delta)) \ge \mathbb{P}(\bigcap_{m=1}^M E'_m(\delta/M)) \ge 1-\sum_{m=1}^M\delta/M \ge 1-\delta$.
\end{proof}

\subsection{Proof of Theorem~\ref{thm.convergencyne}}\label{appendix.proofThm1}

Let \((F^*,f^*)\) denote the unique feedback Nash equilibrium guaranteed by
Assumptions~\ref{ass.def.bestresponsemap}-\ref{ass.Psi_contractive_selfmap},
and let \((\mu,\zeta)\) be the corresponding constants.

\noindent\textbf{Step~1.} We first show that with high probability \((F^{(k)},f^{(k)})\)
converges to \((F^*,f^*)\).
Set \(a_k = \|(F^{(k)},f^{(k)}) - (F^*,f^*)\|_{\mu,1,2}\).
By Lemma~\ref{lemma.LipFf} and Assumption~\ref{ass.Psi_contractive_selfmap},
for any \(\delta\in(0,1)\) there exists \(K(\delta)=O(\log\log(1/\delta))\) such that,
with probability at least \(1-\delta\), for all \(k\ge K(\delta)\),
\[
a_{k+1}\le \zeta a_k + L_{Ff}\,\xi_{k+1}(\delta).
\]

Define the constants
\[
r_1 = \frac{\lambda^{-(\frac12-\max\limits_m\{\nu_m\})}}{\zeta},\qquad
r_2 = \frac{\lambda^{-\min\limits_m\{\nu_m\}}}{\zeta}.
\]
Let \(c_{K(\delta)}=2\sum_{m=1}^M(\kappa_m+\kappa'_m)\) and
\(c_{k+1}=\zeta c_k+L_{Ff}\,\xi_{k+1}(\delta)\).
Unfolding the recursion gives
\begin{equation*}\label{eq:ck_recursion}
c_k = \underbrace{\zeta^{k-K(\delta)}c_{K(\delta)}}_{\text{Term I (transient)}}
     + L_{Ff}\zeta^k\sum_{j=K(\delta)+1}^{k}
       \underbrace{\sqrt{j+\log(1/\delta)}\;r_1^{\,j}}_{\text{Term II (estimation error)}}
     + L_{Ff}\zeta^k\sum_{j=K(\delta)+1}^{k}
       \underbrace{r_2^{\,j}}_{\text{Term III (exploration noise)}} .
\end{equation*}

Term~I: since \(K(\delta)=O(\log\log(1/\delta))\) and \(\zeta<1\),
\(\zeta^{-K(\delta)}\le (\log(1/\delta))^{c_0}\) for some \(c_0>0\), hence
Term~I = \(O\bigl((\log(1/\delta))^{c_0}\,\zeta^{k}\bigr)\).

Term~II: write \(\sqrt{j+\log(1/\delta)}\le \sqrt{k+\log(1/\delta)}\) for all
\(j\le k\).  If \(r_1<1\) the series \(\sum_{j} r_1^{\,j}\) converges, giving
Term~II = \(O(\sqrt{k+\log(1/\delta)}\,\zeta^{k})\); if \(r_1=1\) then
Term~II = \(O(k\sqrt{k+\log(1/\delta)}\,\zeta^{k})\); if \(r_1>1\) the last term
dominates and Term~II = \(O(\sqrt{k+\log(1/\delta)}\,\lambda^{-(\frac12-\max\limits_m\{\nu_m\})k})\).
In all cases the dependence on \(\delta\) is at most \(O(\sqrt{\log(1/\delta)})\).

Term~III: this term does not involve \(\delta\).
If \(r_2<1\) it is \(O(\zeta^{k})\);
if \(r_2=1\) it is \(O(k\zeta^{k})\);
if \(r_2>1\) it is \(O(\lambda^{-\min\limits_m\{\nu_m\} k})\).

Since \(a_k\le c_k\) by induction, \(a_k\) enjoys the same bounds.
Consequently \(a_k\to0\) with probability at least \(1-\delta\), and for
\(k\ge K(\delta)\) and each player \(m\),
\begin{equation}\label{inequality.ak}
\mu\|F_m^{(k)}-F_m^*\|+\|f_m^{(k)}-f_m^*\|_2 \le a_k .
\end{equation}

\noindent\textbf{Step~2.}
We now translate the rates to calendar time \(t\). 
By Condition~\ref{condition:tauk_exp}, we have
\(\tau^{(k)} \le \tau_m^{(k)} \le \tau^{(k)} + \bar{\tau}\).
Summing over epochs gives constants \(0<c\le C<\infty\) such that
\(c\lambda^{k} \le \bar{\tau}_m^{(k)} \le C\lambda^{k}\) and
\(k_m(t)=\log_{\lambda}t+O(1)\) for all \(m\) and all large \(k\).
Choose \(T(\delta)\) so that \(k_m(t)\ge K(\delta)\) for all \(t\ge T(\delta)\)
and every \(m\); the threshold satisfies
\(T(\delta)=O\bigl((\log(1/\delta))^{c_1}\bigr)\) for some \(c_1>0\).

By~\eqref{inequality.ak}, for \(t\ge T(\delta)\),
\[
\|(F^{(k(t))},f^{(k(t))})-(F^*,f^*)\|_{\mu,1,2}
\le M\max_{1\le m\le M}\{a_{k_m(t)}\}.
\]
Substituting the three bounds from Step~1 and converting epoch indices to
time via
\(\zeta^{k_m(t)}\asymp t^{\ln\zeta/\ln\lambda}\),
\(\lambda^{-(\frac12-\max\limits_m\{\nu_m\})k_m(t)}\asymp t^{-(\frac12-\max\limits_m\{\nu_m\})}\),
\(\lambda^{-\min\limits_m\{\nu_m\}k_m(t)}\asymp t^{-\min\limits_m\{\nu_m\}}\), we obtain
\[
\|(F^{(k(t))},f^{(k(t))})-(F^*,f^*)\|_{\mu,1,2}
= O\bigl( \max\{ R_1(t,\delta),\; R_2(t,\delta),\; R_3(t) \} \bigr),
\]
where \(R_1,R_2,R_3\) are defined in \eqref{def.R123}.

\noindent\textbf{Step~3.}
The estimates in Steps~1-2 are derived under the conditional probability
\(\mathbb{P}(\cdot\mid\mathcal{T})\) and show that for any \(\delta\in(0,1)\)
there exists \(T(\delta)\) such that the claimed bounds hold for all
\(t\ge T(\delta)\) with \(\mathbb{P}(\cdot\mid\mathcal{T})\)-probability at
least \(1-\delta\).
To lift this to an unconditional statement, observe that all constants
appearing in the bounds depend only on model parameters and the parameters
in Condition~\ref{condition:tauk_exp}, hence are uniform over all
realizations of the shifts.  The thresholds in the auxiliary lemmas are chosen by making certain
tail probabilities smaller than a prescribed fraction of \(\delta\); because the tail conditions involve
only the deterministic lower bound \(\tau^{(k)}\le\tau_m^{(k)}\), the
resulting thresholds work simultaneously for all realizations of
\(\mathcal{T}\).  Taking expectation over \(\mathcal{T}\) yields the same
bounds under the unconditional probability \(\mathbb{P}\).

\noindent\textbf{Step~4 (Almost‑sure convergence with rate).}
From Steps~1-3, for any \(\delta\in(0,1)\) there exist constants
\(C<\infty\) and \(T(\delta)\in\mathbb N\) such that, with probability at
least \(1-\delta\), for all \(t\ge T(\delta)\),
\begin{equation}\label{eq:highprob_final}
\|(F^{(k(t))},f^{(k(t))})-(F^*,f^*)\|_{\mu,1,2}
\le C\max\{ R_1(t,\delta),\, R_2(t,\delta),\, R_3(t) \},
\end{equation}
where \(R_1,R_2,R_3\) are defined in Step~2.

For each \(j\in\mathbb N\) set \(\delta_j=2^{-j}\) and let
\(\mathcal E_j\) be the event on which the bound in
\eqref{eq:highprob_final} holds with \(\delta=\delta_j\) for all
\(t\ge T(\delta_j)\).  By construction \(\mathbb{P}(\mathcal E_j)\ge 1-2^{-j}\).
Since \(\sum_{j=1}^{\infty}\mathbb{P}(\mathcal E_j^{c})\le\sum_{j=1}^{\infty}2^{-j}<\infty\),
the Borel–Cantelli lemma implies that
\(\mathbb{P}\bigl(\bigcap_{J=1}^{\infty}\bigcup_{j=J}^{\infty}\mathcal E_j^{c}\bigr)=0\).
Equivalently, letting \(\mathcal E = \bigcup_{J=1}^{\infty}\bigcap_{j=J}^{\infty}\mathcal E_j\),
we have \(\mathbb{P}(\mathcal E)=1\).  On the almost‑sure event \(\mathcal E\),
there exists a (random) index \(j_0\) such that \(\mathcal E_j\) occurs for
every \(j\ge j_0\).

Set
\(
q = \min\!\Bigl\{-\tfrac{\ln\zeta}{\ln\lambda},\;
                    \tfrac12-\max\limits_m\{\nu_m\},\;
                    \min\limits_m\{\nu_m\}\Bigr\} > 0 .
\)
Now observe that, for any fixed \(j\),
\[
\frac{t^{q}}{(\log t)^{\frac32}\log\log t}\max\{ R_1(t,\delta_j),\, R_2(t,\delta_j),\, R_3(t) \} \to 0
\qquad\text{as } t\to\infty.
\]
Consequently, on the
almost‑sure event \(\mathcal E\), for any \(\varepsilon>0\) we can first choose \(j\)
large enough so that \(j\ge j_0\) and then choose \(t\) large enough so
that the right‑hand side of \eqref{eq:highprob_final} multiplied by
\(\frac{t^{q}}{(\log t)^{\frac32}\log\log t}\) is smaller than \(\varepsilon\).  Hence
\[
\lim_{t\to\infty} \frac{t^{q}}{(\log t)^{\frac32}\log\log t}\,
\|(F^{(k(t))},f^{(k(t))})-(F^*,f^*)\|_{\mu,1,2} = 0
\qquad\text{almost surely}.
\]
This shows that the strategy profile converges almost surely to the
feedback Nash equilibrium \((F^*,f^*)\) at the rate
\(O\left(t^{-q}(\log t)^{\frac32}\log\log t\right)\).

\subsection{Proofs of Auxiliary Lemmas in Appendix~\ref{appendix.RLSerror}}\label{appendix.additionalproofs}

\subsubsection{Proof of Lemma~\ref{lem:V_eigen}}

We first recall a purely algebraic estimate that will be used repeatedly.

\begin{lemma}\label{lemma.minieigenvalue}
    Let $H$ be a $d \times d$ symmetric positive semidefinite matrix. Suppose there exist matrices $\Gamma_1 \in \mathbb{R}^{d \times p}$ and $\Gamma_2 \in \mathbb{R}^{d \times q}$ with $p+q = d$ such that $[\Gamma_1 \ \Gamma_2] \in \mathbb{R}^{d \times d}$ is invertible and
    \begin{equation*}
    [\Gamma_1 \ \Gamma_2]^\top H [\Gamma_1 \ \Gamma_2] = \begin{bmatrix} H_0 & 0_{p \times q} \\ 0_{q \times p} & H_1 \end{bmatrix},
    \end{equation*}
    where $H_0 = \Gamma_1^\top H \Gamma_1 \in \mathbb{R}^{p \times p}$ and $H_1 = \Gamma_2^\top H \Gamma_2 \in \mathbb{R}^{q \times q}$ are symmetric positive definite. Then
    \begin{equation*}
    \sigma_{\min}(H) \ge \frac{\sigma_{\min}(H_0) \land \sigma_{\min}(H_1)}{4\bigl(\|\Gamma_1\|^2 \vee \|\Gamma_2\|^2\bigr)}.
    \end{equation*}
\end{lemma}
\begin{proof}{Proof}
    For any nonzero \(y\in\mathbb{R}^d\), write \(y=\Gamma_1u+\Gamma_2v\) uniquely with \(u\in\mathbb{R}^p\) and \(v\in\mathbb{R}^q\). Then \(y^\top Hy=u^\top H_0u+v^\top H_1v\ge \sigma_{\min}(H_0)\|u\|_2^2+\sigma_{\min}(H_1)\|v\|_2^2\), while \(\|y\|_2^2=\|\Gamma_1u+\Gamma_2v\|_2^2\le 4(\|\Gamma_1\|^2\vee\|\Gamma_2\|^2)(\|u\|_2^2\vee\|v\|_2^2)\). If \(\|u\|_2^2\ge\|v\|_2^2\), then \(y^\top Hy/\|y\|_2^2\ge \sigma_{\min}(H_0)/[4(\|\Gamma_1\|^2\vee\|\Gamma_2\|^2)]\); the other case is analogous. Taking the infimum over \(y\neq0\) gives the result. 
\end{proof}

\medskip
\noindent
\textbf{Upper bound on \(\|V_m^{(k)}\|\) and the determinant.}
From the definition \(z_t^m=[1,x_t^\top,(u_t^{m})^\top]^\top\) and the feedback
\(u_t^m = -F_m^{(k-1)}x_t + f_m^{(k-1)} + \alpha_m^{(k-1)}v_t^m\), we have
\begin{equation}\label{bound.z2}
\|z_t^m\|_2^2 \le C\bigl(1 + \|x_t\|_2^2 + \|v_t^m\|_2^2\bigr).
\end{equation}

For the exploration noise, since \(\{v_t^m\}\) are i.i.d.\ sub‑Gaussian,
Bernstein's inequality implies that for every \(k\ge K_{v2,m}=O(\log\log(1/\delta))\),
it holds for some event \(E_{v2,k}\) with \(\mathbb{P}(E_{v2,k})\ge 1-\delta_k/40\) that
\(\sum_{t=\bar\tau_m^{(k-1)}}^{\bar\tau_m^{(k)}-1} \|v_t^m\|_2^2 \le 2d_m\tau_m^{(k)}\).

For the state, we rewrite \eqref{def.tildeomega} as
\begin{equation}\label{state:Acl}
    x_{t+1} = A_t^0 + A^{\text{cl}}_t x_t + \tilde{w}_t,
\end{equation}
where \(A_t^0 = A_0 + \sum_{m=1}^{M} B_m f_m^{(k_m(t))}\),
\(A^{\text{cl}}_t = A - \sum_{m=1}^{M} B_m F_m^{(k_m(t))}\), and
\(\tilde{w}_t = \omega_t + \sum_{m=1}^{M} B_m \alpha^{(k_m(t))}_m v_t^m\).
Assumption~\ref{ass.stationary_x} implies \(\|A^{\text{cl}}_t\| \le 1-\gamma\) and
\(\|A_t^0\|_2 \le \overline{C}_{A,0}\) uniformly. Since the state process \(\{x_t\}_{t=0}^\infty\) is driven by sub-Gaussian innovations and the
dynamics are uniformly stable, it is uniformly sub-Gaussian; hence \(\|x_t\|_2^2\) is
uniformly sub-exponential for all \(t\). 
By Young's inequality,
\(\|x_{t+1}\|_2^2 \le (1-\frac{\gamma}{2})\|x_t\|_2^2 + C_{A} + C_{w}\|\tilde w_t\|_2^2\).
Iterating and summing over \(t=\bar\tau_m^{(k-1)},\dots,\bar\tau_m^{(k)}-1\) gives
\[
\sum_{t=\bar\tau_m^{(k-1)}}^{\bar\tau_m^{(k)}-1}\|x_t\|_2^2
\le \frac{2\|x_{\bar\tau_m^{(k-1)}}\|_2^2}{\gamma}
  + \frac{2C_{A}\tau_m^{(k)}}{\gamma}
  + \frac{2C_w}{\gamma}\sum_{j=\bar\tau_m^{(k-1)}}^{\bar\tau_m^{(k)}-2}
      b_{j,k}\,\|\tilde w_j\|_2^2,
\]
where \(0\le b_{j,k}=1-(1-\gamma)^{\bar\tau_m^{(k)}-1-j}\le 1\). 
The variable \(\|x_{\bar\tau_m^{(k-1)}}\|_2^2\) is uniformly sub-exponential and
independent of the future noises \(\{\tilde w_j\}_{j\ge \bar\tau_m^{(k-1)}}\);
moreover, \(\{\|\tilde w_j\|_2^2\}\) is an independent sequence of uniformly
sub-exponential random variables.  Hence, Bernstein's inequality for weighted
sums yields
\begin{equation}\label{bound.x2}
\mathbb{P}\!\left( \sum_{t=\bar\tau_m^{(k-1)}}^{\bar\tau_m^{(k)}-1}\|x_t\|_2^2> C_x \tau_m^{(k)} \right)
\le \exp\!\bigl(-C\tau_m^{(k)}\bigr).
\end{equation}
Hence, for every \(k\ge K_{x2,m}=O(\log\log(1/\delta))\), it holds for some event \(E_{x2,k}\) with \(\mathbb{P}(E_{x2,k})\ge 1-\delta_k/40\) that
\(
\sum_{t=\bar\tau_m^{(k-1)}}^{\bar\tau_m^{(k)}-1} \|x_t\|_2^2 \le C_x \tau_m^{(k)} .
\)
Set \(E_{z,k} = E_{v2,k}\cap E_{x2,k}\); then
\(\mathbb{P}(E_{z,k})\ge 1-\delta_k/20\).
For all \(k\ge K_{z,m}=\max\{K_{v2,m},K_{x2,m}\}\),
on \(E_{z,k}\), by \eqref{bound.z2}, we have
\[
\|V_m^{(k)}\|
\le \beta_m + \sum_{t=\bar\tau_m^{(k-1)}}^{\bar\tau_m^{(k)}-1}\|z_t^m\|_2^2
\le \overline{C}_V\,\tau_m^{(k)},
\qquad
\det(V_m^{(k)}) \le \|V_m^{(k)}\|^{1+n+d_m}
\le \bigl(\overline{C}_V\,\tau_m^{(k)}\bigr)^{1+n+d_m}.
\]

\medskip
\noindent
\textbf{Lower bound on \(\sigma_{\min}(V_m^{(k)})\).}
During the whole epoch \(k\) player \(m\) uses the fixed feedback
\((F_m^{(k-1)},f_m^{(k-1)})\), therefore the regression vector can be
written as
\[
z_t^{m}=G_{m,k}\bar{x}_t+\alpha_m^{(k-1)}J_m v_t^{m},\qquad
\bar{x}_t=\begin{bmatrix}1 & x_t\end{bmatrix}^{\top},
\]
where
\(
G_{m,k}=
\begin{bmatrix}
1 & 0_{1\times n}\\
0_{n\times1} & I_n\\
f_m^{(k-1)} & -F_m^{(k-1)}
\end{bmatrix}\) and 
\(J_m=
\begin{bmatrix}
0_{(1+n)\times d_m}\\
I_{d_m}
\end{bmatrix}.
\)
Both \(\|G_{m,k}\|\) and \(\|J_m\|\) are bounded by \(1+\kappa_m+\kappa'_m\).
Inserting this into the definition of \(V_m^{(k)}\) and normalising by
\(\tau_m^{(k)}(\alpha_m^{(k-1)})^{2}\) yields
\begin{align}\label{eq:V_decomp}
\nonumber \frac{V_m^{(k)}}{\tau_m^{(k)}(\alpha_m^{(k-1)})^{2}}
=&
(\alpha_m^{(k-1)})^{-2}G_{m,k}X_k G_{m,k}^{\top}
+J_m W_k J_m^{\top}
\\
&+(\alpha_m^{(k-1)})^{-1}\bigl(G_{m,k}C_k J_m^{\top}+J_m C_k^{\top}G_{m,k}^{\top}\bigr)
+\frac{\beta_m I}{\tau_m^{(k)}(\alpha_m^{(k-1)})^{2}},
\end{align}
with the empirical matrices
\begin{align*}
X_k=\frac{1}{\tau_m^{(k)}}\sum_{t=\bar\tau_m^{(k-1)}}^{\bar\tau_m^{(k)}-1}\bar{x}_t\bar{x}_t^{\top},\quad 
W_k=\frac{1}{\tau_m^{(k)}}\sum_{t=\bar\tau_m^{(k-1)}}^{\bar\tau_m^{(k)}-1}v_t^{m}(v_t^{m})^{\top},\quad 
C_k=\frac{1}{\tau_m^{(k)}}\sum_{t=\bar\tau_m^{(k-1)}}^{\bar\tau_m^{(k)}-1}\bar{x}_t(v_t^{m})^{\top}.
\end{align*}
We will bound each of these three blocks on a suitable high‑probability event.

\emph{Step~1: State covariance \(X_k\).}
Set
\(
\hat\mu_k=\frac{1}{\tau_m^{(k)}}\sum_{t=\bar\tau_m^{(k-1)}}^{\bar\tau_m^{(k)}-1}x_t\) and 
\(\hat\Gamma_k=\frac{1}{\tau_m^{(k)}}\sum_{t=\bar\tau_m^{(k-1)}}^{\bar\tau_m^{(k)}-1}(x_t-\hat\mu_k)(x_t-\hat\mu_k)^{\top}.
\)
The augmented covariance decomposes as
\begin{equation}\label{eq:Xk_block}
X_k=
\begin{bmatrix}
1 & \hat\mu_k^{\top}\\
\hat\mu_k & \hat\Gamma_k+\hat\mu_k\hat\mu_k^{\top}
\end{bmatrix}.
\end{equation}
Thus it suffices to lower‑bound \(\sigma_{\min}(\hat\Gamma_k)\) and to
upper‑bound \(\|\hat\mu_k\|_2\).

Recall \eqref{state:Acl}, the matrices \(A_t^{0},A_t^{\mathrm{cl}}\) are \(\mathcal F_t\)-measurable,
where \(\mathcal F_t\) denotes the information available at time \(t\) before
the noises \(\omega_t,\{v_t^{j}\}_{1\le j\le M}\) are drawn; by
Assumption~\ref{ass.stationary_x} they satisfy
\(\|A_t^{\mathrm{cl}}\|\le1-\gamma\) uniformly.
The effective noise \(\tilde w_t\) is independent of \(\mathcal F_t\),
conditionally sub‑Gaussian, and its covariance matrix is uniformly bounded below
by \(\Sigma_{\omega}\succ0\).

For any fixed direction \(q\in\mathbb S^{n-1}\) the scalar process
\(d_t=q^{\top}(x_{t+1}-x_t)\) can be written as
\(d_t = q^{\top}(A_t^{0}+(A_t^{\mathrm{cl}}-I)x_t) + q^{\top}\tilde w_t\).
The first term is \(\mathcal F_t\)-measurable, while the second is a
conditionally centred sub‑Gaussian random variable whose conditional
variance satisfies
\(\mathbb{E}[(q^{\top}\tilde w_t)^{2}\mid\mathcal F_t]
\ge\sigma_{\min}(\Sigma_{\omega})>0\).
By a standard Paley–Zygmund argument (Proposition~3.1 of
\cite{simchowitz2018learning}), this implies that the process
\(\{d_t\}\) satisfies the \((1,\nu,p)\)-block martingale small‑ball condition
of~\cite{simchowitz2018learning} with constants \(\nu,p>0\) that depend only
on the sub‑Gaussian norm of \(\tilde w_t\) and on \(\sigma_{\min}(\Sigma_{\omega})\).

Applying Proposition~2.5 of~\cite{simchowitz2018learning} to
\(\{d_t\}_{t=\bar\tau_m^{(k-1)}}^{\bar\tau_m^{(k)}-2}\) and using a net argument
(Corollary~4.2.11 of~\cite{vershynin2018high}) shows that for every
\(k\ge K_{\Gamma,m}=O(\log\log(1/\delta))\), it holds for some event
\(E_{\Gamma,k}\) with \(\mathbb{P}(E_{\Gamma,k})\ge1-\delta_k/20\) that
\[
\sigma_{\min}\!\left(
\frac{1}{\tau_m^{(k)}}\sum_{t=\bar\tau_m^{(k-1)}}^{\bar\tau_m^{(k)}-2}
(x_{t+1}-x_t)(x_{t+1}-x_t)^{\top}
\right)\ge \underline{C}_\Gamma>0.
\]

Now for any \(q\in\mathbb S^{n-1}\) set
\(a_t=q^{\top}(x_t-\hat\mu_k)\); then
\(q^{\top}(x_{t+1}-x_t)=a_{t+1}-a_t\) and
\((a_{t+1}-a_t)^{2}\le2a_{t+1}^{2}+2a_t^{2}\).
Summing this inequality from \(t=\bar\tau_m^{(k-1)}\) to
\(\bar\tau_m^{(k)}-2\) and using that each \(a_t\) appears at most twice gives
\(
\sum_{t=\bar\tau_m^{(k-1)}}^{\bar\tau_m^{(k)}-1}
(q^{\top}(x_t-\hat\mu_k))^{2}
\ge \frac{1}{4}\sum_{t=\bar\tau_m^{(k-1)}}^{\bar\tau_m^{(k)}-2}(a_{t+1}-a_t)^{2}.
\)
Hence on \(E_{\Gamma,k}\),
\(
\sigma_{\min}(\hat\Gamma_k)\ge \frac{\underline{C}_\Gamma}{4}>0.
\)

Now we derive the uniform bound for \(\|\hat\mu_k\|_2\). Similar to the argument leading to \eqref{bound.x2}, for every
\(k\ge K_{\mu,m}=O(\log\log(1/\delta))\), it holds for some event
\(E_{\mu,k}\) with \(\mathbb{P}(E_{\mu,k})\ge 1-\delta_k/20\) that
\(
\|\hat\mu_k\|_2 \le C_\mu .
\)

Now apply Lemma~\ref{lemma.minieigenvalue} to the matrix in
\eqref{eq:Xk_block} with
\(
\Gamma_1=\begin{bmatrix}1\\ -\hat\mu_k\end{bmatrix}\) and \(
\Gamma_2=\begin{bmatrix}0_{1\times n}\\ I_n\end{bmatrix}.
\)
On the event \(E_{0,k}:=E_{\Gamma,k}\cap E_{\mu,k}\), we obtain
\[
\sigma_{\min}(X_k)\ge \underline{C}_{\bar x}>0,
\]
where \(\underline{C}_{\bar x}\) depends only on \(\underline{C}_\Gamma\) and \(C_\mu\).

\emph{Step~2: Exploration covariance \(W_k\).}
The vectors \(\{v_t^{m}\}_{t=\bar\tau_m^{(k-1)}}^{\bar\tau_m^{(k)}-1}\) are i.i.d.\ sub‑Gaussian with
covariance matrix \(I_{d_m}\).  Stack them into an \(\tau_m^{(k)}\times d_m\) matrix
and apply Theorem~4.6.1 of~\cite{vershynin2018high}.  For any \(s>0\),
with probability at least \(1-2e^{-s^{2}}\),
\[
\left\|\frac{1}{\tau_m^{(k)}}\sum_{t=\bar\tau_m^{(k-1)}}^{\bar\tau_m^{(k)}-1}
v_t^{m}(v_t^{m})^{\top}-I_{d_m}\right\|
\le C\left(\sqrt{\frac{d_m+s^2}{\tau_m^{(k)}}}+\frac{d_m+s^2}{\tau_m^{(k)}}\right).
\]
Choose \(s=\sqrt{\log(40/\delta_k)}\) and suppose
\(\tau_m^{(k)}\ge C (d_m+\log(1/\delta_k))\), so that the right‑hand side
is at most \(1/2\). Hence, for every \(k\ge K_{W,m}=O(\log\log(1/\delta))\),
it holds for some event \(E_{W,k}\) with \(\mathbb{P}(E_{W,k})\ge1-\delta_k/20\) that
\[
\frac12 I_{d_m}\preceq W_k\preceq\frac32 I_{d_m}.
\]

\emph{Step~3: Assembly.}
Consider the leading block
\[
\widetilde M_k=
(\alpha_m^{(k-1)})^{-2}G_{m,k}X_k G_{m,k}^{\top}+J_m W_k J_m^{\top}.
\]
On the event \(E_{0,k}\cap E_{W,k}\) we have
\(\sigma_{\min}(X_k)\ge\underline{C}_{\bar x}\) and
\(\sigma_{\min}(W_k)\ge1/2\).
Using Lemma~\ref{lemma.minieigenvalue} with
\(
\Gamma_1=\begin{bmatrix}I_{1+n}\\0\end{bmatrix}\) and \(
\Gamma_2=\begin{bmatrix}-f_m^{(k-1)}\;\;F_m^{(k-1)}\;\;I_{d_m}\end{bmatrix}^{\top},
\)
a short calculation gives
\[
\sigma_{\min}(\widetilde M_k)\ge
\underline{C}_{\widetilde M}:=
\frac{(\underline{C}_{\bar x}/\bar\alpha^{2})\land(1/2)}
{8(1+\kappa_m+\kappa'_m)^{2}}>0.
\]

It remains to control the cross term and the regularisation in
\eqref{eq:V_decomp}.  On the event \(E_{z,k}\), we have \(\|V_m^{(k)}\|\le\overline{C}_V\tau_m^{(k)}\log(\tau_m^{(k)}/\delta_k)\),
\(\det(V_m^{(k)})\le(\overline{C}_V\tau_m^{(k)}\log(\tau_m^{(k)}/\delta_k))^{1+n+d_m}\).  Similar to the argument leading to~\eqref{eq:noise_av} in Lemma~\ref{lem:noise},
for every \(k\ge K_{3,m}=O(\log\log(1/\delta))\), it holds for some event
\(E_{3,k}\) with \(\mathbb{P}(E_{3,k})\ge1-\delta_k/20\) that on \(E_{3,k}\cap E_{z,k}\),
\[
\left\|(\alpha_m^{(k-1)})^{-1}\bigl(G_{m,k}C_k J_m^{\top}+J_m C_k^{\top}G_{m,k}^{\top}\bigr)
+\frac{\beta_m I}{\tau_m^{(k)}(\alpha_m^{(k-1)})^{2}}\right\|
\le C\,\frac{\;\bigl(k^{2}+\log^{2}(1/\delta)\bigr)}
           {\sqrt{\tau_m^{(k)}}\alpha_m^{(k-1)}} .
\]

Since \(\tau_m^{(k)}\asymp\lambda^{k}\) and \(\alpha_m^{(k-1)}\asymp\lambda^{-\nu_m k}\) with \(\nu_m<\frac12\),
the denominator \(\tau_m^{(k)}(\alpha_m^{(k-1)})^{2}\asymp\lambda^{(1-2\nu_m)k}\)
grows exponentially with exponent \(\frac12-\nu_m>0\), while the numerator
contains only polynomial terms in \(k\) and \(\log(1/\delta)\).  Hence
there exists \(K_{3,m}=O(\log\log(1/\delta))\) such that for all
\(k\ge K_{3,m} \ge K_{z,m}\) the right‑hand side is at most \(\underline{C}_{\widetilde M}/4\),
which completes the control of the perturbation term.

\medskip
\noindent
\textbf{Conclusion.}
Define the event
\(
E_{V,k}=E_{z,k}\cap E_{\Gamma,k}\cap E_{\mu,k}\cap E_{W,k}\cap E_{3,k}.
\)
By the union bound \(\mathbb{P}(E_{V,k})\ge1-\delta_k/4\).
On \(E_{V,k}\), for all \(k\) exceeding the maximum
\(K_{V,m}:=\max\{K_{z,m},K_{\Gamma,m},K_{\mu,m},K_{W,m},K_{3,m}\}=O(\log\log(1/\delta))\),
we have
\[
\sigma_{\min}(V_m^{(k)})\ge\underline{C}_V\,\tau_m^{(k)}(\alpha_m^{(k-1)})^{2}
\]
with \(\underline{C}_V=\underline{C}_{\widetilde M}/2\).
Together with the upper bound obtained at the beginning, this completes the
proof of Lemma~\ref{lem:V_eigen}. 

\subsubsection{Proof of Lemma~\ref{lemma.diffTheta1}}

We first give a formal description of the asynchronous tail.
For each epoch \(k\ge1\) define the synchronous phase
\[
S^{(k)} = [\underline\tau^{(k)},\bar\tau^{(k)}],
\qquad
\underline\tau^{(k)} = \max_{1\le j\le M}\bar\tau_j^{(k-1)},\;
\bar\tau^{(k)} = \min_{1\le j\le M}\bar\tau_j^{(k)}-1 .
\]
During \(S^{(k)}\) every player employs the policy obtained after her
\((k-1)\)-st update, so that the effective parameter
\(\Theta_{m,t}^{*}\) equals the constant \(\Theta_{m}^{(k),*}\) throughout
\(S^{(k)}\).  The complementary set
\[
T_m^{(k)} = [\bar\tau_m^{(k-1)},\bar\tau_m^{(k)}-1]\setminus S^{(k)}
\]
is called the asynchronous tail; by Condition~\ref{condition:tauk_exp} its
length satisfies \(|T_m^{(k)}|\le C_\tau k\) for all large \(k\).  On
\(T_m^{(k)}\) we only have the deterministic bound
\(\|\Theta_{m,t}^{*}-\Theta_{m}^{(k),*}\|\le C_\Theta\), which follows from
the compactness of the admissible strategy sets \(\mathcal A_m,\mathcal B_m\).

Now we construct the required event.  As established in the proof of Lemma~\ref{lem:V_eigen}, the state process \(\{x_t\}_{t=0}^\infty\) is uniformly sub‑Gaussian, which implies that the regression vectors \(\{z_t^{m}\}_{t=0}^\infty\) are also uniformly sub‑Gaussian.  It is obvious from its definition that \(\{\tilde w_t^{m}\}_{t=1}^\infty\) is a uniformly sub‑Gaussian sequence.  Applying the standard sub‑Gaussian maximal inequality over the interval \([\bar\tau_m^{(k-1)},\bar\tau_m^{(k)}-1]\) shows that for every \(k\ge K_{A,m}=O(\log\log(1/\delta))\), it holds for some events \(E_{A1,k},E_{A2,k}\) with \(\mathbb{P}(E_{A1,k})\ge1-\delta_k/8\) and \(\mathbb{P}(E_{A2,k})\ge1-\delta_k/8\) that
\[
\max_{\bar\tau_m^{(k-1)}\le t\le\bar\tau_m^{(k)}-1}\|z_t^{m}\|_2
\le C_z\sqrt{\log(\tau_m^{(k)}/\delta_k)},\qquad
\max_{\bar\tau_m^{(k-1)}\le t\le\bar\tau_m^{(k)}-1}\|\tilde w_t^{m}\|_2
\le C_w\sqrt{\log(\tau_m^{(k)}/\delta_k)} .
\]
Set \(E_{A,k}=E_{A1,k}\cap E_{A2,k}\); then \(\mathbb{P}(E_{A,k})\ge1-\delta_k/4\). 
On this event we bound the two quantities appearing in the lemma.

For the first term, using that the summand is non‑zero only on
\(T_m^{(k)}\) and that \(|T_m^{(k)}|\le C_\tau k\),
\[
\left\|\sum_{t=\bar\tau_m^{(k-1)}}^{\bar\tau_m^{(k)}-1}
        z_t^{m}(z_t^{m})^{\top}(\Theta_{m,t}^{*}-\Theta_{m}^{(k),*})^{\top}\right\|
\le \sum_{t\in T_m^{(k)}}\|z_t^{m}\|_2^{2}\,
                         \|\Theta_{m,t}^{*}-\Theta_{m}^{(k),*}\|
\le C\,k\,\log(\tau_m^{(k)}/\delta_k).
\]

For the second term,
\[
\left\|\sum_{t=\bar\tau_m^{(k-1)}}^{\bar\tau_m^{(k)}-1}
        (\Theta_{m,t}^{*}-\Theta_{m}^{(k),*})^{\top}z_t^{m}(\tilde w_t^{m})^{\top}\right\|
\le \sum_{t\in T_m^{(k)}}
      \|\Theta_{m,t}^{*}-\Theta_{m}^{(k),*}\|\,
      \|z_t^{m}\|_2\,\|\tilde w_t^{m}\|_2
\le C\,k\,\log(\tau_m^{(k)}/\delta_k).
\]

Since
\(k\log(\tau_m^{(k)}/\delta_k)
= k\bigl[\,O(k)+\log(1/\delta)+2\log(k+1)\bigr]
\le C\bigl(k^{2}+\log(1/\delta)\bigr)\), both quantities are bounded by
\(C(k^{2}+\log(1/\delta))\) on \(E_{A,k}\), as required.

\subsubsection{Proof of Lemma~\ref{lem:cov_noise}}

We treat the synchronous phase \(S^{(k)}\) and the asynchronous tail
\(T_m^{(k)}\) separately.  Recall that \(|T_m^{(k)}|\le C_\tau k\) and
\(|S^{(k)}|\ge \tau_m^{(k)}-C_\tau k\).

We first deal with the synchronous phase \(S^{(k)}\).
On \(S^{(k)}\) all players use fixed strategies, hence the effective noise
vectors \(\tilde w_t^{m}\) are i.i.d.\ with zero mean and covariance matrix 
\(\Sigma_m^{(k),*}\).  Let
\(y_t = (\Sigma_m^{(k),*})^{-1/2}\tilde w_t^{m}\) for \(t\in S^{(k)}\).
Then \(\{y_t\}_{t\in S^{(k)}}\) are i.i.d.\ sub‑Gaussian with
\(\mathbb{E}[y_t]=0\) and \(\mathbb{E}[y_t y_t^{\top}]=I_n\).
Applying Theorem~4.6.1 of~\cite{vershynin2018high} exactly as in Step~2
of the proof of Lemma~\ref{lem:V_eigen} shows that for every
\(k\ge K_{c1,m}=O(\log\log(1/\delta))\), it holds for some event
\(E_{c1,k}\) with \(\mathbb{P}(E_{c1,k})\ge1-\delta_k/8\) that
\[
\left\|\frac{1}{|S^{(k)}|}\sum_{t\in S^{(k)}} y_t y_t^{\top}-I_n\right\|
\le C\sqrt{\frac{n+\log(1/\delta_k)}{|S^{(k)}|}} .
\]
Multiplying by \((\Sigma_m^{(k),*})^{1/2}\) on both sides, using
\(\|\Sigma_m^{(k),*}\|\le C\) and \(|S^{(k)}|\asymp\tau_m^{(k)}\), we obtain
\[
\left\|\frac{1}{\tau_m^{(k)}}\sum_{t\in S^{(k)}}\bigl(\tilde w_t^{m}(\tilde w_t^{m})^{\top}
      -\Sigma_m^{(k),*}\bigr)\right\|
\le C\sqrt{\frac{n+\log(1/\delta_k)}{\tau_m^{(k)}}} .
\]

For the asynchronous tail \(T_m^{(k)}\), we use the sub‑Gaussian maximal inequality as
in the proof of Lemma~\ref{lemma.diffTheta1}.  This shows that for every \(k\ge K_{c2,m}=O(\log\log(1/\delta))\), it holds for some event
\(E_{c2,k}\) with \(\mathbb{P}(E_{c2,k})\ge1-\delta_k/8\) that
\[
\max_{\bar\tau_m^{(k-1)}\le t\le\bar\tau_m^{(k)}-1}\|\tilde w_t^{m}\|_2^{2}
\le C\log(\tau_m^{(k)}/\delta_k).
\]
On \(E_{c2,k}\),
\[
\left\|\frac{1}{\tau_m^{(k)}}\sum_{t\in T_m^{(k)}} \bigl(\tilde w_t^{m}(\tilde w_t^{m})^{\top}
      -\Sigma_{m}^{(k),*}\bigr)\right\|
\le \frac{|T_m^{(k)}|}{\tau_m^{(k)}}\Bigl(\max_{\bar\tau_m^{(k-1)}\le t\le\bar\tau_m^{(k)}-1}\|\tilde w_t^{m}\|_2^{2}+C\Bigr)
\le C\,\frac{k\log(\tau_m^{(k)}/\delta_k)}{\tau_m^{(k)}} .
\]

Now set \(E_{\mathrm{cov},k}=E_{c1,k}\cap E_{c2,k}\) and
\(K_{\mathrm{cov},m}=\max\{K_{c1,m},K_{c2,m}\}=O(\log\log(1/\delta))\).
Then \(\mathbb{P}(E_{\mathrm{cov},k})\ge1-\delta_k/4\) and for all
\(k\ge K_{\mathrm{cov},m}\), on \(E_{\mathrm{cov},k}\),
\[
\left\|\frac1{\tau_m^{(k)}}\!\sum_{t=\bar\tau_m^{(k-1)}}^{\bar\tau_m^{(k)}-1}
      \bigl(\tilde w_t^{m}(\tilde w_t^{m})^{\top}-\Sigma_{m}^{(k),*}\bigr)\right\|
\le C\sqrt{\frac{n+\log(1/\delta_k)}{\tau_m^{(k)}}}
   + C\,\frac{k\log(\tau_m^{(k)}/\delta_k)}{\tau_m^{(k)}} .
\]

Since \(\tau_m^{(k)}\asymp\lambda^{k}\) and \(\nu_m<\frac12\), both terms
on the right‑hand side are of smaller order than \(\eta_m^{(k)}(\delta)\) in \eqref{def.etak}.
Hence, after possibly increasing \(K_{\mathrm{cov},m}\) by an absolute
constant, the total error is bounded by \(C\eta_m^{(k)}(\delta)\) for all
\(k\ge K_{\mathrm{cov},m}\).

\section{Algorithm Implementation and Sufficient Conditions  in Section  \ref{sec:dynamic_oligopoly}}\label{appendix.oligopoly}

This section provides the implementation details of Algorithm~\ref{algorithm.indepedentlearning.Mplayer} for the dynamic Cournot competition model, followed by the theoretical analysis that establishes sufficient conditions for Assumptions~\ref{ass.def.bestresponsemap}, \ref{ass.Psi_contractive_selfmap}, and \ref{ass.uniform.Phimcontinuity}. Appendix \ref{appendix.algoimplement} derives a directly verifiable characterization of the feasible set $\mathcal{S}_m$ and a practical method for computing the greedy strategy (Algorithm~\ref{algorithm.oligopoly.greedypolicycalculation}). Appendix \ref{appendix.oligopoly.optimal} identify conditions under which the optimal strategy of \eqref{oligopoly.costfunction} is unique, independent of the initial price, and Lipschitz continuous with respect to the perceived parameters. Using these conditions together with the equivalence relationship \eqref{eqn.relationship},  Appendix \ref{appendix.oligopoly.suff_condition} establishes a set of sufficient conditions that guarantee Assumptions~\ref{ass.def.bestresponsemap}, \ref{ass.Psi_contractive_selfmap}, and \ref{ass.uniform.Phimcontinuity}. The proofs of the lemmas in Appendices~\ref{appendix.algoimplement} and \ref{appendix.oligopoly.suff_condition} are collected in Appendix~\ref{appendix.additionalproofs2}.  
In the first two subsections, we focus on a representative firm $m$ and consider its problem \eqref{oligopoly.costfunction} under a perceived parameter pair \((\Theta_m = [(1 - s')a',\; s',\; -(1 - s')b']^{\top},\; \Sigma_m=\sigma^2_m)\).

\subsection{Algorithm Implementation}\label{appendix.algoimplement}

To derive an implementable way to check feasibility and calculate the greedy strategy in Algorithm~\ref{algorithm.indepedentlearning.Mplayer}, we first derive an explicit representation of the cost function \eqref{oligopoly.costfunction} under the perceived model \eqref{oligopoly.statisticalmodel}. For a given linear strategy \((F_m, f_m, \alpha_m)\) with \((F_m, f_m) \in \mathcal{A}_m \times \mathcal{B}_m\) and \(\alpha_m \ge 0\), the price process implied by \eqref{oligopoly.statisticalmodel} becomes
\begin{equation}\label{dynamic.oligopoly.Ffm}
p_{t+1} = \bigl(s' + (1 - s')b'F_m\bigr) p_t + (1 - s')a' - (1 - s')b' f_m - (1 - s')b'\alpha_m v^m_t + \omega'_t,
\end{equation}
with initial price \(p_0 = p\). The exploration noise \(\{v^m_t\}_{t=0}^{\infty}\) and the perceived market noise \(\{\omega'_t\}_{t=0}^{\infty}\) are mutually independent i.i.d.\ sub-Gaussian sequences with mean zero and variances \(1\) and \(\sigma_m^2\), respectively. Substituting \eqref{dynamic.oligopoly.Ffm} into the cost function \eqref{oligopoly.costfunction} yields
\begin{equation}\label{oligopoly.Jm}
J^{\Theta_m,\Sigma_m}_m\big( (F_m,f_m); p,\alpha_m \big)=\sum_{t=0}^{\infty} \rho^t \Bigl( Q_m(F_m) \mathbb{E}[p_t^2] + R_m(F_m,f_m) \mathbb{E}[p_t] + O_m(F_m,f_m) \Bigr),
\end{equation}
where \(\{p_t\}\) evolves according to \eqref{dynamic.oligopoly.Ffm}, and
\begin{align*}
Q_m(F_m) &= (\tfrac12 +(1 - s')b')F_m^{2} + s' F_m,\\
R_m(F_m,f_m) &= -(1 + 2(1 - s')b')F_m f_m - s' f_m + ((1 - s')a' - c_m)F_m,\\
O_m(F_m,f_m) &= (\tfrac12 +(1 - s')b')f_m^{2} + (c_m - (1 - s')a')f_m + (\tfrac12  +(1 - s')b')\alpha_m^2.
\end{align*} 
The expressions \eqref{oligopoly.Jm} provide the foundation for a directly verifiable characterization of the feasible set \(\mathcal{S}_m\) in \eqref{def:feasible_set} and a practical greedy strategy calculation method developed below.

\subsubsection{Explicit Characterization of \(\mathcal{S}_m\) }

We first provide a more explicit characterization of the feasible parameter set  \(\mathcal{S}_m\)  defined in \eqref{def:feasible_set} 
for player $m$. 

For all $F_m\in \mathcal A_m$, 
write \(A^m_{\mathrm{cl}}(F_m) = s'+ (1 - s')b'F_m\). This quantity characterizes the asymptotic stationarity and growth rate of \(\{p_t\}\): if \(|A^m_{\mathrm{cl}}(F_m)| < 1\), the process is asymptotically stationary; otherwise, 
\(|\mathbb{E}[p_t]|\) and \(\mathbb{E}[p_t^2]\) grow exponentially at the rates \(|A^m_{\mathrm{cl}}(F_m)|^t\) and \((A^m_{\mathrm{cl}}(F_m))^{2t}\), respectively. 
The cost \eqref{oligopoly.Jm} can be   \(-\infty\) (i.e., infinite profit) if and only if either (i) \(Q_m(F_m) < 0\) and \(\sqrt{\rho}\,|A^m_{\mathrm{cl}}(F_m)| \ge 1\), or (ii) \(Q_m(F_m) = 0\) and 
\(\sum_{t=0}^{\infty} \rho^t R_m(F_m,f_m) \mathbb{E}[p_t] = -\infty\). 
To exclude these pathological cases, we define the following set \(\hat{\mathcal{S}}_m\) 
 \begin{equation}\label{eg:feasible_set_oligopoly}
\hat{\mathcal{S}}_m = \left\{ (\Theta_m,\Sigma_m) : \forall F_m \in \mathcal{A}_m,\;
\begin{aligned}
&\text{if } Q_m(F_m) < 0 \text{ then } |A^m_{cl}(F_m)| < \tfrac{1}{\sqrt{\rho}},\\
&\text{if } Q_m(F_m) = 0 \text{ then } |A^m_{cl}(F_m)| < \tfrac{1}{\rho}.
\end{aligned}
 \right\}.
\end{equation}

The conditions in \eqref{eg:feasible_set_oligopoly} rule out any possibility of unbounded negative cost. 
The following lemma, whose proof can be found in Appendix~\ref{appendix.additionalproofs2}, establishes that this auxiliary set coincides with the feasible set \(\mathcal{S}_m\) required by Algorithm~\ref{algorithm.indepedentlearning.Mplayer}.

\begin{lemma}\label{lemma.oligopoly.feasibleset}
For the dynamic Cournot competition model, the feasible set \(\mathcal{S}_m\) defined in \eqref{def:feasible_set} is identical to the set \(\hat{\mathcal{S}}_m\) given by \eqref{eg:feasible_set_oligopoly}.
\end{lemma}

Because \(Q_m(F_m)\) and \(A^m_{\mathrm{cl}}(F_m)\) are respectively quadratic and linear functions of \(F_m\), checking the conditions in \eqref{eg:feasible_set_oligopoly} is straightforward. When implementing Algorithm~\ref{algorithm.indepedentlearning.Mplayer} 
(line 8), we decide whether to accept an estimate candidate or fall back by verifying these conditions.

\subsubsection{Greedy Strategy Calculation}\label{sec:greedystrategy}

This subsection describes how to compute the greedy strategy in \eqref{eq:Ff} for  Step 9 of  Algorithm~\ref{algorithm.indepedentlearning.Mplayer}. Directly searching over the grid \((F_m, f_m) \in \mathcal{A}_m \times \mathcal{B}_m\) to solve \eqref{eq:Ff} is computationally costly. To circumvent this issue, we exploit the problem structure, specifically the cost function \eqref{oligopoly.Jm}, to develop a more efficient solution.

At the \(k\)-th update epoch, for a given accepted parameter set \((\tilde{\Theta}^{(k)}_m,\tilde{\Sigma}^{(k)}_{m}) \in \hat{\mathcal{S}}_m\), an admissible strategy \((F_m,f_m) \in \mathcal{A}_m \times \mathcal{B}_m\) yields either a finite cost or an infinite positive cost. 
To determine when the cost is finite, we use the following lemma, which follows by applying the standard LQ formulas (or, equivalently, solving the Bellman equation) to the linear perceived dynamics \eqref{dynamic.oligopoly.Ffm} with the quadratic cost structure \eqref{oligopoly.Jm}.
\begin{lemma}\label{lemma.costfunction.oligopoly}
For any given \((\Theta_m,\Sigma_m)\in \hat{\mathcal{S}}_m\), \(p>0\), and \(\alpha_m \ge 0\), the cost \eqref{oligopoly.Jm} is finite if and only if one of the following holds:
\begin{enumerate}
    \item \(\sqrt{\rho}\,|A^m_{\mathrm{cl}}(F_m)| < 1\);
    \item \(\rho\,|A^m_{\mathrm{cl}}(F_m)| < 1\) and \(Q_m(F_m) = 0\);
    \item \(Q_m(F_m) = 0\) and \(R_m(F_m,f_m) = 0\).
\end{enumerate}
In the first two cases, for a fixed \(F_m\) the cost is a quadratic function of \(f_m\).
\end{lemma}

Thus, to minimize \eqref{oligopoly.costfunction}, firm \(m\) searches over strategies satisfying the conditions of Lemma~\ref{lemma.costfunction.oligopoly} (otherwise the cost would be infinite). Because \(\mathcal{A}_m\) is one‑dimensional, we first perform a grid search over \(F_m \in \mathcal{A}_m\). For a fixed \(F_m\), when the first two conditions of the lemma hold, the optimal \(f_m\) can be found analytically: since \(\mathcal{B}_m = [-\kappa'_m, \kappa'_m]\) is a closed interval and the cost is quadratic in \(f_m\), the minimizer is either the unconstrained quadratic minimizer clipped to the interval, or one of the endpoints. In the third condition, the values of \(f_m\) that satisfy \(R_m(F_m,f_m)=0\) are directly computable because \(R_m\) is linear in \(f_m\). This procedure yields 
an efficient greedy strategy computation method, summarized in Algorithm~\ref{algorithm.oligopoly.greedypolicycalculation}. For other types of LQ games, particularly multi‑dimensional ones, we provide Algorithm~\ref{algorithm.multidim.greedypolicycalculation} for greedy strategy calculation in Appendix~\ref{appendix.multidim}.

\begin{algorithm}[ht]
\caption{Cournot Competition: Greedy Strategy Calculation for Firm $m$ at Time $\bar{\tau}^{(k)}_m$}
\label{algorithm.oligopoly.greedypolicycalculation}
\begin{algorithmic}[1]
\STATE \textbf{Require:} Accepted parameter pair $(\tilde{\Theta}^{(k)}_m,\tilde{\Sigma}^{(k)}_{m})$, exploration noise scale $\alpha_m^{(k)}$, initial price $p_{\bar{\tau}^{(k)}_m}$
\STATE Initialize candidate set $\mathcal{C} = \emptyset$
\FOR{each $F_m \in \mathcal{A}_m$} 
    \IF{$\sqrt{\rho}\,|A^m_{\text{cl}}(F_m)| < 1$ or ($\rho\,|A^m_{\text{cl}}(F_m)| < 1$ and $Q_m(F_m) = 0$)}
        \STATE Compute $f_{\min} = \arg\min\limits_{f_m \in \mathcal{B}_m} \, J_m^{\tilde{\Theta}^{(k)}_m,\tilde{\Sigma}^{(k)}_{m}}\bigl((F_m,f_m); p_{\bar{\tau}^{(k)}_m},\alpha_m^{(k)}\bigr)$
        \STATE Evaluate $J_{\min} = J_m^{\tilde{\Theta}^{(k)}_m,\tilde{\Sigma}^{(k)}_{m}}\bigl((F_m,f_{\min}); p_{\bar{\tau}^{(k)}_m},\alpha_m^{(k)}\bigr)$
        \STATE Add $(F_m, f_{\min}, J_{\min})$ to $\mathcal{C}$
    \ENDIF
    \IF{$Q_m(F_m) = 0$}
        \FOR{each $f_m \in \mathcal{B}_m$ such that $R_m(F_m,f_m) = 0$}
            \STATE Compute $J = J_m^{\tilde{\Theta}^{(k)}_m,\tilde{\Sigma}^{(k)}_{m}}\bigl((F_m,f_m); p_{\bar{\tau}^{(k)}_m},\alpha_m^{(k)}\bigr)$
            \STATE Add $(F_m, f_m, J)$ to $\mathcal{C}$
        \ENDFOR
    \ENDIF
\ENDFOR
\STATE Select the triple $(F_m^*,f_m^*,J^*)$ in $\mathcal{C}$ with the smallest cost $J^*$
\STATE \textbf{return} $(F_m^*,f_m^*)$
\end{algorithmic}
\end{algorithm}

\subsection{An explicit characterization of the optimal strategy of \eqref{oligopoly.Jm}}\label{appendix.oligopoly.optimal}

In this section, we study the setting in which the optimal strategy of the problem \eqref{oligopoly.Jm} over the admissible set \(\mathcal{A}_m \times \mathcal{B}_m\)  coincides with that of the corresponding unconstrained problem (the minimizer of \eqref{oligopoly.Jm} over \(\mathbb{R} \times \mathbb{R}\)). When this coincidence occurs, the optimal strategy over \(\mathcal{A}_m \times \mathcal{B}_m\) is unique and admits 
a closed‑form solution that is independent of the initial price \(p\) and Lipschitz continuous with respect to the parameters \((\Theta_m, \Sigma_m, \alpha_m)\). Leveraging the equivalence relationship \eqref{eqn.relationship}, these favorable properties allow us to establish Assumptions~\ref{ass.def.bestresponsemap}, \ref{ass.Psi_contractive_selfmap}, and \ref{ass.uniform.Phimcontinuity} under some sufficient conditions in Appendix~\ref{appendix.oligopoly.suff_condition}. The following lemma provides an explicit characterization of the optimal strategy for the unconstrained problem and states when it also solves the constrained problem.

\begin{lemma}\label{lemma.riccati.oligopoly}

If 
  \begin{equation}\label{eqn.oligopoly.riccati.well}
          (1 - s')b'>0, \quad \sqrt{\rho}|s'|<1, 
    \end{equation}
    then for any $p\in \mathbb{R}$, $\alpha_m\ge 0$, the following hold.

\noindent(1) The unconstrained problem $\inf\limits_{(F_m,f_m)\in \mathbb{R}\times \mathbb{R}} J^{\Theta_m,\Sigma_m}_m((F_m,f_m);p, \alpha_m)$ has a finite optimal value, and the unique minimizer is given by
    \begin{equation}\label{oligopoly.Phim}
    F_m^{\Theta_m}=\frac{-s' (2\rho P^{\Theta_m}_m (1 - s')b'+1)}{1+2(1 - s')b' + 2\rho P^{\Theta_m}_m ((1 - s')b')^2},
\end{equation}
\begin{equation}\label{oligopoly.phim}
    f_m^{\Theta_m}=\frac{ (1 - \rho s')((1 - s')a' - c_m) + 2\rho P^{\Theta_m}_m (1 - s')b' (1 - s')a' }{ \big( 1+2(1 - s')b' + 2\rho P^{\Theta_m}_m ((1 - s')b')^2 \big) (1 - \rho (s' +(1 - s')b' F_m^{\Theta_m})) },
\end{equation}
where
\begin{equation}\label{oligopoly.Pm}
    P^{\Theta_m}_m=\frac{\frac{\rho}2  (s')^2-(\frac12 +(1 - s')b')+\sqrt{\big[ (\frac12+(1 - s')b') - \frac{\rho}2  (s')^2 \big]^2-\rho ((1 - s')b')^2 (s')^2}}{2\rho ((1 - s')b')^2}.
\end{equation}
\noindent(2) If additionally $(F_m^{\Theta_m},f_m^{\Theta_m})\in \mathcal{A}_m \times \mathcal{B}_m$, then this pair is also the unique minimizer of the problem \eqref{oligopoly.Jm} over \(\mathcal{A}_m \times \mathcal{B}_m\). In particular, the optimal value is finite, 
so $(\Theta_m,\Sigma_m)\in\mathcal{S}_m$.
\end{lemma}
\begin{proof}{Proof}
We proceed in three steps.

\noindent\textbf{Step 1: Riccati equation and closed-loop stability.}
Consider the Riccati equation
\begin{equation}\label{eqn.riccati.oligopoly}
    h(P_m)=4\rho ((1 - s')b')^2 P_m^2 + 4\big[ (\frac12 +(1 - s')b') -  \frac{\rho}2 (s')^2 \big] P_m + (s')^2 = 0.
\end{equation}
Under condition~\eqref{eqn.oligopoly.riccati.well}, the discriminant of the quadratic~\eqref{eqn.riccati.oligopoly} satisfies $\Delta > 0$, and evaluating $h(P_m)$ at $P_m = 0$ and at $P_m = -\tfrac{\frac{1}{2}+(1 - s')b'}{\rho((1 - s')b')^2}$ shows $h(P_m) > 0$ at both endpoints. Hence Vieta's formulas yield two negative roots $-\tfrac{\frac{1}{2}+(1 - s')b'}{\rho((1 - s')b')^2}<P^{(2)}_m < P^{(1)}_m < 0$. Setting $X^{(i)} = 1 + 2(1 - s')b' + 2\rho P^{(i)}_m((1 - s')b')^2$ and 
substituting the relation into~\eqref{eqn.riccati.oligopoly} to obtain the 
quadratic equation satisfied by \(X^{(i)}\), Vieta's formulas give $X^{(1)}X^{(2)} = \rho(s'(1+(1 - s')b'))^2$, hence $X^{(1)} > \sqrt{\rho}|s'(1+(1 - s')b')| > X^{(2)} > 0$. Therefore $P_m$ in~\eqref{oligopoly.Pm} (i.e.\ $P_m^{(1)}$) is the unique root satisfying  the required closed-loop stability condition $\sqrt{\rho}|s' +(1 - s')b' F_m^{\Theta_m}| < 1$.

\noindent\textbf{Step 2: Identifying the minimizer.}
We distinguish two cases according to $F_m$.

First, by the characterization \eqref{oligopoly.Jm}, for any $F_m$ outside the interval
\(
\left(-\frac{s'+\frac{1}{\sqrt{\rho}}}{(1 - s')b'},\; \frac{\frac{1}{\sqrt{\rho}}-s'}{(1 - s')b'}\right),
\)
we have $Q_m(F_m)>0$, hence $J^{\Theta_m,\Sigma_m}_m((F_m,f_m);p,\alpha_m)=\infty$ for all $f_m$. Such strategies cannot be optimal.

Second, restrict attention to $F_m$ in this interval and admissible $f_m$. To identify the optimal strategy, introduce the auxiliary controlled dynamics
\begin{equation}\label{dynamic.ptm.Sigma'}
p_{t+1} = (1 - s')a' + s' p_t -(1 - s')b' \tilde{q}_t^m - (1 - s')b'\alpha_m v^m_t + \omega'_t
\end{equation}
for $\{\tilde{q}_t^m\}_{t\ge 0}\in\mathcal{U}^{\Theta_m}_{m}$, where
\[
\mathcal{U}^{\Theta_m}_{m}= \left\{ \tilde{q}_t^m, t\ge 0 \,\middle|\, \tilde{q}_t^m \text{ is } \mathcal{F}_t \text{-measurable}, 
\mathbb{E}\left[\sum_{t=0}^{\infty} \rho^{t}(p_t^2 + (\tilde{q}_t^m)^2)\right] < \infty \right\}.
\]
For any $F_m$ in this interval and any admissible $f_m$, the sequence $\{\tilde{q}_t^m\}_{t\ge 0}$ given by $\tilde{q}_t^m = -F_m p_t + f_m$ lies in $\mathcal{U}^{\Theta_m}_{m}$, and \eqref{dynamic.ptm.Sigma'} coincides with \eqref{dynamic.oligopoly.Ffm} under this $\{\tilde{q}_t^m\}_{t\ge 0}$. Standard dynamic programming applied to the auxiliary problem
\[
\inf_{(\tilde{q}_t^m)_{t \geq 0} \in\mathcal{U}^{\Theta_m}_{m}}\mathbb{E}\!\left[\sum_{t=0}^{\infty} \rho^{t} \left( \tfrac{1}{2} \tilde{q}^m_{t} - p_{t+1} + c_m \right) \tilde{q}^m_{t}\right],
\]
subject to \eqref{dynamic.ptm.Sigma'}, shows that the unique minimizer is the linear feedback strategy $\tilde{q}_t^{m,*}=-F_m^{\Theta_m}p_t+f_m^{\Theta_m}$.

Combining the two cases, $(F_m^{\Theta_m},f_m^{\Theta_m})$ is the unique minimizer of $J^{\Theta_m,\Sigma_m}_m((F_m,f_m);p,\alpha_m)$ over $\mathbb{R}\times\mathbb{R}$, completing the proof of part (1).

\noindent\textbf{Step 3: Constrained optimality.} Finally, if $(F_m^{\Theta_m},f_m^{\Theta_m})\in\mathcal{A}_m\times\mathcal{B}_m$, then it is admissible for the constrained problem. Since it is the unique minimizer over the larger unconstrained domain, it must also be the unique minimizer over the subset $\mathcal{A}_m\times\mathcal{B}_m$.  
\end{proof}

\subsection{Sufficient Conditions for Assumptions \ref{ass.def.bestresponsemap}, \ref{ass.Psi_contractive_selfmap}, and \ref{ass.uniform.Phimcontinuity}}\label{appendix.oligopoly.suff_condition}

In this section, we provide directly verifiable sufficient conditions under which Assumptions~\ref{ass.def.bestresponsemap}, \ref{ass.Psi_contractive_selfmap}, and \ref{ass.uniform.Phimcontinuity} hold for the Cournot competition model. We state two main conditions. Condition~\ref{condition.oligopoly.selfmap} ensures that each player's best response is well-defined, thereby establishing Assumption~\ref{ass.def.bestresponsemap} and guaranteeing the existence of a feedback Nash equilibrium; through a perturbation analysis, it also implies Assumption~\ref{ass.uniform.Phimcontinuity}. Building on Condition~\ref{condition.oligopoly.selfmap}, Condition~\ref{condition.oligopoly.contractivity} additionally guarantees the stability of this equilibrium, i.e., Assumption~\ref{ass.Psi_contractive_selfmap}. Together, the two conditions imply all three required assumptions. 
To facilitate the statement of the conditions, we first introduce the following constants:
\begin{align}
\underline{C}_{A,m} &:=s-(1-s)b\sum_{j\neq m} \kappa_{j},\quad 
        \overline{C}_{A,m}:=s+(1-s)b\sum_{j\neq m} \kappa_{j},  \label{def.barCAm} \\ 
        \underline{C}_G & :=\frac12 +(1 - s)b- \rho ((1 - s)b)^2 \overline{C}_P, \label{def.barCG} \\
\overline{C}_P &:=\frac{ \frac12(1-\rho \overline{C}_{A,m}^2) + (1-s)b - \sqrt{ \big( \frac12(1-\rho \overline{C}_{A,m}^2) + (1-s)b \big)^2 - \rho ((1-s)b)^2 \overline{C}_{A,m}^2 }}{2\rho ((1 - s)b)^2}, \label{def.barCP} \\
C_{f,m} &:= \frac{\big( ( 1-\rho \underline{C}_{A,m} ) \vee ( 2\rho (1-s)b \overline{C}_P +\rho \overline{C}_{A,m}-1) \big)(1-s)(a+b\sum\limits_{j\neq m} \kappa'_j)+( 1-\rho \underline{C}_{A,m} ) c_m}{2\underline{C}_G-\rho \overline{C}_{A,m} (1+(1 - s)b)}. \label{def.barCfm}
\end{align}

With these constants in place, we state the first condition.
\begin{condition}\label{condition.oligopoly.selfmap}
For all $1\le m \le M$, the following hold:
    \begin{enumerate}[1)]
        \item $\sqrt{\rho}\overline{C}_{A,m}<1$,
        \item $\dfrac{\overline{C}_{A,m}}{1 +2(1 - s)b- 2\rho ((1 - s)b)^2 \overline{C}_P} <\kappa_{m} $,
        \item $C_{f,m} <\kappa'_m $.
    \end{enumerate}
\end{condition}

Part 1) of Condition \ref{condition.oligopoly.selfmap} ensures that for each player, condition \eqref{eqn.oligopoly.riccati.well} holds under any admissible strategy profile of others, so that the unconstrained problem is well-posed and the unconstrained optimum exists. Parts 2) and 3) then guarantees this unconstrained optimal strategy falls strictly inside the admissible sets $\mathcal{A}_m\times\mathcal{B}_m$. 
Consequently, Condition~\ref{condition.oligopoly.selfmap} implies Assumption~\ref{ass.def.bestresponsemap}, as stated in Lemma~\ref{lemma.oligopoly.selfmap}; the same lemma also establishes the existence of a feedback Nash equilibrium, characterized by the system therein. Its proof is deferred to Appendix~\ref{appendix.additionalproofs2}.  For notational convenience, for any admissible strategy profile \((F_{-m},f_{-m})\) of the other players, define
\[
A^{(F_{-m})} = s+(1-s)b\sum_{j\neq m}F_j, \qquad
 A^{(f_{-m})}_0 = (1-s)a-(1-s)b\sum_{j\neq m}f_j.
\]

\begin{lemma}\label{lemma.oligopoly.selfmap}
    Under Condition~\ref{condition.oligopoly.selfmap}, the following hold:
    \begin{enumerate}[1)]
        \item Assumption~\ref{ass.def.bestresponsemap} holds with $(\psi_m(F_{-m},f_{-m}),\varphi_m(F_{-m},f_{-m}))$ given explicitly in \eqref{def.psivarphi.oligopoly} for all $1\le m \le M$ and $(F_{-m},f_{-m})\in\mathcal{A}_{-m}\times \mathcal{B}_{-m}$.
        \item There exists a feedback Nash equilibrium, characterized by the solution to the coupled system
        \begin{equation}\label{NE.oligopoly.system}
\begin{cases}
    F_m=-\dfrac{A^{(F_{-m})} (2\rho P_m (1 - s)b +1)}{2(\frac12 +(1 - s)b) + 2\rho P_m ((1 - s)b)^2},\\[6pt]
    f_m=\dfrac{ (1 - \rho A^{(F_{-m})})(A^{(f_{-m})}_0 - c_m) + 2\rho P_m (1 - s)b A^{(f_{-m})}_0 }
               { 2\big( (\frac12 +(1 - s)b) + \rho P_m ((1 - s)b)^2 \big) \bigl(1 - \rho (A^{(F_{-m})} +(1 - s)b F_m)\bigr) },\\[6pt]
    P_m=\dfrac{\frac{\rho}2  (A^{(F_{-m})})^2-(\frac12 +(1 - s)b)+\sqrt{\big[ (\frac12 +(1 - s)b) - \frac{\rho}2  (A^{(F_{-m})})^2 \big]^2-\rho ((1 - s)b)^2 (A^{(F_{-m})})^2}}
                {2\rho ((1 - s)b)^2},
\end{cases}
\end{equation}
for $1\le m \le M$.
\end{enumerate}
\end{lemma}

Via a perturbation argument, Condition~\ref{condition.oligopoly.selfmap} also establishes Assumption~\ref{ass.uniform.Phimcontinuity}, as stated in Lemma~\ref{lemma.uniform.Phimcontinuity} below; its proof is deferred to Appendix~\ref{appendix.additionalproofs2}.

\begin{lemma}\label{lemma.uniform.Phimcontinuity}
    Under Condition~\ref{condition.oligopoly.selfmap}, 
    Assumption~\ref{ass.uniform.Phimcontinuity} holds.
\end{lemma}

To find a condition that guarantees the stability of the equilibrium characterized in Lemma~\ref{lemma.oligopoly.selfmap}, we introduce the following constants:
\begin{align}
    L_{P,m} &:= \frac{\big(|A^{(F^*_{-m})}| + \overline{C}_{A,m}\big)(1 + 2\rho \overline{C}_P)}{2 + 4(1-s)b - 2\rho \big(A^{(F^*_{-m})}\big)^2 + 4\rho ((1-s)b)^2 \big(P_m^* - \overline{C}_P\big)}, \label{def.LPm} \\
    L_{F,m} &= \frac{|A^{(F^*_{-m})}| \rho (1-s)b (1+(1-s)b) L_{P,m}}
                       {2\underline{C}_G \big|\frac12 + (1-s)b + \rho P_m^* ((1-s)b)^2\big|} + \frac{1}{2\underline{C}_G}, \label{def.LFm} \\
    L_{f,m} &:= \frac{\big|1 - \rho A^{(F^*_{-m})} + 2\rho P_m^* (1-s)b\big| }{2 \big|\frac12 + (1-s)b + \rho P_m^* ((1-s)b)^2\big| \;\big|1 - \rho\big(A^{(F^*_{-m})} + (1-s)b F^*_{m}\big)\big|}. \label{def.Lfm} 
\end{align}
With these constants in place, we state the contraction condition.
\begin{condition}\label{condition.oligopoly.contractivity}
Let $(F^*,f^*,\{P^*_m\}_{1\le m \le M})$ be a solution of \eqref{NE.oligopoly.system}. For all $1\le m \le M$, assume
    \begin{enumerate}[1)]
        \item $4\big(\frac12 + (1-s)b\big) - 2\rho \big(A^{(F^*_{-m})}\big)^2 + 4\rho ((1-s)b)^2 \big(P_m^* - \overline{C}_P\big) > 0$,
        \item $(1-s)b \sum\limits_{j\neq m} L_{F,j}  < 1$, 
        \item $(1-s)b \sum\limits_{j\neq m} L_{f,j}  < 1$,
    \end{enumerate}
    where $\overline{C}_P$, $L_{F,m}$, and $L_{f,m}$ are given by \eqref{def.barCP}, \eqref{def.LFm}, and \eqref{def.Lfm}, respectively.
\end{condition}

Condition \ref{condition.oligopoly.contractivity} is expressed solely in terms of the model parameters and the equilibrium quantities themselves. Part 1) ensures a non‑degenerate denominator in the Lipschitz estimates, while parts 2) and 3) impose that the cumulative influence of other players through their feedback gains and offsets remains contracting. Lemma \ref{lemma.contraction} below formalizes this implication, giving the stability coefficient \(\zeta\) explicitly; we omit the explicit expression for \(\mu\), which does not affect the convergence rate, for brevity. Its proof is deferred to Appendix~\ref{appendix.additionalproofs2}.

\begin{lemma}\label{lemma.contraction}
    Under Conditions \ref{condition.oligopoly.selfmap} and \ref{condition.oligopoly.contractivity}, there exists a constant $\mu>0$ (depending on the model parameters) such that Assumption \ref{ass.Psi_contractive_selfmap} holds with $\zeta$ given by
    \begin{align}\label{def.oligopoly.zeta}
\zeta = \max\left\{
   \frac12+\frac12\max_{1\le j\le M} \Big\{(1-s)b \sum_{m\neq j} L_{F,m}\Big\},\;
   \max_{1\le j\le M} \Big\{(1-s)b \sum_{m\neq j} L_{f,m}\Big\}
   \right\}. 
\end{align}
\end{lemma}

\subsection{Proofs of Auxiliary Lemmas in Appendix \ref{appendix.oligopoly}}\label{appendix.additionalproofs2}

\subsubsection{Proof of Lemma \ref{lemma.oligopoly.feasibleset}}\label{sec.feasibleset.oligopoly}

\begin{proof}{Proof}
We prove the two inclusions separately.

\noindent\textbf{Part 1: \(\hat{\mathcal{S}}_m \subseteq \mathcal{S}_m\).}
Take any \((\Theta_m,\Sigma_m) \in \hat{\mathcal{S}}_m\). For all $p$ and $\alpha_m \ge 0$, the cost at \((0,0)\) is finite, so the infimum is \(< \infty\). For any admissible \((F_m,f_m)\), the moments satisfy \(\mathbb{E}[p_t] \asymp (A^m_{\mathrm{cl}}(F_m))^t\) and \(\mathbb{E}[p_t^2] \asymp (A^m_{\mathrm{cl}}(F_m))^{2t}\). Consequently, in \eqref{oligopoly.Jm}, the term with \(Q_m(F_m)\) converges if \(Q_m(F_m)=0\) or \(\rho (A^m_{\mathrm{cl}})^2<1\); the term with \(R_m(F_m,f_m)\) converges if \(R_m(F_m,f_m)=0\) or \(\rho|A^m_{\mathrm{cl}}|<1\).

Now we examine each sign of \(Q_m(F_m)\). If \(Q_m(F_m)\le 0\), the definition of \(\hat{\mathcal{S}}_m\) ensures that the cost is finite. If \(Q_m(F_m)>0\), then when \(\sqrt{\rho}|A^m_{\mathrm{cl}}|<1\) the cost is finite; otherwise it diverges to \(+\infty\) because the positive \(Q_m\)-term dominates. Thus the cost is never \(-\infty\). Lower semicontinuity on the compact set \(\mathcal{A}_m\times\mathcal{B}_m\) implies the infimum is attained and finite. Therefore \((\Theta_m,\Sigma_m)\in\mathcal{S}_m\).

\noindent\textbf{Part 2: \(\mathcal{S}_m \subseteq \hat{\mathcal{S}}_m\).}
We prove the contrapositive. Suppose \((\Theta_m,\Sigma_m)\notin\hat{\mathcal{S}}_m\). Then there exists \(F_m\in\mathcal{A}_m\) violating the condition.
\begin{enumerate}[1)]
    \item \(Q_m(F_m)<0\) but \(\sqrt{\rho}|A^m_{\mathrm{cl}}|\ge 1\): then \(\sum_t \rho^t Q_m(F_m)\mathbb{E}[p_t^2]\) diverges to \(-\infty\), so the infimum is \(-\infty\).
    \item \(Q_m(F_m)=0\) but \(\rho|A^m_{\mathrm{cl}}|\ge 1\): choose an admissible \(f_m\) with \(R_m(F_m,f_m)\neq0\) (possible because \(R_m\) is affine in \(f_m\)); then the linear term diverges to \(-\infty\) for a suitable initial price \(p\), making the infimum \(-\infty\).
\end{enumerate}
In either case \((\Theta_m,\Sigma_m)\notin\mathcal{S}_m\). Hence \(\mathcal{S}_m\subseteq\hat{\mathcal{S}}_m\).

Both inclusions yield \(\mathcal{S}_m=\hat{\mathcal{S}}_m\).  
\end{proof}

\subsubsection{Proof of Lemma \ref{lemma.oligopoly.selfmap}}

\begin{proof}{Proof}
We first prove part 1). For each player, under any admissible strategy profile of others, condition \eqref{eqn.oligopoly.riccati.well} holds, so that Lemma \ref{lemma.riccati.oligopoly} applies, and the unconstrained optimum exists. We then track how bounds propagate along the chain 
\eqref{oligopoly.Pm}$\to$\eqref{oligopoly.Phim}$\to$\eqref{oligopoly.phim}, showing at each step that the output remains strictly inside the admissible set. Part 2) then follows from Brouwer's fixed-point theorem, using the continuity of the best-response map established in part 1).

We take player $m$ as a representative player. 
For any admissible strategy profile 
$(F_{-m},f_{-m})$ of the other players, let $\Theta^{(F_{-m},f_{-m})}_m = [A^{(f_{-m})}_0,\, A^{(F_{-m})},\, -(1-s)b]^\top$ 
as in~\eqref{def.Thetam.Ff}. 
By replacing \((F_{-m}^{(k-1)}, f_{-m}^{(k-1)})\) in \eqref{eqn.relationship} with an arbitrary \((F_{-m},f_{-m})\in\mathcal{A}_{-m}\times\mathcal{B}_{-m}\) (and \(\Theta_m^{(k),*}\) with the corresponding \(\Theta_m^{(F_{-m},f_{-m})}\)), the minimizer of the complete-information problem \(J_m((F_m,f_m); (F_{-m},f_{-m}), p)\) coincides with that of the control problem \eqref{oligopoly.Jm} with the perceived parameter pair \((\Theta_m^{(F_{-m},f_{-m})},\Sigma_{\omega})\). 
Thus, it suffices to analyze the latter, which is the focus of the rest of this subsection. By definition,    
$\underline{C}_{A,m} \le A^{(F_{-m})} \le \overline{C}_{A,m}$, 
so Part 1) of Condition \ref{condition.oligopoly.selfmap} gives 
\begin{equation}\label{condition.riccati.oligopoly}
    \sqrt{\rho}|A^{(F_{-m})}|<1,\quad (1-s)b>0,
\end{equation}
which means condition \eqref{eqn.oligopoly.riccati.well} holds. Lemma \ref{lemma.riccati.oligopoly} therefore applies and 
$(F_m^{\Theta^{(F_{-m},f_{-m})}_m},\, f_m^{\Theta^{(F_{-m},f_{-m})}_m})$ 
attains the minimum over $\mathbb{R}\times\mathbb{R}$. 

\noindent\textbf{Step 1: Bound on $P_m$.} 
Since $|A^{(F_{-m})}|\le \overline{C}_{A,m}$, 
one verifies by rationalizing the numerator of~\eqref{oligopoly.Pm} that 
the expression is monotone increasing in $|A^{(F_{-m})}|$. 
Hence $P_m^{\Theta^{(F_{-m},f_{-m})}_m}$ is negative and bounded below by $\overline{C}_P$ in \eqref{def.barCP}. With \(\underline{C}_G\) as defined in \eqref{def.barCG}, we obtain
\begin{equation}\label{oligopoly.estimatePG}
     -\overline{C}_P \le P_m^{\Theta^{(F_{-m},f_{-m})}_m} < 0, \quad\frac12 +(1 - s)b+ \rho P_m^{\Theta^{ (F_{-m},f_{-m})}_m} ((1 - s)b)^2\ge \underline{C}_G>0.
\end{equation}

\noindent\textbf{Step 2: Bound on $F_m$.} 
Since $0>2\rho P_m^{\Theta^{ (F_{-m},f_{-m})}_m}(1-s)b\ge -2\rho(1-s)b\overline{C}_P>-1$, from \eqref{oligopoly.Phim} we obtain the estimate $\bigl|F_m^{\Theta^{ (F_{-m},f_{-m})}_m}\bigr| < \frac{\overline{C}_{A,m}}{2\underline{C}_G}$. Part 2) of Condition \ref{condition.oligopoly.selfmap} then yields 
\begin{equation}\label{oligopoly.estimate.F}
   - \kappa_{m}< F^{\Theta^{ (F_{-m},f_{-m})}_m}_m <\kappa_{m}.
\end{equation}

\noindent\textbf{Step 3: Bound on $f_m$.} 
Using~\eqref{oligopoly.Phim}, we have
\begin{equation*}
A^{(F_{-m})} +(1 - s)b F_m^{\Theta^{ (F_{-m},f_{-m})}_m}
   = \frac{A^{(F_{-m})} (1+(1 - s)b)}{2\big( (\frac12 +(1 - s)b) + \rho P_m^{\Theta^{ (F_{-m},f_{-m})}_m} ((1 - s)b)^2 \big)}.
\end{equation*}
Together with $|A^{(F_{-m})}|\le \overline{C}_{A,m}$ and~\eqref{oligopoly.estimatePG}, 
the denominator of~\eqref{oligopoly.phim} satisfies
{\small
\begin{equation}\label{oligopoly.estimatef1}
   2\big( (\frac12 +(1 - s)b) + \rho P^{\Theta^{ (F_{-m},f_{-m})}_m}_m ((1 - s)b)^2 \big) (1 - \rho (A^{(F_{-m})} +(1 - s)b F^{\Theta^{ (F_{-m},f_{-m})}_m}_m)) \ge 2\underline{C}_G-\rho \overline{C}_{A,m} (1+(1 - s)b),
\end{equation}
}
where the positivity of the lower bound follows directly from the definitions of $\overline{C}_P$ and $\underline{C}_G$. Combining this with the bounds on $A^{(f_{-m})}_0$, $|A^{(f_{-m})}_0|\le (1-s)(a+b\sum\limits_{j\neq m}\kappa'_j)$, and \eqref{oligopoly.estimatePG}, we obtain from \eqref{oligopoly.phim} the estimate $|f_m^{\Theta^{(F_{-m},f_{-m})}_m}|\le C_{f,m}$, where $C_{f,m}$ is given by \eqref{def.barCfm}. Part 3) of Condition \ref{condition.oligopoly.selfmap} then guarantees 
\begin{equation}\label{oligopoly.estimate.f}
    -\kappa'_m<f^{\Theta^{ (F_{-m},f_{-m})}}_m<\kappa'_m.
\end{equation}

By part (2) of Lemma~\ref{lemma.riccati.oligopoly}, 
the pair \((F_m^{\Theta^{(F_{-m},f_{-m})}_m}, f_m^{\Theta^{(F_{-m},f_{-m})}_m})\) is the unique minimizer, independent of the initial price $p$, and continuous in \((F_{-m},f_{-m})\), so 
Assumption~\ref{ass.def.bestresponsemap} holds with
\begin{equation}\label{def.psivarphi.oligopoly}
    \psi_m(F_{-m},f_{-m}) = F_m^{\Theta^{(F_{-m},f_{-m})}_m}, \quad
    \varphi_m(F_{-m},f_{-m}) = f_m^{\Theta^{(F_{-m},f_{-m})}_m}. 
\end{equation}
The continuity of \(\Psi\) follows from \eqref{def.psivarphi.oligopoly} and the continuity of the closed-form expressions in the parameters \((F_{-m},f_{-m})\). Since \(\mathcal{A}\times\mathcal{B}\) is compact, Brouwer's fixed-point theorem guarantees a fixed point \((F^*,f^*)\) of \(\Psi\), which by definition satisfies the coupled system \eqref{NE.oligopoly.system}. Lemma~\ref{lemma:characterize_ne} then implies that \((F^*,f^*)\) is a feedback Nash equilibrium. This proves part 2). The proof is complete.  
\end{proof}

\subsubsection{Proof of Lemma \ref{lemma.uniform.Phimcontinuity}}

\begin{proof}{Proof}
Fix any $(F_{-m},f_{-m})\in\mathcal{A}_{-m}\times\mathcal{B}_{-m}$ 
and let $\Theta^{(F_{-m},f_{-m})}_m$ be as defined in~\eqref{def.Thetam.Ff}.

\noindent\textbf{Part 1.} 
Lemma~\ref{lemma.oligopoly.selfmap} shows that 
$\Theta^{(F_{-m},f_{-m})}_m$ satisfies the conditions of 
Lemma~\ref{lemma.riccati.oligopoly} with strict inequalities 
(see~\eqref{condition.riccati.oligopoly}, \eqref{oligopoly.estimate.F}, 
\eqref{oligopoly.estimate.f}). 
Since these inequalities involve only the components of $\Theta^{(F_{-m},f_{-m})}_m$ 
and are strict, and since the bounds used in the estimates are uniform over the 
compact set $\mathcal{A}_{-m}\times\mathcal{B}_{-m}$, there exists 
$\delta_m>0$ independent of $(F_{-m},f_{-m})$ such that for any 
$\tilde{\Theta}_m$ with $\|\tilde{\Theta}_m-\Theta^{(F_{-m},f_{-m})}_m\|\le\delta_m$, 
the same strict inequalities hold. 
Consequently, Lemma~\ref{lemma.riccati.oligopoly}(2) applies and yields 
$(\tilde{\Theta}_m,\tilde{\Sigma}_m)\in\mathcal{S}_m$ for any 
$\tilde{\Sigma}_m$ with $\|\tilde{\Sigma}_m-\Sigma_{\omega}\|\le\delta_m$.
This establishes part~1) of Assumption~\ref{ass.uniform.Phimcontinuity}.

\noindent\textbf{Part 2.} 
The explicit formulas~\eqref{oligopoly.Phim}-\eqref{oligopoly.phim} are smooth 
functions of the components of $\Theta^{(F_{-m},f_{-m})}_m$ over the compact 
parameter range induced by $\mathcal{A}_{-m}\times\mathcal{B}_{-m}$. 
The uniform lower bounds~\eqref{oligopoly.estimatePG} and~\eqref{oligopoly.estimatef1} 
ensure that the denominators in these formulas stay bounded away from zero 
uniformly over $(F_{-m},f_{-m})$. 
By continuity and compactness, there exists a constant $L_m<\infty$ such that 
for all $(F_{-m},f_{-m})\in\mathcal{A}_{-m}\times\mathcal{B}_{-m}$ and all 
$(\tilde{\Theta}_m,\tilde{\Sigma}_m,\alpha_m)$ with 
\(
\max\{\|\tilde{\Theta}_m-\Theta^{(F_{-m},f_{-m})}_m\|,
\|\tilde{\Sigma}_m-\Sigma_{\omega}\|,\alpha_m\}\le\delta_m,
\)
the Lipschitz bounds required by part~2) of 
Assumption~\ref{ass.uniform.Phimcontinuity} hold. 
The independence of the minimizer from the initial state $x$ follows from 
Lemma~\ref{lemma.riccati.oligopoly}(2).  
\end{proof}

\subsubsection{Proof of Lemma \ref{lemma.contraction}}

\begin{proof}{Proof}
By Lemma \ref{lemma.oligopoly.selfmap}, $(F^*,f^*)$ is a fixed point of $\Psi$. We consider any $(\hat{F},\hat{f})\in\mathcal{A}\times\mathcal{B}$ and let 
$(F,f)=\Psi((\hat{F},\hat{f}))$. Under Condition \ref{condition.oligopoly.selfmap}, Lemma \ref{lemma.oligopoly.selfmap} provides the explicit formulas \eqref{def.psivarphi.oligopoly} for $(F_m,f_m)$.

\noindent\textbf{Step 1: Lipschitz estimate for $P_m$.}
Applying a differencing argument to formula \eqref{oligopoly.Pm} and using the 
uniform bound $|P_m| \le \overline{C}_P$ from~\eqref{oligopoly.estimatePG}, we obtain $|P_m - P_m^*| \le L_{P,m}|A^{(\hat{F}_{-m})}-A^{(F^*_{-m})}|$, where $L_{P,m}$ is defined in \eqref{def.LPm} and its denominator is positive by Part 1) of 
Condition \ref{condition.oligopoly.contractivity}.

\noindent\textbf{Step 2: Lipschitz estimate for $F_m$.}
Applying a differencing argument to formula~\eqref{oligopoly.Phim} and substituting 
the bound from Step~1 and the uniform lower bound on the denominator 
from~\eqref{oligopoly.estimatePG}, we obtain $|F_m-F_m^*|\le L_{F,m}|A^{(\hat{F}_{-m})}-A^{(F^*_{-m})}|$, where $L_{F,m}$ is given by \eqref{def.LFm}. Since $|A^{(\hat{F}_{-m})}-A^{(F^*_{-m})}| = (1-s)b\sum\limits_{j\neq m}|\hat{F}_j-F_j^*|$, summing over $m$ and rearranging yields
\begin{align*}
\sum_{m=1}^M |F_m-F_m^*| \le \max_{1\le j\le M} \Big\{(1-s)b \sum_{m\neq j} L_{F,m}\Big\} \sum_{j=1}^M |\hat{F}_j-F_j^*|.
\end{align*}

\noindent\textbf{Step 3: Lipschitz estimate for $f_m$.}
Applying a differencing argument to formula~\eqref{oligopoly.phim} and substituting 
the bounds from Steps~1-2 together with~\eqref{oligopoly.estimatef1},  
we obtain $|f_m-f_m^*|\le L_{f,m}|A^{(\hat{f}_{-m})}_0-A^{(f^*_{-m})}_0| 
+ L_{fF,m}|A^{(\hat{F}_{-m})}-A^{(F^*_{-m})}|$, where $L_{f,m}$ is given by \eqref{def.Lfm} and $L_{fF,m}<\infty$ 
depends only on the model parameters and the bounds already established. Summing over $m$ and rearranging analogously to Step~2 then gives
\begin{align*}
\sum_{m=1}^M |f_m-f_m^*| \le \max_{1\le j\le M} \Big\{(1-s)b \sum_{m\neq j} L_{fF,m}\Big\} \sum_{j=1}^M |\hat{F}_j-F_j^*|  + \max_{1\le j\le M} \Big\{(1-s)b \sum_{m\neq j} L_{f,m}\Big\} \sum_{j=1}^M |\hat{f}_j-f_j^*|.
\end{align*}

\noindent\textbf{Step 4.}
Let $\zeta$ be defined as in \eqref{def.oligopoly.zeta}, and choose $\mu   = \frac{2 \max\limits_{1\le j\le M} \big\{(1-s)b \sum\limits_{m\neq j} L_{fF,m}\big\}}
            {1-\max\limits_{1\le j\le M} \big\{(1-s)b \sum\limits_{m\neq j} L_{F,m}\big\}}$.  
Substituting the estimates from Steps~2-3 into the definition of 
\(\|\cdot\|_{\mu,1,2}\) in Assumption~\ref{ass.Psi_contractive_selfmap},  
a direct calculation yields
\begin{equation*}
\bigl\|\Psi((\hat{F},\hat{f})-\Psi((F^*,f^*))\bigr\|_{\mu,1,2}
\le \zeta\,\bigl\|(\hat{F},\hat{f})-(F^*,f^*)\bigr\|_{\mu,1,2}.
\end{equation*}
Parts 2) and 3) of Condition \ref{condition.oligopoly.contractivity} ensure that $\zeta<1$, so Assumption \ref{ass.Psi_contractive_selfmap} is satisfied.  
\end{proof}

\section{A Numerical Example with Multiple Equilibria}\label{appendix.multidim}

In this section, we consider a    linear-quadratic games with no cross term (\(H_m=0\)) and \(Q_m\succ0\), which allow for multi-dimensional state and action spaces. This setting is standard in the LQ game literature \citep{mazumdar2020policy,hambly2023policy}. We first describe the implementation of  Algorithm~\ref{algorithm.indepedentlearning.Mplayer} in this setting, and then present a numerical example with two distinct Nash equilibria, for which  the learning dynamics may converge to different equilibria or oscillate across runs.

\subsection{Algorithm Implementation}

We first analyze the unconstrained control problem and characterize its solution. Using this result, we then provide a verifiable characterization of the feasible set $\mathcal{S}_m$ and describe a practical method for computing the greedy strategy. 
Throughout this subsection, player $m$ serves as the representative agent, and we denote its parameter estimate by $(\Theta_m, \Sigma_m)$. For convenience, we partition $\Theta_m$ as $\Theta_m = [A'_0,\; A',\; B'_m]^\top$, where $A'_0 \in \mathbb{R}^{n \times 1}$ is the coefficient vector for the constant term, $A' \in \mathbb{R}^{n \times n}$ for the state $x_t$, and $B'_m \in \mathbb{R}^{n \times d_m}$ for player $m$'s control $u_t^m$ in the perceived dynamics \eqref{dynamic:statistical}.

\subsubsection{The Unconstrained Optimal Strategy}

We begin by studying the LQ control problem without the compact constraints $\mathcal{A}_m\times\mathcal{B}_m$. The following lemma can be derived using standard results from linear quadratic control theory (see, e.g., \cite{hager1976convergence}). 

\begin{lemma}\label{lemma:unconstrained_multidim}
Assume that the pair $(\sqrt{\rho}A', \sqrt{\rho}B'_m)$ is stabilizable and that $(\sqrt{\rho}A', Q_m)$ is detectable (e.g., $Q_m\succ0$ guarantees detectability). Then the following hold.

1) For every initial state $x\in\mathbb{R}^n$ and exploration scale $\alpha_m\ge0$, the infimum
\begin{equation*}
 \inf_{(F_m,f_m)\in \mathbb{R}^{d_m\times n}\times\mathbb{R}^{d_m}} J_m^{\Theta_m,\Sigma_m}\big((F_m,f_m);x,\alpha_m\big)
\end{equation*}
is finite and uniquely attained at the pair $(F_m^{\Theta_m}, f_m^{\Theta_m})$ given by
{\small
\begin{align}
F_m^{\Theta_m} &= \rho \big(R_m + \rho B_m'^\top P_m^{\Theta_m} B_m'\big)^{-1} B_m'^\top P_m^{\Theta_m} A', \label{eq:unconstrained_F}\\[4pt]
f_m^{\Theta_m} &= \big(R_m + \rho B_m'^\top P_m^{\Theta_m} B_m'\big)^{-1}\Big( \rho B_m'^\top \big(I - \rho (A' - B_m' F_m^{\Theta_m})^\top\big)^{-1} \big(\tfrac12 ((F_m^{\Theta_m})^\top r_m-q_m)-P_m^{\Theta_m} A_0' \big)-  \tfrac{r_m}2\Big), \label{eq:unconstrained_f}
\end{align}
}
where $P_m^{\Theta_m}$ is the unique positive semidefinite solution of the 
following equation
\begin{equation*}\label{eq:riccati_multidim}
P_m^{\Theta_m} = Q_m + \rho A'^\top P_m^{\Theta_m} A' - \rho^2 A'^\top P_m^{\Theta_m} B_m' \big(R_m + \rho B_m'^\top P_m^{\Theta_m} B_m'\big)^{-1} B_m'^\top P_m^{\Theta_m} A'. 
\end{equation*}
    
2) If in addition $(F_m^{\Theta_m}, f_m^{\Theta_m})\in\mathcal{A}_m\times\mathcal{B}_m$, 
    then this pair is the unique minimizer of the constrained problem 
    \eqref{discount_cost_m} for every $x$ and $\alpha_m$.

\end{lemma}

\subsubsection{Characterization of the Feasible Set $\mathcal{S}_m$}

For LQ games with \(H_m=0\) and \(Q_m\succ0\), the feasible set \(\mathcal{S}_m\) defined in \eqref{def:feasible_set} admits a spectral characterization stated in Lemma~\ref{lemma.multidim.feasibleset}. We omit its proof, which follows by applying   \cite[Theorem~5.6.12]{horn2012matrix} to the dynamics under linear feedback, together with the observation that the positive definiteness of \(Q_m\) rules out unbounded negative cost.

\begin{lemma}\label{lemma.multidim.feasibleset}
For LQ games with $H_m=0$ and $Q_m\succ0$, the feasible set $\mathcal{S}_m$ defined in \eqref{def:feasible_set} coincides with \begin{equation}\label{feasible_set_multidim}
    \tilde{\mathcal{S}}_m = \left\{(\Theta_m,\Sigma_m) : \exists F_m\in \mathcal{A}_m \text{ such that } r(\sqrt{\rho}(A'-B'_m F_m)) < 1 \right\}.
\end{equation}
\end{lemma}

Note that $\tilde{\mathcal{S}}_m$ imposes no restriction on $\Sigma_m$ because the well‑posedness of problem \eqref{discount_cost_m} is determined solely by the drift parameters. When implementing Algorithm \ref{algorithm.indepedentlearning.Mplayer}, we decide whether to accept a given $(\Theta_m,\Sigma_m)$ or fall back by checking whether it lies in $\tilde{\mathcal{S}}_m$; specifically, this requires determining if there exists an $F_m\in\mathcal{A}_m$ such that $r(\sqrt{\rho}(A'-B'_m F_m)) < 1$. For low-dimensional problems, a simple grid search over $\mathcal{A}_m$ can be used. In general, we treat the problem as a feasibility optimization: minimize $r(\sqrt{\rho}(A'-B'_m F_m))$ subject to $F_m\in\mathcal{A}_m$; if the minimal value is less than $1$, then $(\Theta_m,\Sigma_m)$ is feasible. This is a non‑convex constrained problem. A Sequential Quadratic Programming (SQP) solver 
(or its variants adapted to non‑smooth objectives) with multiple random 
initializations~\citep{nocedal2006numerical,boggs1995sequential} can be employed. This numerical verification is practical for moderate dimensions. In the next subsection, we assume that the accepted parameter set already lies in $\tilde{\mathcal{S}}_m$ and focus on computing the greedy strategy.

\subsubsection{Greedy Strategy Calculation}\label{Appendix.greedy}

We now provide a practical method for computing the greedy strategy in Algorithm~\ref{algorithm.indepedentlearning.Mplayer}. For any accepted parameter set $(\tilde{\Theta}^{(k)}_m,\tilde{\Sigma}^{(k)}_{m})\in\tilde{\mathcal{S}}_m$, the unconstrained optimum is well defined (Lemma~\ref{lemma:unconstrained_multidim}). If this unconstrained optimum lies inside the admissible set $\mathcal{A}_m\times\mathcal{B}_m$, it is optimal for the constrained problem \eqref{discount_cost_m}. Otherwise, we must solve a constrained optimization problem. Because the problem is generally non‑convex in $(F_m,f_m)$ \citep{fazel2018global}, we adopt the SQP approach described  
above, combined with a multi‑start heuristic to mitigate the risk of converging to a poor local minimum. The cost function for the SQP solver is given by the standard LQ representation. The following algorithm summarizes the procedure.

\begin{algorithm}[ht]
\caption{LQ Game without Cross Terms: Greedy Strategy Calculation for Player $m$}
\label{algorithm.multidim.greedypolicycalculation}
\begin{algorithmic}[1]
\STATE \textbf{Require:} Accepted parameter set $(\tilde{\Theta}^{(k)}_m,\tilde{\Sigma}^{(k)}_{m})$, exploration noise scale $\alpha_m^{(k)}$, initial state $x_{\bar{\tau}^{(k)}_m}$
\STATE Compute the unconstrained optimal strategy $(F_m^{\tilde{\Theta}^{(k)}_m},f_m^{\tilde{\Theta}^{(k)}_m})$ via (\ref{eq:unconstrained_F})-(\ref{eq:unconstrained_f}).
\IF{$(F_m^{\tilde{\Theta}^{(k)}_m},f_m^{\tilde{\Theta}^{(k)}_m}) \in \mathcal{A}_m \times \mathcal{B}_m$}
    \STATE \textbf{return} $(F_m^{\tilde{\Theta}^{(k)}_m},f_m^{\tilde{\Theta}^{(k)}_m})$.
\ELSE
    \STATE Generate $N_{\text{start}}$ random initial points uniformly in $\mathcal{A}_m\times\mathcal{B}_m$.
    \FOR{$i = 1$ to $N_{\text{start}}$}
        \STATE Run an SQP solver from the $i$-th initial point for  a local minimizer $(F_m^{(i)},f_m^{(i)})$ and its 
        cost $J^{(i)}$.
    \ENDFOR
    \STATE Return the pair with the smallest cost: $(F_m^*,f_m^*) = \arg\min_{i} J^{(i)}$.
\ENDIF
\end{algorithmic}
\end{algorithm}

\subsection{Numerical Example with Multiple Equilibria}\label{appendix.multiNE}

This section presents a numerical experiment illustrating the behavior of Algorithm~\ref{algorithm.indepedentlearning.Mplayer} in a two‑player LQ game with at least two  distinct complete-information Nash equilibria. Gaussian noise is used for both the state and exploration noises throughout. The system parameters are \(A=1.0700\), \(B_1=0.10254\), \(B_2=0.085934\), \(A_0=0\), \(\Sigma_{\omega}=0.1^2\), \(Q_1=0.11112\), \(R_1=0.30872\), \(Q_2=0.25806\), \(R_2=0.4949\), and \(H_m=q_m=r_m=0\) for \(m=1,2\). The discount factor is \(\rho = 0.9999\), and the admissible bounds are \(\kappa_m = 2.0\) and \(\kappa'_m = 1.0\) for \(m=1,2\). The algorithm parameters are \(\tau=300\), \(\lambda=1.1\), \(\underline{\tau}=2000\), \(\bar{\tau}=200\), \(\nu=0.25\), and \(\beta_m=0.01\) for each player. At each update, the asynchronous shift for each player is sampled independently from a uniform distribution on \([0, \bar{\tau}]\). The total number of time steps is \(T\approx 2.8\times10^8\), ensuring that every player completes 120 updates per run. 
We find two feedback Nash equilibria:
\begin{itemize}
    \item \textbf{Equilibrium 1:} \(F_1=0.8908,\; f_1=0.0000,\; F_2=0.4980,\; f_2=0.0000\). The spectral radius of \(D\Psi\) at this equilibrium is \(0.9709\).
    \item \textbf{Equilibrium 2:} \(F_1=0.3718,\; f_1=0.0000,\; F_2=1.1423,\; f_2=0.0000\). The spectral radius of \(D\Psi\) at this equilibrium is \(0.9203\).
\end{itemize}

Both equilibria are locally stable, but the existence of multiple fixed points violates Assumption~\ref{ass.Psi_contractive_selfmap}. Additionally, Assumption~\ref{ass.stationary_x} is violated since  \(|A|>1\).

We run 100 independent simulations of Algorithm~\ref{algorithm.indepedentlearning.Mplayer}, where each simulation independently selects an initial guess uniformly at random from the admissible strategy sets \(\mathcal{A}_m \times \mathcal{B}_m\). Outcomes are classified into four categories following the same criteria as in Section~\ref{numerical.violateassumption}: explode (the state grows without bound, or the algorithm diverges), oscillatory (persistent fluctuations without settling), converge to Equilibrium 1, and converge to Equilibrium 2. Table~\ref{tab:multiNE_outcomes} summarizes the results.

\begin{table}[!htbp]
\caption{Distribution of experimental outcomes for the multiple‑equilibria example}
\centering
\begin{tabular}{cccc}
\hline
 Explode (\%) & Oscillatory (\%) & Converge to Equilibrium 1 (\%) & Converge to Equilibrium 2 (\%) \\
\hline
 35.00 & 36.00 & 5.00 & 24.00 \\
\hline
\end{tabular}
\label{tab:multiNE_outcomes}
\end{table}

Because Assumption~\ref{ass.stationary_x} is violated, some simulations explode, causing the algorithm to fail. In the remaining runs, we observe three types of behavior: convergence to Equilibrium 1, convergence to Equilibrium 2, or persistent best‑response cycles. This phenomenon is consistent with the synchronous Riccati iteration results reported in \cite{nortmann2024nash} (see Figure 3(c) there). Unlike the deterministic best‑response iteration studied there, Algorithm~\ref{algorithm.indepedentlearning.Mplayer} is subject to estimation errors and exploration noises; hence there is no sharp deterministic partition of the region of initial guesses that determines convergence to a specific equilibrium. However, because both equilibria are locally stable, each possesses a neighborhood such that once the algorithm stays within that neighborhood, it tends to converge to the corresponding equilibrium. This explains why the behavior in this example differs from that of Section~\ref{numerical.violateassumption} (Group 2), where all simulations exhibited oscillations and none converged, whereas here some simulations do converge to equilibria.

\end{document}